%% file: Cdimone.tex
\documentclass[reqno,11pt,letterpaper]{amsart}
\usepackage[utf8]{inputenc}
\usepackage{abbrmath_insung}
\usepackage{todonotes}
\usepackage{tikz-cd,tikz}
\usepackage{amsmath,amssymb}
\NewCommandCopy{\origboxtimes}{\boxtimes}
\usepackage{mathabx}
\usepackage{graphicx,caption,subcaption}
\usepackage{enumitem}
\usepackage{adjustbox}
\usepackage{minibox}

\usepackage{textcomp}

\usepackage[colorlinks=true, allcolors=blue]{hyperref}

\def\V{\mathrm{Vertex}}
\def\Edge{\mathrm{Edge}}

\def\R{\mathcal{R}}
\def\to{\rightarrow}
\def\pcf{post-critically~finite~}
\def\TmapA{f\colon (S^2,A)\righttoleftarrow}
\def\hCbb{\hat{\Cbb}}
\def\Sier{Sierpi\'{n}ski~carpet~}
\def\from{\colon}

\def\bfe{\mathbf{e}}
\def\bft{\mathbf{t}}
\def\band{(t;(e_1,s_1),(e_2,s_2))}
\def\bandn{(t^n;(e_1^n,s^n_1),(e_2^n,s^n_2))}
\def\bandm{(t^m;(e_1^m, s^m_1),(e_2^m,s^m_2))}

\DeclareMathOperator{\Crit}{Crit}


\DeclareMathOperator*{\esssup}{ess.sup}
\DeclareMathOperator{\ARC}{\mathrm{ARC.dim}}
\DeclareMathOperator{\Fill}{\mathrm{Fill}}
\DeclareMathSymbol{\mhyphen}{\mathord}{AMSa}{"39}

\newtheorem{thm}{Theorem}[section]
\newtheorem{coro}[thm]{Corollary}
\newtheorem{lem}[thm]{Lemma}
\newtheorem{prop}[thm]{Proposition}
\newtheorem{conj}[thm]{Conjecture}

\newtheorem{theorem}{Theorem}

\newtheorem{conjecture}[theorem]{Conjecture}

\theoremstyle{definition}

\newtheorem{defn}[thm]{Definition}
\newtheorem{eg}[thm]{Example}
\newtheorem{rem}[thm]{Remark}

\title{Julia sets with Ahlfors-regular conformal dimension one}
\author{Insung Park}
\address{Institute for Mathematical Sciences, Stony Brook University, Stony Brook NY 11794-3660, USA}
\email{insung.park@stonybrook.edu}
\date{}

\begin{document}

\begin{abstract}
	For a \pcf hyperbolic rational map $f$, we show that its Julia set $\Jcal_f$ has Ahlfors-regular conformal dimension one if and only if $f$ is a crochet map, i.e., there is an $f$-invariant connected graph $G$ containing the post-critical set such that $f|_G$ has topological entropy zero.
	
	We use finite subdivision rules to obtain graph virtual endomorphisms, which are 1-dimensional models of \pcf rational maps, and we approximate the asymptotic conformal energies of graph virtual endomorphisms to estimate the Ahlfors-regular conformal dimensions of Julia sets. To prove the main theorem, we also establish the monotonicity of asymptotic conformal energies under the decomposition of rational maps by invariant multicurves.
\end{abstract}

\maketitle

\tableofcontents

\section{Introduction}

Julia sets are fractals defined by the dynamical properties of iterations of rational maps. The dynamics of rational maps tends to be more complicated as their Julia sets are more intricate. Hausdorff dimension is a classical invariant of fractals and has been widely used as a measurement of the complexity of Julia sets. However, Hausdorff dimension is sometimes too sensitive to understand topological properties of the dynamics of rational maps. Ahlfors-regular conformal dimension is a variant of Hausdorff dimension that is less sensitive to geometric deformations. For example, rational maps in the same hyperbolic component have the same Ahlfors-regular conformal dimension, while their Hausdorff dimensions vary. Julia sets of post-critically finite rational maps have Ahlfors-regular conformal dimension between $1$ and $2$. A Julia set has the maximal Ahlfors-regular conformal dimension, $\ARC(\Jcal_f)=2$, if and only if $\Jcal_f$ is the whole Riemann sphere. However, the other extreme case, $\ARC(\Jcal_f)=1$, includes various Julia sets, such as the Julia sets of \pcf polynomials or Newton maps. In this paper, we characterize \pcf hyperbolic rational maps whose Julia sets have Ahlfors-regular conformal dimension one.

\begin{theorem}\label{theorem:main}
For a hyperbolic \pcf rational map $f$, the Ahlfors-regular conformal dimension of the Julia set $\Jcal_f$ is one if and only if $f$ is a crochet map.
\end{theorem}

Let us see more details and backgrounds on the terminologies written in Theorem \ref{theorem:main}.

\subsection*{Hyperbolic \pcf rational maps.}
For a rational map $f\from \hCbb\righttoleftarrow$, a point $z\in \hCbb$ is {\it critical} if $f'(z)=0$, or, equivalently, $f$ is not a local homeomorphism near $z$. Denote by $\mathrm{Crit}(f)$ the set of all critical points of $f$. We define the {\it post-critical set} $P_f$ by
\[
    P_f=\{f^n(c): c \in \mathrm{Crit}(f)~\mathrm{and}~n>0\}.
\]
We say that $f$ is {\it \pcf} if $P_f$ is finite. One reason for the popularity of post-critically finite rational maps is Thurston's characterization, which gives rise to a one-to-one correspondence between \pcf rational maps and certain homotopy classes of \pcf topological branched coverings \cite{DH_ThuChac}. Thus topological ideas can be well applied to investigate the dynamical properties of \pcf rational maps.

A rational map $f\from \hCbb\righttoleftarrow$ is {\it hyperbolic} if every critical point is attracted to an attracting periodic cycle. The set of hyperbolic rational maps of degree $d$, denoted by $\Hcal_d$, is open in the set of all rational maps of degree $d$, denoted by $\mathrm{Rat}_d$. It is a famous conjecture in complex dynamics that $\Hcal_d$ is dense in $\mathrm{Rat}_d$. Each connected component of $\Hcal_d$ is called a {\it hyperbolic component}. By \cite{MMS_JStable}, if two rational maps $f$ and $g$ are in the same hyperbolic component, then their Julia sets $\Jcal_f$ and $\Jcal_g$ are quasi-symmetrically equivalent with respect to the spherical metric. Moreover, if the Julia sets of a hyperbolic component are connected, then the hyperbolic component contains a unique rational map that is post-critically finite, see \cite[Corollary 3.6]{McM_AutoRatlMap} and \cite[Corollary 5.2]{Mil_HypComp}. To sum up, a conjecturally generic (an element in an open and dense subset in this context) connected Julia set is quasi-symmetrically equivalent to the Julia set of a hyperbolic \pcf rational map.

\subsection*{Ahlfors-regular conformal dimension} Two metric spaces $(X,d_X)$ and $(Y,d_Y)$ are {\it quasi-symmetric}, or {\it quasi-symmetrically equivalent}, if there is a distortion function $\eta\from[0,\infty)\righttoleftarrow$ that is a homeomorphism such that, for every distinct $x,y,z\in X$, we have
\[
\frac{d_Y(f(x),f(y))}{d_Y(f(x),f(z))} \le \eta\left( \frac{d_X(x,y)}{d_X(x,z)} \right).
\]
We write $X \sim_{q.s.} Y$ if $X$ and $Y$ are quasi-symmetric. Quasi-symmetric classes of Julia sets have been studied in diverse perspectives, e.g., quasi-symmetries of Julia sets \cite{BonkLyuMeren_QSofCarpets, LyuMeren_QSBasilica, BelkForrest_ThompBasilica, QYZ_QSGeomCarpet} and quasi-symmetric uniformizations of Julia sets \cite{QY_QSUnifCantorCircles, Bonk_QSUnifCarpet, QYZ_QSGeomCarpet}.

For a compact metric space $X$, we denote its Hausdorff dimension by $\mathrm{H.dim}(X)$. A compact metric space $X$ is {\it Ahlfors-regular} if there is a uniform constant $C>1$ such that, for every $x\in X$ and $0<r<\mathrm{diam}(X)$, we have
\[
    \frac{1}{C} r^d < \Hcal^d(B(x,r)) < C r^d,
\]
where $B(x,r)$ is the ball of radius $r$ with center $x$, $d=\mathrm{H.dim}(X)$, and $\Hcal^d$ is the $d$-dimensional Hausdorff measure.

The (resp.\@ {\it Ahlfors-regular}) {\it conformal dimension} of a compact metric space $X$, denoted by $\mathrm{C.dim}(X)$ (resp.\@ $\ARC(X)$), is the infimum of the Hausdorff dimension of $Y$ satisfying $X \sim_{q.s.} Y$ (resp.\@ with $Y$ Ahlfors-regular). Thus, $\mathrm{C.dim}(X)$ and $\ARC(X)$ are invariants of the quasi-symmetric class of the metric space $X$. We say that the (Ahlfors-regular) conformal dimension is {\em attained} if it is equal to the minimum of Hausdorff dimension in the quasi-symmetric class. See \cite{HP_CXC, HP_QIneqCarpets, HP_MinConfDimCXC} for some previous works on the Ahlfors-regular conformal dimension of Julia set. 

The notion of conformal dimension was introduced by Pansu \cite{Pansu_Confdim} in the study of negatively curved rank one symmetric spaces of non-compact type and their boundaries.
Then the idea of conformal dimension was applied in the study of Gromov hyperbolic spaces. For a Gromov hyperbolic space $X$, its boundary at infinity $\partial_\infty X$ has a metric, called a {\em visual metric}, that is determined up to quasi-symmetry. In this respect, taking the infimum of Hausdorff dimension over a quasi-symmetric class is a natural way to obtain an invariant of $X$ or $\partial_\infty X$. See \cite{MackayTyson_Confdim} for a comprehensive account of the conformal dimension.

More recently, Eriksson-Bique published a theorem showing that the conformal dimension and the Ahlfors-regular conformal dimension are the same for compact, connected, locally connected, quasiself-similar metric spaces \cite{EB_EqualityConfDim}. The quasiself-similarity in \cite{EB_EqualityConfDim} is a typical feature of semi-hyperbolic rational maps, where any small part is conformally isomorphic to a region of a fixed size, as formalized, for example, by the [Roundness distortion] axiom in \cite{HP_CXC}. The author has chosen to retain the term ``Ahlfors-regular'' in this paper, but the reader may pay less attention the Ahlfors-regularity condition.

\subsection*{Applications of conformal dimensions in realization problems} For a motivation of studying conformal dimensions, let us review some applications of conformal dimensions to realization problems in geometric group theory and complex dynamics.

Cannon conjectured that for a Gromov hyperbolic group $G$, if $\partial_\infty G$ is homeomorphic to the 2-sphere, then $G$ is virtually isomorphic to the fundamental group of a compact hyperbolic 3-manifold \cite[Conjecture 11.34]{Cannon_Conjecture_11_34}. Bonk and Kleiner proved a partial result that if $\ARC(\partial_\infty G)$ is equal to 2 and attained, then $G$ is virtually isomorphic to the fundamental group of a compact hyperbolic 3-manifold \cite{BonkKleiner_ConfdimHypGp}.

Realization problems in complex dynamics ask what topological branched self-coverings of the 2-sphere are topologically conjugate (sometimes up to homotopy) to the dynamics of rational maps. Thurston's characterization is one of the most famous realization theorems in complex dynamics \cite{DH_ThuChac}. In \cite{HP_MinConfDimCXC}, Ha\"{i}ssinsky and Pilgrim showed that for a coarse expanding conformal map $f$ with the repellor $X$ equal to $S^2$, $\ARC(X)=2$ and attained if and only if $f$ is topologically conjugate to a semi-hyperbolic rational map, which is an analogue of the Bonk--Kleiner's theorem. However, there are many coarse expanding conformal maps such that their repellors are homeomorphic to $S^2$ but they are not topologically conjugate to rational maps, which give counter examples to an analogue of the Cannon's conjecture for coarse expanding conformal maps.

More recently, Ha\"{i}ssinsky showed that for a Gromov hyperbolic group $G$, if $\ARC(\partial_\infty G)<2$ then $G$ is virtually isomorphic to a convex cocompact Kleinian group \cite{Hais_PlanarBdry}. Using a similar technique, Ha\"{i}ssinsky announced a result that a coarse expanding conformal map is topologically conjugate to a rational map if the Ahlfors-regular conformal dimension of its repellor is less than 2. A summary of its proof is in \cite{Haiss_TopChacOfRatlMaps}.

By a work of Bartholid and Dudko \cite{BartholdiDudko_expanding}, if a \pcf topological branched covering $f$ of the sphere does not have a Levy cycle, then the Julia set $\Jcal_f$ is still well-defined. Then, Conjecture \ref{conj:StrictDecr} together with Theorem \ref{thm:PropertyOfComfEnergy} implies that if $f$ is of hyperbolic-type and $\ARC(\Jcal_f)<2$, then $f$ is realized as a rational map. We say that a \pcf topological branched covering is of {\it hyperbolic-type} if every critical point (marked point in general) has a critical periodic point in its forward orbit, see Section \ref{sec:Obstructions} for more details.

\subsection*{Attainment of $\ARC(\Jcal_f)=1$}
As stated in the previous paragraph, rigidity theorems follow from the assumptions that Ahlfors-regular conformal dimensions are attained. The rigidity for the attainment of ${\rm C.dim}(\Jcal_f)=1$ (hence, of $\ARC(\Jcal_f)=1$ as well) was obtained by Wu and the author \cite{ParkWu_Yshape}.

\begin{thm}[{\cite[Theorem 1.6]{ParkWu_Yshape}}]
    Let $f$ be a semi-hyperbolic rational map of degree $d$ with a connected Julia set $\Jcal_f$. If ${\rm C.dim}(\Jcal_f)=1$ and attained, then $\Jcal_f$ is quasi-symmetric to $\Sbb^1$ or $[0,1]$. Moreover, $f$ is a sub-hyperbolic rational map that is quasi-conformally conjugate near Julia sets to $z^d$, $1/z^d$, the degree-$d$ Chebyshev polynomial, or the degree-$d$ negated Chebyshev polynomial.
\end{thm}

\subsection*{Computation of conformal dimensions}
Unlike the case of negatively curved symmetric spaces \cite{Pansu_Confdim}, explicit values of conformal dimensions of Julia sets or boundaries of Gromov hyperbolic groups are mostly unknown.
For a \pcf rational map $f$, the Julia set $\Jcal_f$ has Ahlfors-regular conformal dimension between $1$ and $2$. $\ARC(\Jcal_f)$ is equal to two if and only if $\Jcal_f=\hCbb$. The other extreme case, when $\ARC(\Jcal_f)=1$, is the subject of this article. It was previously shown that $\ARC(\Jcal_f)=1$ if $f$ is a semi-hyperbolic polynomial \cite{Carrasco_ConfDimCombiMod, Kinneberg_ConfDimPlanarDomain} or the mating of the Rabbit and the Basilica polynomials \cite{PilThu_ARConfDim}, see Figure \ref{fig:RabbitandBasilica}.

\subsection*{Decomposition of branched coverings}
To obtain lower bounds of Ahlfors-regular conformal dimensions of Julia sets, we use the decompositions of branched coverings, which are rigorously formulated in \cite{Pilgrim_Combination}. Consider a \pcf branched covering $\TmapA$. Suppose that $\Gamma$ is a {\it completely $f$-invariant} multicurve, i.e., every essential connected component of $f^{-1}(\Gamma)$ is contained in $\Gamma$ up to homotopy ({\it backward invariant}) and every component of $\Gamma$ is homotopic to a component of $f^{-1}(\Gamma)$ ({\it forward invariant}), see Section \ref{sec:Obstructions}. By pinching the sphere along $\Gamma$, we decompose the sphere into several small spheres. There is an induced dynamics among the small spheres. The first return maps of periodic small spheres define \pcf branched coverings $f_i\from (S^2(i),A(i))\righttoleftarrow$, which are called the {\it small branched coverings} of $(f,\Gamma)$. See Section \ref{sec:MonoBrCov} for more details.

\subsection*{Crochet maps}
A \pcf rational map $f$ is a {\it crochet map} if there exists an $f$-invariant connected graph $G$ so that $P_f\subset G$ and the topological entropy of $f|_G\from G\to G$ is zero. Crochet maps contain many interesting families of \pcf rational maps. For example, spiders of \pcf polynomials \cite{Poi_CritPort, BFH_ClassCritPreper, spider} can be used to construct $f$-invariant graphs with zero topological entropy. Also, critically fixed rational maps \cite{pilgrim_Classification_criticallyfixed,Hlushchanka_criticallyfixed}, (second iterates of) critically fixed anti-rational maps \cite{Geyer_CritFixAnti}, \pcf Newton maps \cite{DMRS_ClassifiPCFixedNewton, LMS_ClassfiPCFiniteNewton}, and matings of two \pcf polynomials one of which has core entropy zero, Proposition \ref{prop:MatingCrochet}. See Section \ref{sec:crochet} for more details.

Crochet maps were recently introduced by Dudko--Hlushchanka--Schleicher as elementary building blocks of rational maps, together with \Sier maps, in their decomposition theory \cite{DHS_Decomp}. We call a rational map a {\it \Sier map}, if its Julia set is homeomorphic to the Sierpi\'{n}ski carpet. They proved that if $f$ is a \pcf rational map with a non-empty Fatou set, then there is a canonical completely $f$-invariant multicurve $\Gamma$ such that each small rational map of the $\Gamma$-decomposition is either a \Sier map or a crochet map. They also proved that if $f$ is not a crochet map, then either (1) $\Gamma$ is a Cantor multicurve or (2) the $\Gamma$-decomposition has at least one small \Sier map. We say that $\Gamma$ is a {\it Cantor multicurve} if the number of essential components of $f^{-n}(\Gamma)$ isotopic to elements of $\Gamma$ grows exponentially fast with $n$. In the appendix, we sketch the idea of Dudko--Hlushchanka--Schleicher's decomposition theory.

There is intuitive evidence that every non-crochet map has conformal dimension strictly greater than one. Suppose that $f$ is not a crochet map, so that we are in case (1) or (2) of the previous paragraph. There is a natural map into $\Jcal_f$ either from a Cantor set times the circle, $\Ccal \times \Sbb^1$, in the case (1), or from a \Sier Julia set in the case (2). The space $\Ccal \times \Sbb^1$ is a well-known example having conformal dimension strictly greater than one \cite[Proposition 4.1.11]{MackayTyson_Confdim}. A self-similar \Sier also has conformal dimension strictly greater than one by \cite{Mackay_ConfdimCarpet}. We use the theory of asymptotic conformal energies to prove this rigorously.

\begin{figure}
	\centering
	\def\svgwidth{0.5\textwidth}
	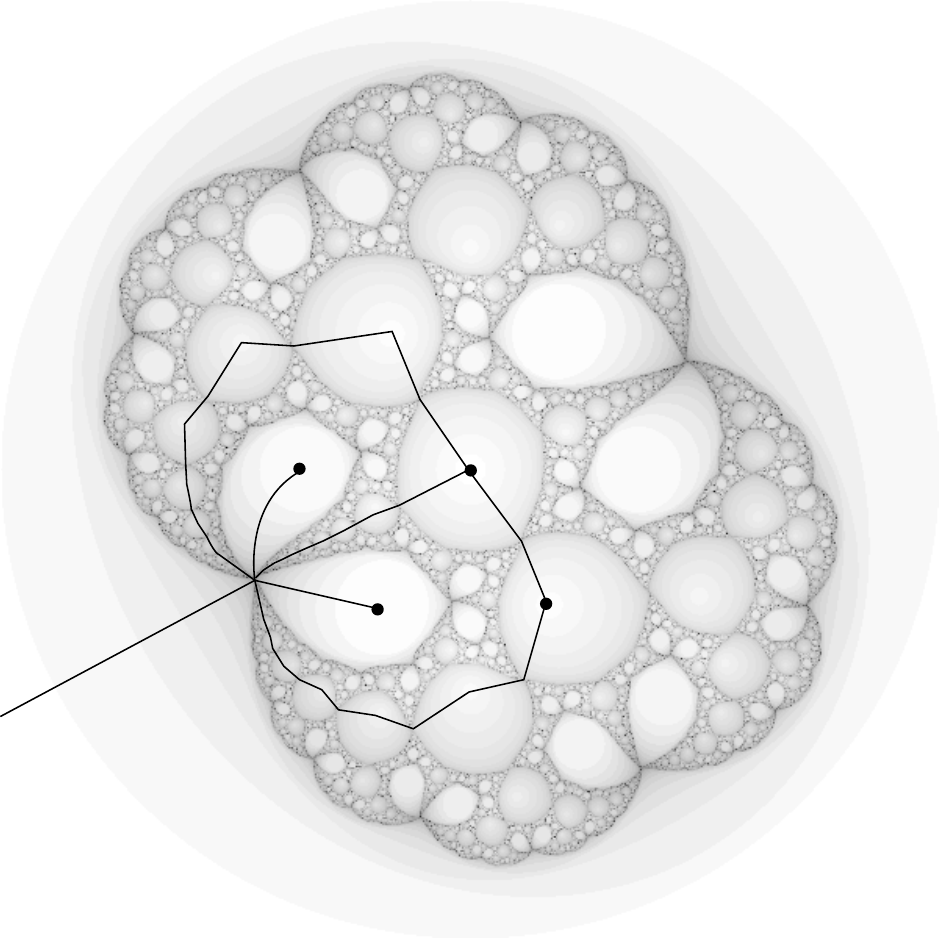
	\caption{The mating of the Rabbit and the Basilica polynomials, $f(z)\approx \frac{z^2+(0.5-0.866 i)}{z^2+(0.5+0.866 i)}$, with an $f$-invariant graph with topological entropy zero drawn. The graph is obtained from the Hubbard tree of the Basilica polynomial and the spider of the Rabbit polynomial, see Proposition \ref{prop:MatingCrochet}.}
		\label{fig:RabbitandBasilica}
	\end{figure}

\subsection*{Asymptotic $p$-conformal energies}
One major technique applied in the proof of Theorem \ref{theorem:main} is the theory of $p$-conformal energies introduced by D.\@ Thurston \cite{ThuElaGraph, Thu_RubberBand, ThuPosChac}. For every \pcf hyperbolic-type topological branched covering of the sphere $\TmapA$ and $p\in [1,\infty]$, we define the {\it asymptotic $p$-conformal energy} $\overline{E}^p(f) \in \Rbb_+$, see Section \ref{sec:AsympConfEng}. Below is a summary of known properties of the asymptotic $p$-conformal energies.

\begin{thm}\label{thm:PropertyOfComfEnergy}
	For every \pcf hyperbolic-type topological branched covering of the sphere $\TmapA$, we have the following properties.
	\begin{enumerate}
		\item $\overline{E}^1(f)\ge 1$ \cite[p.\@ 36]{ThuPosChac}.
        \item $\overline{E}^2(f)<1$ if and only if $f$ is combinatorially equivalent to a rational map \cite[Theorem 1$'$]{ThuPosChac}.
		\item $\overline{E}^\infty(f)<1$ if and only if $f$ does not have a Levy cycle \cite{BartholdiDudko_expanding, Nek_combmodel}.
		\item As a function of $p\in[1,\infty]$, the $\overline{E}^p(f)$ is monotonically decreasing, i.e., $\overline{E}^p(f)\le \overline{E}^q(f)$ if $p>q$ \cite[Proposition 6.10]{ThuPosChac}.
		\item For $p_*=\ARC(\Jcal_f)$, we have $\overline{E}^{p_*}(f)=1$ \cite[Theorem A]{PilThu_ARConfDim}.
	\end{enumerate}
\end{thm}

\begin{rem}
	As for (3), the equivalence between $\overline{E}^\infty(f)<1$ and having contracting biset is almost immediate by \cite{Nek_combmodel}, and the equivalence between having contracting biset and the non-existence of Levy cycles follows from \cite{BartholdiDudko_expanding}. See \cite[Proposition 7.21]{Park_thesis} for more details.
\end{rem}

The affirmative answer to the following conjecture, together with (5) in Theorem \ref{thm:PropertyOfComfEnergy}, implies that Ahlfors-regular conformal dimensions of Julia sets can be uniquely determined by asymptotic conformal energies.

\begin{conj}[{\cite[Conjecture 1.1]{PilThu_ARConfDim}}]\label{conj:StrictDecr}
For a \pcf hyperbolic-type branched covering $\TmapA$, the asymptotic $p$-conformal energy $\overline{E}^p(f)$ is either constant or strictly decreasing as a function of $p \in [1,\infty]$.
\end{conj}

Let $\Gamma$ be a completely $f$-invariant multicurve of $\TmapA$. Following \cite{HP_ThuObsConfDim}, we define in Section \ref{sec:Obstructions} a linear transformation $f_{p,\Gamma}\from  \Rbb^\Gamma \to \Rbb^\Gamma$ for $p\in[1,\infty]$. When $p=2$, the transformation $f_{2,\Gamma}$ is equivalent to the linear transformation used in Thurston's characterization \cite{DH_ThuChac}. Let $\lambda_p(\Gamma)$ be the leading eigenvalue of $f_{p,\Gamma}$. If $\TmapA$ has no Levy cycles, then $\lambda_p(\Gamma)$ is strictly decreasing in $p$ \cite[Lemma A.2]{HP_ThuObsConfDim}. Hence there is a unique number $Q(\Gamma)>0$ so that $\lambda_{Q(\Gamma)}(\Gamma)=1$.  Our second main result is the monotonicity of conformal energies under decompositions by multicurves.

\begin{theorem}[Monotonicity of conformal energies]\label{theorem:MonoBrCov}
Suppose $\TmapA$ is a \pcf hyperbolic-type branched covering with no Levy cycles. Let $\Gamma$ be a completely $f$-invariant multicurve such that 
\[
    \{f_i\from (S^2(i),A(i))\righttoleftarrow:i=1,2,\dots,n\}
\]
is the set of small branched coverings of $(f,\Gamma)$ with the first return times $\{\tau_i\}_{i=1,2,\dots,n}$. Then, for every $p\in[0,\infty]$, we have
\[
    \overline{E}^p(f) \ge \max\left\{ \{\left(\overline{E}^p(f_i)\right)^{1/\tau_i}:1\le i \le n\},~(\lambda_p(\Gamma))^{1/p}\right\}.
\]
Moreover, we have $\ARC(\Jcal_f) \ge Q(\Gamma)$. 
\end{theorem}

\begin{rem}
    The inequality $\ARC(\Jcal_f) \ge Q(\Gamma)$ was previously known \cite[Theorem 1.5]{HP_ThuObsConfDim}.
\end{rem}

The number $(\lambda_p(\Gamma))^{1/p}$ can also be viewed as the asymptotic $p$-conformal energy $\overline{E}^p(f|_\Gamma\from \Gamma\righttoleftarrow)$ of the restricted dynamics to $\Gamma$, see Proposition \ref{prop:EngOfMultiCurve}.

\vspace{5pt}

The next conjecture follows from Conjecture \ref{conj:StrictDecr} and Theorem \ref{theorem:MonoBrCov}.

\begin{conjecture}\label{conjecture:MonoConfdim}
In the setting of Theorem \ref{theorem:MonoBrCov}, we also have
\[
    \ARC(\Jcal_f) \ge \max\left\{ \ARC(\Jcal_{f_i}):1\le i\le n \right\}.
\]
\end{conjecture}

The inequality in Conjecture \ref{conjecture:MonoConfdim} can be strict. Let $f_{(\theta_1,\theta_2)}$ be the post-critically finite hyperbolic quadratic polynomial whose parabolic root of the hyperbolic component is the common landing point of the parameter rays of angles $\theta_1$ and $\theta_2$, see \cite{Hubbard_vol2}. Let $f$ be the mating of two polynomials $f_1=f_{(3/7,4/7)}$ and $f_2=f_{(3/31,4/31)}$ and let $\Gamma=\{\gamma\}$ be the equator. Then $\Jcal_f$ is a Sierpi\'nski carpet so that $\ARC(\Jcal_f)>1$ by \cite{Mackay_ConfdimCarpet}. However, $\ARC(\Jcal_{f_1})=\ARC(\Jcal_{f_2})=1$ by \cite{Carrasco_ConfDimCombiMod}.

\subsection*{Sullivan's dictionary} There exists a long list of analogies, called Sullivan's dictionary, between the dynamics of rational functions and and the action of Kleinian groups. Under Sullivan's dictionary, Julia sets are compared with limit sets. Theorems in this article are comparable with the work on hyperbolic groups by Carrasco--Mackay \cite{CarrascoMackay_ConfdimGp}.

\begin{thm}[{\cite[Theorem 1.1]{CarrascoMackay_ConfdimGp}}]\label{thm:ConfdimGp}
	Let $G$ be a hyperbolic group that is not virtually free. Suppose that there is a graph of groups decomposition of $G$, with vertex groups $\{G_i\}$ and elementary edge groups. Then we have
	\[
	\ARC(\partial_\infty G) =\max\left\{ \{1\} \cup \{\ARC(\partial_\infty G_i):|G_i|=\infty\} \right\}.
	\]
\end{thm}

Theorem \ref{theorem:MonoBrCov} and Conjecture \ref{conjecture:MonoConfdim} can be compared with Theorem \ref{thm:ConfdimGp}. However, the inequality in Conjecture \ref{conjecture:MonoConfdim} may be strict, as previously explained.

Like Theorem \ref{theorem:main}, there is also a characterization of hyperbolic groups whose boundaries have conformal dimension one.

\begin{thm}[{\cite[Corollary 1.2]{CarrascoMackay_ConfdimGp}}]\label{thm:ConfdimOneGp}
	Let $G$ be a hyperbolic group that is not virtually free and has no $2$-torsion. Then, $\ARC(\partial_\infty G)=1$ if and only if the $G$ has a hierarchical decomposition over elementary edge groups so that every vertex group is either elementary or virtually Fuchsian.
\end{thm}

\subsection*{Finite subdivision rules}
A {\it finite subdivision rule} is a version of Markov partition for branched coverings of the sphere. It consists of a CW-complex structure on the sphere, called a {\it subdivision complex}, and a description on how the CW-complex changes by the pull-back of a branched self-covering of the sphere, which is called a {\it subdivision map}. It was originally developed to understand the boundaries of Gromov hyperbolic groups in an attempt to prove the Cannon's conjecture \cite{finite_suvdivision_rule_first}. Later, it turned out that finite subdivision rules are a useful combinatorial tool to describe the dynamics of rational maps \cite{ConstructingRatFromSubdiv, ConstructingSubdivFromRat}. 

Finite subdivision rules are also very useful to detect the existence of Levy cycles \cite{finite_subdivision_rule_exp1, finite_subdivision_rule_exp2,Park_Levy}. According to \cite{BartholdiDudko_expanding}, the non-existence of Levy cycles is equivalent to the dynamical systems being expanding.

In this article, we use finite subdivision rules that are defined by the zero-entropy graphs of crochet maps in \cite{DHS_Decomp}. From the dual $1$-skeletons of the subdivision complexes, in Section~\ref{sec:AsympConfEng}, we construct graph virtual endomorphisms, which are necessary for calculating asymptotic conformal energies. In Section~\ref{sec:DualRecSklton}, we study separated recurrence, a property of finite subdivision rules that reflects the expanding behavior of the subdivision map.

\subsection*{Order theory} When estimating $\overline{E}^p(f)$, we utilize the preorder $\preceq$ on the edge set ${\rm Edge}(S_\R)$ defined by the dynamics of $f$. We extend $\preceq$ to a total preorder $\preceq^*$ satisfying $e\preceq^* e'$ and $e'\preceq^* e$ whenever $e$ and $e'$ are incident to a common Julia vertex. To this end, we discuss a version of order extension theorem in Theorem \ref{thm:PreorderExt}, which is then applied in Theorem \ref{thm:TotalPreorderGraphs} and \ref{thm:TotalPreorderCrochet} for the dynamics of finite subdivision rules of crochet maps.

\subsection*{Non-hyperbolic \pcf crochet maps.}
The proofs of the main theorems in this paper rely on two theories: the theory of asymptotic $p$-conformal energies by D.~Thurston, and the decomposition theory by Dudko--Hlushchanka--Schleicher. The decomposition theory applies to non-hyperbolic rational maps as well. However, the generalization of asymptotic $p$-conformal energies to non-hyperbolic rational maps is still a work in progress. The author expects that the proofs in this paper would remain valid for non-hyperbolic crochet maps using the generalized notion of asymptotic $p$-conformal energies. See Section~\ref{sec:non-hyp} for further details.

\subsection*{Proof of ``crochet map $\Leftarrow$ ARC.dim=1''} If $f$ is not a crochet map, then by \cite{DHS_Decomp} there exists an $f$-invariant multicurve $\Gamma$ such that either (a) $\Gamma$ is Cantor multicurve or (b) at least one small rational map, say $f_i$, of $f$ has a Sierpi\'{n}ski carpet Julia set, see Theorem \ref{thm:DHScrochet}. We have $Q(\Gamma)>1$ for the case (a) and $\overline{E}^1(f_i)>1$ for the case (b) by \cite{PilThu_ARConfDim}. Then $\ARC(\Jcal_f)>1$ follows from the monotonicity theorem (Theorem \ref{theorem:MonoBrCov}), and properties (4) and (5) in Theorem \ref{thm:PropertyOfComfEnergy}.
\begin{align*}
    \begin{minipage}{10cm}
        Overview of post-critically finite rational maps\\
        Monotonicity of conformal energies (Theorem \ref{theorem:MonoBrCov})\\
        Crochet decomposition theory
    \end{minipage}&\minibox{Section \ref{sec:Obstructions}\\ Section \ref{sec:AsympConfEng}\\Sections \ref{sec:crochet}}
\end{align*}

\subsection*{Proof of ``crochet map $\Rightarrow$ ARC.dim=1''} We will show that if $f$ is a crochet map, then $\overline{E}^p(f)<1$ for every $p\in(1,\infty]$. Then by Theorem \ref{thm:PropertyOfComfEnergy} we have $\ARC(\Jcal_f)=1$. 

To define the asymptotic $p$-conformal energy $\overline{E}^p(f)$, we need to choose (a) a graph $H_0$ that is a deformation retract to a certain punctured sphere and (b) a $p$-length structure on $H_0$.

We begin with the zero-entropy graph $G$ constructed in \cite{DHS_Decomp}, and then consider the finite subdivision rule whose $1$-skeleton is $G(\subset \hat{\mathbb{C}})$. In Section~\ref{sec:naturalmap}, we obtain the desired graph $H_0$ via the $\origboxtimes$-deformation to the dual graph of $G$.

On the other hand, we investigate the construction of the zero-entropy graph $G$ and show that there exists a total preorder on $\Edge(G)$ satisfying a certain property. By utilizing this total preorder, we define a $p$-length structure $\alpha_K$ on $H_0$ that satisfies a certain condition, referred to as the {\it $K\,\hat{}$-expanding} property for $K>1$.

By using $H_0$ and $\alpha_K$, we first obtain an upper bound for $\overline{E}^p(f)$ that is slightly greater than one. We then refine this upper bound to be less than one by applying a suitably small homotopy and choosing $K$ sufficiently large, relying on the {\it separated recurrence} property of finite subdivision rules.

\begin{align*}
     \begin{minipage}{10cm}
        The zero-entropy graph $G$ and total preorder \\ Finite subdivision rules and separated recurrence \\ $\origboxtimes$-deformations and estimation of $\overline{E}^p(f)$
     \end{minipage} & \minibox{Sections \ref{Sec:GraphExt and Preorder},\ref{sec:crochet}\\Sections \ref{sec:FSR},\ref{sec:DualRecSklton} \\ Sections \ref{sec:AsympConfEng},\ref{sec:AsympConfEngCro}
    }
\end{align*}









\subsection*{Acknowledgment}

This works is based on author's Ph.D. thesis at Indiana University Bloomington. The author particularly thanks Dylan Thurston and Kevin Pilgrim for their constant support and encouragement throughout the graduate school years. The author also thanks Laurent Bartholdi, Mario Bonk, Dzmitry Dudko, Mikhail Hlushchanka, Jeremy Kahn, Dierk Schleicher, and anonymous referees for useful comments.

The author was supported by the Simons Foundation Institute Grant Award ID 507536 while the author was in residence at the Institute for Computational and Experimental Research in Mathematics in Providence, RI.

\section{Obstructions of branched coverings}\label{sec:Obstructions}

In this section, we review basic definitions and well-known theorems on post-critically finite branched coverings of the $2$-sphere.

A continuous map $f\from S^2 \rightarrow S^2$ is a {\it topological branched covering} if it is locally $z \mapsto z^d$ for some integer $d>0$ under suitable local coordinates. A point $x \in S^2$ is a {\it critical point} of $f$ if $f$ is not locally injective at $x$. The {\it critical set} $\Crit(f)$ of $f$ is the set of all critical points, and its forward orbit $P_f:=\bigcup_{k=1}^\infty f^{k}(\Crit(f))$ is the {\it post-critical set}, where $f^k$ means the $k^{th}$-iteration $f\circ f\circ \dots \circ f$ of $f$. If $P_f$ is finite, we call $f$ a {\it \pcf}branched covering.

For a \pcf branched covering $f$, a finite set $A \subset S^2$ is called a {\it set of marked points} if $f(A) \cup P_f \subset A$. Every element $a\in A$ is called a {\it marked point}. We say that $\TmapA$ is \pcf branched covering {\it with a set of marked points $A$}. When the set of marked points $A$ is understood from the context, we simply call $\TmapA$ a branched covering, and we add more adjectives when they need to be recalled or specified.

\begin{defn}[Combinatorial equivalence]
We say that two branched coverings $f\from (S^2,A)\righttoleftarrow$ and $g\from (S^2,B)\righttoleftarrow$ are {\it combinatorially equivalent (\,by $\phi_0$ and $\phi_1$)} if there exist orientation-preserving homeomorphisms
\[
    \phi_0,\phi_1\from (S^2,A) \rightarrow (S^2,B)
\]
such that (a) $\phi_0(A)=\phi_1(A)=B$, (b) $\phi_0$ and $\phi_1$ are homotopic relative to $A$, and (c) the following diagram commutes:
\[
\begin{tikzcd}
    (S^2,A) \arrow[r,"\phi_0"] \arrow[d,"f"]
& (S^2,B) \arrow[d,"g"] \\
(S^2,A) \arrow[r,"\phi_1"]
& (S^2,B)
\end{tikzcd}
\]
\end{defn}

\subsection*{Fatou and Julia marked points, hyperbolic-type branched coverings} Let $\TmapA$ be a branched covering. A marked point $a\in A$ is a {\it Fatou point} if $f^k(a)$ is a periodic critical point for some $k\ge 0$. Otherwise $a$ is a {\it Julia point}. The branched covering $\TmapA$ is of {\it hyperbolic-type} if every marked point is a Fatou point. We remark that the hyperbolicity depends not only on the covering map but also on the set of marked points.

\vspace{5pt}
Let $A\subset S^2$ be a finite set. For $a \in A$, a simple closed curve $\gamma$ is a {\it peripheral loop} of $a$ if $\gamma$ bounds a disk $D$ with $D \cap A=\{a\}$. In this case, we say that $\gamma$ is {\it peripheral to $a\in A$}. A simple closed curve $\gamma\subset S^2 \setminus A$ is {\it essential relative to A} if  $\gamma$ is neither homotopically trivial relative to $A$ nor peripheral to any $a\in A$. A {\it multicurve} $\Gamma$ of $(S^2,A)$ is a collection of disjoint essential simple closed curves (or their homotopy classes, depending on the context) that are non-isotopic to each other relative to $A$.

\subsection*{Invariance of graphs and multicurves.}
Let $f\from (S^2,A)\righttoleftarrow$ be a branched covering.
\begin{itemize}
    \item A graph $G\subset S^2$ is {\it forward $f$-invariant up to isotopy relative to $A$} if there exist a subgraph $H$ of $f^{-1}(G)$ and a homeomorphism $\phi\from S^2\to S^2$ such that $\phi(H)=G$ and $\phi$ is isotopic to the identity map relative to $A$.
    \item A graph $G$ is {\it forward $f$-invariant}, or shortly {\it $f$-invariant}, if $f(G)\subset G$ and $f(\V(G))\subset \V(G)$.
    \item A multicurve $\Gamma$ on $(S^2,A)$ is {\it forward $f$-invariant up to isotopy relative to $A$} if it is so as a graph, i.e., for every $\gamma\in \Gamma$, there exists $\gamma'\in \Gamma$ such that $\gamma$ is isotopic to a connected component of $f^{-1}(\gamma')$.
    \item A multicurve $\Gamma$ is {\it backward $f$-invariant up to isotopy relative to $A$}, or {\it $f$-stable}, if every connected component of $f^{-1}(\gamma)$ for $\gamma \in \Gamma$ is either isotopic to an element of $\Gamma$ or inessential relative to $A$.
    \item A multicurve $\Gamma$ is {\it completely $f$-invariant up to isotopy relative to $A$} if $\Gamma$ is both forward and backward invariant under $f$ up to isotopy relative to $A$.
\end{itemize}
When the set of marked points $A$ is understood, we omit the phrase ``relative to $A$''. Multicurves are always considered up to isotopy relative to $A$, so we also omit the phrase ``up to isotopy'' when referring to multicurves.

\subsection*{Linear $p$-transformations}

Let $f\from (S^2,A)\righttoleftarrow$ be a branched covering and $\Gamma$ be a multicurve of $(S^2,A)$. For $p\in[1,\infty]$, the {\it linear $p$-transformation of $\Gamma$} is a linear map $f_{p,\Gamma}\from \Rbb^\Gamma \to \Rbb^\Gamma$ defined by
\[
    f_{p,\Gamma}[\gamma_i]=\sum_{\gamma_j\in \Gamma}\left(\sum_{\gamma_j\sim\gamma\subset f^{-1}(\gamma_i)} \left(\deg(f|_\gamma\from \gamma\to\gamma_i)\right)^{1-p}\right)[\gamma_j],
\]
where $[\gamma_i]$ means the isotopy class of $\gamma_i$ and the inner sum is taken over connected components $\gamma$ of $f^{-1}(\gamma_i)$ that are isotopic to $\gamma_j$. We define the inner sum to be zero if there is no such $\gamma$. For $p=\infty$, we define
\[
    f_{\infty,\Gamma}[\gamma_i]:=\lim_{p\to\infty} f_{p,\Gamma}[\gamma_i].
\]
Note that $\lim_{p\to \infty}\left(\deg(f|_\gamma\from \gamma\to\gamma_i)\right)^{1-p}$ is either $1$ or $0$ depending on whether $\deg(f|_\gamma\from \gamma\to\gamma_i)$ is $1$ or greater than $1$.

Denote by $\lambda_p(\Gamma)$ the leading eigenvalue of $f_{p,\Gamma}$, which is a non-negative real number by the Perron-Frobenius theorem.

Under a suitable choice of ordered basis, $f_{p,\Gamma}$ can be represented as an upper block triangular matrix whose entries are non-negative integers.
\begin{equation*}
    \begin{pmatrix}
    A_{p,1} & * & \dots & * & *\\
        0 & A_{p,2} & \dots & * & *\\
        \vdots & \vdots & \ddots& \vdots & \vdots\\
        0& 0 & \dots & A_{p,k-1} & *\\
        0& 0 & \dots & 0 & A_{p,k}
    \end{pmatrix}    
\end{equation*}
Let $A_{p,1},A_{p,2},\dots,A_{p,k}$ be the square matrices on the diagonal and $\lambda_{p,i}$ be the leading eigenvalue of $A_{p,i}$. Then $\lambda_p(\Gamma)=\max_i \lambda_{p,i}$. Each $A_{p,i}$ is either zero or irreducible. For each $A_{p,i}$, there is a corresponding sub-multicurve $\Gamma_i$ of $\Gamma$, possibly $\Gamma_i=\Gamma$ when $k=1$, such that $f_{p,\Gamma_i}=A_{p,i}$. Note that $A_{p,i}$ is irreducible for some $p\in [1,\infty)$ if and only if $A_{p,i}$ is irreducible for every $p\in[1,\infty)$. We say that $\Gamma_i$ is {\it irreducible} if $A_{p,i}$ is irreducible for $p\in [1,\infty)$.

It follows from the definition that every $A_{1,i}$ and $A_{\infty,i}$ consist of non-negative integers. Hence, for $p\in \{1,\infty\}$, if $A_{p,i}$ is non-zero, then $\lambda_{p,i} \ge 1$, see \cite[Lemma 3.8]{Park_Levy}. If $A_{\infty,i}$ is non-zero, then $\Gamma_i$ contains a {\it Levy cycle}, which is defined as follows.
\begin{defn}[Levy cycle]
Let $\TmapA$ be a branched covering. A multicurve $\Gamma=\{\gamma_1,\gamma_2,\dots,\gamma_n\}$ of $(S^2,A)$ is a {\it Levy cycle} if, for every $i$ (mod $n$), there is a connected component $\gamma_{i+1}'$ of $f^{-1}(\gamma_i)$ that is isotopic relative to $A$ to $\gamma_{i+1}$ such that $\deg(f|_{\gamma_{i+1}'})=1$.
\end{defn}
There are two cases \cite[Lemma A.2]{HP_ThuObsConfDim}:
\begin{itemize}
    \item [(1)] If $A_{p,i}=0$ for every $1\le i \le k$, then $f_{p,\Gamma}$ is nilpotent and $\lambda_p(\Gamma)=0$ for every $p\in[1,\infty]$;
    \item [(2)] If at least one $A_{p,i}$ is irreducible, then $\lambda_1(\Gamma)\ge 1$ and
    \begin{itemize}
        \item [(2-i)]  $\lambda_\infty(\Gamma)=0$ and $\lambda_p(\Gamma)$ is strictly decreasing in $p$, or 
        \item [(2-ii)] $\lambda_\infty(\Gamma)\ge1$ and $\Gamma$ contains a Levy cycle.
    \end{itemize} 
\end{itemize}

Assume that $\Gamma$ does not contain any Levy cycles but does contain an irreducible sub-multicurve. It follows from (2) above that there is a unique $Q(\Gamma)\ge 1$ with $\lambda_{Q(\Gamma)}(\Gamma)=1$. We say that $Q(\Gamma)$ is the {\it critical exponent} of $\lambda_p(\Gamma)$ as a function of $p$. The critical exponent $Q(\Gamma)$ can also be defined as the conformal dimension of a Cantor circle, which is the product space of a Cantor set and a circle, associated with the pair $(f,\Gamma)$. See \cite[p.\@ 55]{PilThu_ARConfDim}.

An $f$-stable multicurve $\Gamma$ with $\lambda_2(\Gamma)\ge 1$ is called a {\it Thurston obstruction}, which prevents $\TmapA$ from being combinatorially equivalent to a \pcf rational map \cite{DH_ThuChac}.

\begin{rem}
    In Proposition \ref{prop:EngOfMultiCurve}, we allow $\Gamma$ to contain peripheral loops around Fatou marked points. The above discussion still holds for this generalization.
\end{rem}

\section{Finite subdivision rules}\label{sec:FSR}

A finite subdivision rule is a recursive formula to generate a sequence of subdivisions of a given CW-complex. In this section, we review basic definitions of finite subdivisions rules and introduce new notions, such as bands, which will be used in subsequent sections. Then, we define directed graphs encoding the dynamics of finite subdivision rules.

\subsection{Definitions of finite subdivision rules}

A {\it finite subdivision rule} $\R$ consists of
\begin{itemize}
    \item a finite CW-complex structure $S_\R$ of the sphere $S^2$ whose 2-cells have at least two 1-cells on their boundaries,
    \item a subdivision $\R(S_\R)$ of $S_\R$, and
    \item a continuous map $f\from \R(S_\R)\to S_\R$ that is {\it cellular}, i.e., it sends each open cell onto an open cell homeomorphically. We call $f$ the {\it subdivision map} of $\R$.
\end{itemize}
See Figure \ref{fig:BasilicaFSR} for an example.

By considering $\R(S_\R)$ and $S_\R$ as different CW-complex structures on the same sphere $S^2$, we see $f$ as a map from $S^2$ to itself. A finite subdivision rule $\R$ is {\it orientation preserving} (resp.\@ {\it reversing}) if $f$ is orientation preserving (resp.\@ reversing). In this paper, we always assume $\R$ is orientation preserving. Because the subdivision map $f\from \R(S_\R) \to S_\R$ is cellular, $f$ is a \pcf branched covering of the sphere with $P_f \subset \V(S_\R)$.

We refer to closed $0$-cells, $1$-cells, and $2$-cells as {\it vertices}, {\it edges}, and {\it tiles} respectively. For a CW-complex $X$, the $1$-skeleton $X^{(1)}$ is a CW-subcomplex of $X$ consisting of all the 0-cells and 1-cells of $X$.

\subsection*{Subdivision complexes.} Consider the subdivision map $f\from \R(S_\R) \to S_\R$ of a finite subdivision rule $\R$. By subdividing the range $S_\R$ to $\R(S_\R)$ and pulling it back via $f$, we obtain a further subdivision $\R^2(S_\R)$ of the domain $\R(S_\R)$. By iterating this process, for each $n\ge0$, we obtain the {\it level-$n$ subdivision complex $\R^n(S_\R)$}. A {\it level-$n$ vertex}, a {\it level-$n$ edge}, or a {\it level-$n$ tile} is a vertex, an edge, or a tile of $\R^n(S_\R)$, respectively.

\subsection*{$\R$-complexes} For a finite subdivision rule $\R$, a CW-complex $X$ is a {\it $\R$-complex} if there is a cellular continuous map $h\from X \to S_\R$. By pulling back the level-$n$ subdivision complex $\R^n(S_\R)$ via $h$, we obtain the level-$n$ subdivision complex $\R^n(X)$ of $X$.

\subsection*{Edge and tile types}
Consider the subdivision map $f\from \R(S_\R) \to S_\R$ of a finite subdivision rule $\R$. For $k\ge 2$, a {\it polygon (or $k$-gon)} $X$ is a CW-complex whose underlying space is homeomorphic to the closed $2$-disk such that $X$ has one $2$-cell and $k$ edges and vertices on the boundary. Let $t$ be a closed level-$0$ tile. There is a characteristic map $\phi_t\from \bft \to t$ where $\bft$ is a polygon. We refer to $\bft$ as the {\it tile type} of $t$. {\it Edge types} are similarly defined as the domains $\bfe$ of characteristic maps $\phi_e\from \bfe \to e$ for level-$0$ edges $e$. Note that every edge or tile type is a $\R$-complex. Hence, the level-$n$ subdivisions $\R^n(\bfe)$ and $\R^n(\bft)$ of an edge type $\bfe$ and a tile type $\bft$ are well-defined. We define level-$n$ subdivisions $\R^n(e)$ and $\R^n(t)$ of $e$ and $t$ as the images of $\R^n(\bfe)$ and $\R^n(\bft)$ under characteristic maps, respectively.

There is a 1-1 correspondence between tile (resp.\@ edge) types and level-$0$ tiles (resp.\@ edges). We often use these terms interchangeably, but when distinction is needed, we use bold letters for edge or tile types and regular letters for edges or tiles in the subdivision complex.

For a level-$0$ tile $t$, a level-$n$ tile $t^n$ is {\it of type $t$ } (or {\it type $\bft$}) if $f^n(t^n)=t$. A level-$n$ edge $e^n$ is {\it of type $e$} (or {\it type $\bfe$}) if $f^n(e^n)=e$ where $e$ is a level-$0$ edge. Note that we use level-$0$ edges or tiles even when referring to their types. While this could be a bit confusing, it significantly simplifies the notations in later sections.

If $\deg(f)=d$, then for every $n\ge 0$, there are $d^n$ level-$n$ tiles or edges of the same type. If $t^n$ is of type $t$, then there is a characteristic map $\phi_{t^n}\from \bft \to t^n$ so that
\[
    f^n \circ \phi_{t^n}=\phi_t\from \bft \to t\subset S_\R.
\]
For $n>m\ge0$, a level-$n$ tile $t^n$ is a {\it subtile} of a level-$m$ tile $t^m$ if $t^n\subset t^m$, and a level-$n$ edge $e^n$ is a {\it subedge} of a level-$m$ edge $e^m$ if $e^n \subset e^m$.

As stated above, our convention is to use superscripts to indicate the levels of edges and tiles. Since level-$0$ objects are frequently used, we often omit the superscript $^0$ for level-$0$ objects.

\subsection*{Incidence among tiles, edges, and vertices}
Consider a level-$n$ tile $t^n$. If a level-$n$ vertex $v^n$ is a vertex of $t^n$, we say that $t^n$ is {\it incident} to $v^n$ and vice versa. Let $\phi_{t^n}\from \bft \to t^n$ be the characteristic map of $t^n$. We define the {\it degree of incidence} of $t^n$ at $v^n$ to be $|\phi_{t^n}^{-1}(v^n)|$. We say that $t^n$ is {\it singly incident} to $v^n$ if the degree of incidence is one, and {\it multiply incident} to $v^n$ otherwise. We also say that $t^n$ is {\it doubly} incident and {\it triply} incident to $v^n$ if the degree of incidence is two and three, respectively. We define the degree of incidence of $t^n$ at $v^n$ as zero if $t^n$ is not incident to $v^n$. Similarly, we define the {\it degree of incidence} of a level-$n$ tile $t^n$ to a level-$n$ edge $e^n$ by the number of edges of $\bft$ mapped onto $e^n$ by the characteristic map $\phi_{t^n}\from \bft \to t^n$. The degree of incidence of $t^n$ to $e^n$ is always zero, one, or two.

\subsection*{Sides of edges and corners at vertices}
Recall that $\R^n(S_\R)^{(1)}$ denotes the $1$-skeleton of the level-$n$ subdivision complex $\R^n(S_\R)$.

For each level-$n$ edge $e^n$, take a small closed disk neighborhood $D(e^n)$ of $e^n$ with $e^n = D(e^n) \cap \R^n(S_\R)^{(1)}$. Define a set ${\rm Side}(e^n):=\pi_0(D(e^n)\setminus e^n)$. Then $|{\rm Side}(e^n)|=2$. Each element of ${\rm Side}(e^n)$ is called a {\it side} of $e^n$, which represents a connected component of $D(e^n)\setminus e^n$. For each $s\in {\rm Side}(e^n)$, we define {\it the level-$n$ tile on the side $s$ of $e^n$}, denoted by ${\rm Tile}(e^n,s)$, as the level-$n$ tile that contains the connected component of $D(e^n)\setminus e^n$ represented by $s$.

Consider a level-$n$ tile $t^n$ with the characteristic map $\phi_{t^n}\from \bft \to t^n$. Then, there exists a natural bijection between the boundary edges $\bfe$'s of $\bft$ and the pairs $(e^n,s^n)$'s where $e^n$ is a boundary edge of $t^n$ and $s^n\in {\rm Side}(e^n)$ with ${\rm Tile}(e^n,s^n)=t^n$. We call the edge $\bfe$ corresponding to $(e^n,s^n)$ the {\it lift} of $(e^n,s^n)$ via $\phi_{t^n}$.

For each level-$n$ vertex $v^n$, we choose a small closed disk neighborhood $D(v^n)$ such that, for every level-$n$ edge $e^n$, we have $e^n\cap D(v^n)\neq \emptyset$ if and only if $v^n\in e^n$. Define a set ${\rm Corner}^n(v^n):= \pi_0 \left( D(v^n)\setminus \R^n(S_\R)^{(1)}\right)$. Then $|{\rm Corner}^n(v^n)|=\deg_{\R^n(S_\R)}(v^n)$. Each element of ${\rm Corner}^n(v^n)$ is referred to as a {\it (level-$n$) corner} at $v^n$, which represents a connected component of $D(v^n)\setminus (\R^n(S_\R))^{(1)}$. For each $c \in {\rm Corner}^n(v^n)$, we define {\it the level-$n$ tile containing the corner $c$ of $v^n$}, denoted by ${\rm Tile}(v^n,c)$, as the level-$n$ tile that contains the connected component of $D(v^n)\setminus (\R^n(S_\R))^{(1)}$ represented by $c$.

For any $n>m$, each level-$m$ vertex $v^m$ is also a vertex of $\R^n(S_\R)$. However, the set of corners at $v^m$ may differ between levels; in general, we have $|{\rm Corner}^n(v^m)|\ge|{\rm Corner}^m(v^m)|$, and the inequality may be strict.

Below are some natural relations among sides and corners.
\begin{itemize}
    \item For $e^m \in \Edge(\R^m(S_\R))$ and $e^n \in \Edge(\R^n(S_\R))$ with $n>m>0$ and $e^n \subset e^m$, there is a bijection ${\rm Side}(e^n) \leftrightarrow {\rm Side}(e^m)$.
    \item For all $n>m\ge0$ and $v^m \in \V(\R^m(S_\R))$, there is a surjection ${\rm Corner}^n(v^m) \twoheadrightarrow {\rm Corner}^m(v^m)$. We note that $|{\rm Corner}^n(v^m)|=\deg_{\R^n(S_\R)}(v^m) \ge \deg_{\R^m(S_\R)}(v^m)=|{\rm Corner}^m(v^m)|$. The equality holds for periodic Julia vertices (Lemma \ref{lem:DegVertices}).
    \item For $e^n \in \Edge(\R^n(S_\R))$ and a vertex $v^n$ of $e^n$, there is a map ${\rm Side}(e^n)\to {\rm Corner}^n(v^n)$, which is an injection if $\deg_{\R^n(S_\R)}(v^n)\ge 2$.
    \item For $e^0 \in \Edge(S_\R)$ and $v^n \in \V(\R^n(S_\R))$ with $v^n \in {\rm int}(e^0)$ and $n\ge1$, there is a surjection ${\rm Corner}^n(v^n)\twoheadrightarrow {\rm Side}(e^0)$.
\end{itemize}

\subsection*{Finite subdivision rules associated with \pcf rational maps}

\begin{defn}[Finite subdivision rules associated with \pcf rational maps]\label{defn:FSR Asso to PCF}
    Let $f$ be a \pcf rational map. Suppose that $G$ is an $f$-invariant connected graph with $f(\V(G))\cup P_f\subset \V(G)$. Then, one can define a finite subdivision rule $\R$ such that the $1$-skeletons of the level-$0$ and level-$1$ subdivision complexes, $S_\R^{(1)}$ and $(\R(S_\R))^{(1)}$, are equal to $G$ and $f^{-1}(G)$, respectively, and $f$ is the subdivision map. We refer to $\R$ as a finite subdivision rule {\it associated with a \pcf rational map $f$ (and its invariant graph $G$)}.
\end{defn}

Let $\R$ be a finite subdivision rule. The statement that $\R$ is the finite subdivision rule associated with a \pcf rational map $F$ is stronger than the statement that its subdivision map $f$ is combinatorially equivalent to a \pcf rational map $F$. The latter is equivalent to the existence of a connected graph $H$ with $f(\V(H)) \cup P_f \subset \V(H)$ that is forward $F$-invariant only up to isotopy relative to $P_F$, whereas the former requires the existence of such a graph that is invariant without passing to isotopy.

\subsection*{Expanding conformal orbifold metrics}
A key feature distinguishing \pcf rational maps from general topological branched coverings is the expanding property with respect to orbifold metrics. We give a brief explanation; see \cite[Section 19]{Milnor_Book} for further details.

For a \pcf rational map $f$, its {canonical orbifold} $(\hCbb,\nu)$ is defined by the ramification function $\nu\from \hCbb \to \{1,2,\dots,\infty\}$, which is defined by
\[
    \nu(x):={\rm lcm}\{\deg_y(f):f^n(y)=(x) {\rm ~for~some~}n>0\}.
\]
Then, the function $\nu$ satisfies (a) $\nu(x)=1$ for $ x \notin P_f$ and (b) $\nu(y) \deg_y(f):\nu(x)$ for each $y\in f^{-1}(x)$. 

The universal cover $\widetilde{S}$ of $(\hCbb,\nu)$ is biholomorphic to either $\Cbb$ or $\Dbb$. We equip $\widetilde{S}$ with the Euclidean metric if $\widetilde{S}\cong \Cbb$ and the hyperbolic metric if $\widetilde{S}\cong \Dbb$, and then descend the metric to $(\hCbb,\nu)$ to obtain the {\it canonical orbifold metric} $\mu$ on $(\hCbb,\nu)$.

The metric $\mu$ is contracting under lifting in the following sense: Define $V_\infty:=\{x\in \hCbb : \nu(x)=\infty\}$. Suppose that $\alpha$ is a $\mu$-rectifiable curve in $\hCbb\setminus V_\infty$ whose length $l_\mu(\alpha)$ with respect to $\mu$ satisfies  $l_\mu(\alpha)<C$ for some $C>0$. For each lift $\widetilde{\alpha}$ of $\alpha$ through $f$, we have
\begin{equation}\label{eqn:Contract_Conf.metric}
    l_\mu(\widetilde{\alpha})< D\,l_\mu(\alpha),
\end{equation}
where the constant $D\in(0,1)$ depends only on $C$. This property can be viewed as forward-expanding (or backward-contracting), and it gives rise to a notion of combinatorial expansion for finite subdivision rules, called separated recurrence; see Theorem~\ref{thm:SepRecurr}.

\subsection*{Bands} Consider a finite subdivision rule $\R$ and its level-$n$ subdivision complexes $\R^n(S_\R)$. We introduce the notion of {\it bands}, which capture the idea of connecting two boundary edges of a tile within the tile.

A level-$n$ {\it band} is a tuple $b^n=\bandn$ where (a) $t^n$ is a level-$n$ tile and (b) $e^n_1,e^n_2$ are boundary edges of $t^n$ and $s^n_1\in {\rm Side}(e^n_1),s^n_2\in {\rm Side}(e^n_2)$ such that $t^n={\rm Tile}(e^n_1,s^n_1)={\rm Tile}(e^n_2,s^n_2)$ and $(e^n_1,s^n_{1})\neq (e^n_2,s^n_{2})$. The equality $e^n_1=e^n_2(=:e^n)$ can happen only when $t^n$ is doubly incident to $e^n$. The edges $e^n_1$ and $e^n_2$ are called the {\it side edges} of the band $b^n$. There are many level-$n$ bands whose side edges are not subedges of level-$0$ edges, which are not of our interest. So we additionally assume that (c) $e^n_1$ and $e^n_2$ are subedges of level-$0$ edges.

We can also define a level-$n$ band by using a triple $(t^n;\bfe_1,\bfe_2)$ where $\bfe_1$ and $\bfe_2$ are two boundary edges of the tile type $\bft$ of $t^n$. Then, $(t^n;\bfe_1,\bfe_2)$ can be bijectively corresponded with $\bandn$ via the characteristic map $\phi_{t^n}\from \bft \to t^n$, where for $i\in\{1,2\}$, $\bfe_i$ is the lift of $(e^n_i,s^n_i)$ via $\phi_{t^n}$.

If $t^n$ is singly incident to $e^n_i$ for $i \in \{1,2\}$, then there exists a unique side $s^n_i$ of $e^n_i$ such that $t^n = \mathrm{Tile}(e^n_i, s^n_i)$. In this case, it is redundant to specify the side $s^n_i$ in the definition of the band $b^n$.

For $n>m\ge0$, a level-$n$ band $\bandn$ is a {\it subband} of a level-$m$ band $\bandm$ if $t^n\subset t^m$ and for each $i\in \{1,2\}$, $e^n_i\subset e^m_i$ and $s_i^n\mapsto s_i^m$ under the natural bijection ${\rm Side}(e^n_i)\leftrightarrow {\rm Side}(e^m_i)$.

A {\it band type} is a level-$0$ band. Let $b=\band$ be a level-$0$ band. A level-$n$ band $b^n=\bandn$ is {\it of type $b$} if $f^n(b^n)=b$, i.e., $f^n(t^n)=t$, $f^n(e^n_i)=e_i$, and $f^n(s^n_i)=s_i$ for $i \in \{1,2\}$. When $f^n(e_i^n) = e_i$, the map $f^n$ naturally induces a bijection $f^n \colon \mathrm{Side}(e_i^n) \to \mathrm{Side}(e_i)$.

\subsection*{Julia or Fatou vertices, edges, and tiles.} A vertex $v \in \V(\R^n(S_\R))$ is a {\it Fatou vertex} if its forward orbit contains a periodic critical point of the subdivision map $f$. Otherwise, we call $v$ a Julia vertex. We denote by $\V_F(\R^n(S_\R))$ (resp.\@ $\V_J(\R^n(S_\R))$) the set of level-$n$ Fatou (resp.\@ Julia) vertices. For $v\in \R^n (S_\R)$, we denote by $\deg_{\R^n(S_\R)}(v)$ the degree of $v$ as a vertex of the graph $\R^n(S_\R)^{(1)}$.

A vertex of $\R^n(S_\R)$ is periodic if it is a periodic point of the subdivision map $f$. For every $n,m>0$, we have $f^m(\V(\R^{n+m}(S_\R)))=\V(\R^n(S_\R))$. Hence, every level-$n$ vertex that is periodic is also a level-$0$ vertex.

\begin{lem}\label{lem:DegVertices}
    Consider a periodic vertex $v\in \V(S_\R)$. If $v$ is a Julia vertex, then
    \begin{enumerate}
        \item $\deg_{S_\R}(v)=\deg_{S_\R}(w)$ for every $w\in \V(S_\R)$ in the periodic cycle of $v$ and
        \item $\deg_{\R^n(S_\R)}(v)=\deg_{S_\R}(v)$ for every $n\ge0$.
    \end{enumerate}
    If $v$ is a Fatou vertex, then $\lim_{n\to \infty}\deg_{\R^n(S_\R)}(v)= \infty$.
\end{lem}
\begin{proof}
    The lemma follows from the fact that for each $n\ge0$, we have
    \begin{equation}\label{eqn:Deg Vertex SubdivCplx}
        \deg_{\R^{n+1}(S_\R)}(v)=\deg_v(f)\cdot \deg_{\R^n(S_\R)}(f(v)),
    \end{equation}
    where $\deg_v(f)$ is the local degree of $f$ at $v$.
\end{proof}

\begin{lem}\label{lem:Edge_inc_PerVert}
    Consider a periodic Julia vertex $v\in \V(S_\R)$. Recall that $v$ is also a vertex of $\R^n(S_\R)$ for all $n\ge0$. Every level-$n$ edge $e^n$ incident to $v$ is recurrent.
\end{lem}
\begin{proof}
    The proof is immediate from Lemma \ref{lem:DegVertices}-(2).
\end{proof}

We call an edge $e\in\Edge(\R^n(S_\R))$ an {\it $FF$-edge} or a {\it $JJ$-edge} if its endpoints are both Fatou vertices or both Julia vertices, respectively. If $e$ has one Fatou vertex and one Julia vertex, then we refer to it as an {\it $FJ$-edge}. We denote by $\Edge_{FF}(\R^n(S_\R))$, $\Edge_{FJ}(\R^n(S_\R))$, and $\Edge_{JJ}(\R^n(S_\R))$ the sets of $FF$-edges, $FJ$-edges, and $JJ$-edges of $\R^n(S_\R)$, respectively.

Consider a tile $t$ with the characteristic map $\phi_t \colon \bft \to t$. 
A vertex ${\bf v}$ of $\bft$ is called a Fatou (resp.\ Julia) vertex if $\phi_t({\bf v})$ 
is a Fatou (resp.\ Julia) vertex. 
The $FF$- and $JJ$-edges of the tile type $\bft$ are defined accordingly.

\begin{rem}
    For the proofs of the main theorems of this paper, we only need to consider the case in which every edge is an $FJ$-edge, namely property (P1) in Section~\ref{Sec:GraphExt and Preorder}. However, for future reference, we discuss the general case of finite subdivision rules in Sections~\ref{sec:FSR}--\ref{sec:DualRecSklton}.
\end{rem}

\subsection{Directed graphs of subdivisions}

We use directed graphs to describe the dynamical properties of subdivision maps.

\subsubsection{Definitions and properties of directed graphs}
Let $G$ be a finite directed graph. We allow loop-edges, which are edges whose initial and terminal vertices are the same, and multiple edges, which are different edges with the same initial and terminal vertices.

A {\it (directed) path} is a sequence of edges $e_1,e_2,\dots, e_n$ where the terminal vertex of $e_i$ is equal to the initial vertex of $e_{i+1}$ for every $i\in \{1,2,\dots,n-1\}$. The {\it length} of a path is defined as the number of edges in the path. The initial vertex of $e_1$ (resp.\@ the terminal vertex of $e_n$) is the initial (resp.\@ terminal) vertex of the path. If $v$ and $w$ are the initial and terminal vertices of a path $p$, then we say that $p$ is a {\it path from $v$ to $w$}. A {\it cycle} is a path whose initial and terminal vertex coincide. We say that two cycles are {\it different} if one is not a cyclic permutation of the other.

\begin{defn}[Recurrent paths and vertices]\label{defn:RecPath}
Let $G$ be a finite directed graph. A path $p$ from $v$ to $w$ is {\it recurrent} if there exists a path from $w$ to $v$. A vertex $v\in \V(G)$ is {\it recurrent} if it is contained in a cycle.
\end{defn}

By definition, if a path from a vertex $v$ is recurrent, then $v$ is recurrent.

\subsection*{Growth rate of paths in directed graphs.} For $v\in \V(G)$ and $n\ge 0$, we denote by $P(v,n)$ the number of directed paths of length $n$ starting from~$v$.

Let $\{a_n\}_{n\ge0}$ be a sequence of non-negative numbers. We say that $\{a_n\}$ {\it has an exponential growth rate} or {\it grows exponentially fast with $n$} if there exist $C\ge D>1$ so that for any $n\ge0$, we have
\[
    D \le \log a_n \le C.
\]
We say that $\{a_n\}$ {\it has a polynomial growth rate} or {\it grows polynomially fast with $n$} if there exist $C>D\ge0$ and $d_2\ge d_1\ge0$ so that for any $n\ge0$, we have
\[
    D n^{d_1} \le a_n \le C n^{d_2}.
\]
If $d_1=d_2=d$, we call $d$ the {\it degree of the polynomial growth rate} and denote it by $\deg(a_n;n)$.

\begin{prop}\label{prop:GrowthNumPath}
 Let $G$ be a directed graph and $v \in \V(G)$.
 
 \begin{itemize}
     \item [(1)] If there exist two different cycles passing through $v$, then $P(v,n)$ grows exponentially fast with $n$.
     \item [(2)] If there exists $w\in \V(G)$ such that there are two different cycles passing through $w$ and a path from $v$ to $w$, then $P(v,n)$ grows exponentially fast with $n$.
     \item [(3)] If $v$ does not satisfy $(1)$ or $(2)$, then $P(v,n)$ grows polynomially fast with $n$. Moreover, if the maximum number of disjoint cycles that a path from $v$ can intersect is $d+1$, then $\deg(P(v,n);n)=d$. When the maximum number of cycles is zero, then $P(v,n)=0$ for every sufficiently large $n$, and we define the degree of the polynomial growth rate to be $-1$.
\end{itemize}

\end{prop}
\begin{proof}
See \cite[Theorem 3.6]{Park_Levy}.
\end{proof}


\begin{lem}[Uniqueness of recurrent paths]\label{lem:UniqRecurPath}
Let $G$ be a finite directed graph and $v\in \V(G)$. Suppose $P(v,n)$ grows polynomially fast with $n$ and $v$ is a recurrent vertex. Then, for every $n\ge 0$, there exists a unique recurrent path of length $n$ that starts from $v$.
\end{lem}
\begin{proof}
By Proposition \ref{prop:GrowthNumPath}, there is only one cycle passing through $v$. Hence, there is only one path that starts from $v$ and supported in the cycle.
\end{proof}

\subsubsection{Directed graphs from subdivisions} There are several directed graphs encoding the dynamics of finite subdivision rules $\R$.

\noindent A {\it directed graph $\Ecal_\R$ of edge subdivisions} is defined in such a way that
\begin{itemize}[label=-]
    \item $\V(\Ecal_\R)$ is the set of all level-$0$ edges (we use bracket $[e]$ when we need to distinguish vertices of $\Ecal_\R$ from actual edges of $S_\R$), and 
    \item every edge of $\Ecal_\R$ from $[e]$ to $[e']$ corresponds to a level-$1$ subedge $e^1$ of $e$ with $f(e^1)=e'$.
\end{itemize}

\noindent Similarly, A {\it directed graph $\Tcal_\R$ of tile subdivisions} is defined in such a way that
\begin{itemize}[label=-]
    \item $\V(\Tcal_\R)$ is the set of all level-$0$ tiles (we use bracket $[t]$ when we need to distinguish vertices of $\Tcal_\R$ from actual tiles of $S_\R$), and 
    \item every edge of $\Tcal_\R$ from $[t]$ to $[t']$ corresponds to a level-$1$ subtile $t^1$ of $t$ with $f(t^1)=t'$.
\end{itemize}

\noindent A {\it directed graph $\Bcal_\R$ of bands} is defined in such a way that
\begin{itemize}[label=-]
    \item $\V(\Bcal_\R)$ is the set of all level-$0$ bands (we use bracket $[b]$ when we need to distinguish vertices of $\Bcal_\R$ from actual bands of $S_\R$), and 
    \item every edge of $\Bcal_\R$ from $[b]$ to $[b']$ corresponds to a level-$1$ subband $b^1$ of $b$ with $f(b^1)=b'$.
\end{itemize}

Consider a level-$0$ edge $e$. Every level-$n$ subedge $e^n$ of $e$ is bijectively corresponded to a path of length $n$ in $\Ecal_\R$ starting from $[e]$. The terminal vertex of the path is the level-$0$ edge that is the type of $e^n$. There are similar 1-1 correspondences for tiles and bands as well.

\begin{defn}[Recurrent edges and bands]
For $n>0$, a level-$n$ subedge $e^n$ of a level-$0$ edge $e^0$ (resp.\@ level-$n$ subband $b^n$ of a level-$0$ band $b^0$) is a {\it recurrent subedge} of $e^0$ (resp.\@ a {\it recurrent subband} of $b^0$) if the corresponding path in $\Ecal_\R$ (resp.\@ in $\Bcal_\R$) is recurrent. For notational simplicity, we say that a level-$n$ edge (resp.\@ band) is {\it recurrent} if it is a recurrent subedge of some level-$0$ edge (resp.\@ band). A level-$0$ edge $e^0$ (resp.\@ band $b^0$) is {\it recurrent} if there is a cycle in $\Ecal_\R$ (resp.\@ $\Bcal_\R$) that passes through $[e^0]$ (resp.\@ $[b^0]$). 
\end{defn}

It follows from the definition that for $n>m\ge0$, if a level-$n$ subedge $e^n$ (resp.\@ subband $b^n$) of a level-$m$ edge $e^m$ (resp.\@ band $b^m$) is recurrent, then so is $e^m$ (resp.\@ $b^m$). Moreover, if a level-$n$ band $b^n$ is recurrent, then its sides edges are also recurrent.

\subsection{Growth rate of edge subdivisions}
\begin{defn}[Growth rate of edge subdivisions]\label{defn:GRateEdgeSubdiv}
For a level-$0$ edge $e$ of a finite subdivision rule $\R$, we define 
\begin{equation}\label{eqn:GrowthEdgeSubdiv}
    \rho(e) := \lim\limits_{n\rightarrow \infty} \left( |\R^n(e)| \right) ^{\frac{1}{n}}.
\end{equation}
We say that the edge $e$ {\it has a sub-exponential (resp.\@ exponential) growth rate of subdivisions} if $\rho(e)=1$ (resp.\@ $\rho(e)>1$). A finite subdivision rule $\R$ has {\it a sub-exponential growth rate of edge subdivisions} if every edge type has sub-exponential growth rate of subdivisions. By Proposition \ref{prop:subexpedgesubdiv}, we can substitute the word ``sub-exponential'' for ``{\it polynomial}\,''.
\end{defn}

\begin{prop}\label{prop:subexpedgesubdiv}
A finite subdivision rule $\Rcal$ has sub-exponential growth rate of edge subdivisions if and only if the cycles in $\Ecal_\R$ are disjoint. In this case, for each level-$0$ edge $e$, $\# \left\{ \textup{level-}n~\textup{subedges~of}~e \right\}$ grows polynomially fast with $n$.
\end{prop}
\begin{proof}
The proposition is straightforward from Proposition \ref{prop:GrowthNumPath} and the correspondence between paths of length $n$ in $\Ecal_\R$ and level-$n$ subedges of level-$0$ edges.
\end{proof}



Consider the restriction $f|_{S_\R^{(1)}}\from S_\R^{(1)} \righttoleftarrow$ of $f\from \R(S_\R)\to S_\R$ to the $1$-skeleton $S_\R^{(1)}$ of the level-$0$ complex. Then, the map $f|_{S_\R^{(1)}}$ is a Markov map, in a sense that each level-$1$ subedge of an edge in $S_\R$ is homeomorphically mapped onto an edge in $S_\R$. The adjacency matrix of the directed graph of edge subdivision $\Ecal_\R$ is equivalent to the incidence matrix of the Markov map $f|_{S_\R^{(1)}}$.

\begin{prop}\label{prop:entpysubexpedgesubdiv}
A finite subdivision rule $\R$ has a polynomial growth rate of edge subdivisions if and only if $h_{top}(f|_{S_\R^{(1)}})=0$, where $h_{top}$ means topological entropy.
\end{prop}
\begin{proof}
	See {\cite[Proposition 9.3]{Park_Levy}}.
\end{proof}

\subsection{Shifts and powers of finite subdivision rules}

Let $\R$ be a finite subdivision rule with the subdivision map $f\from \R(S_\R)\to S_\R$.

\begin{defn}[Shifts]\label{defn:shift}
For $k\ge 1$, the {\it $k^{th}$-shift} of $\R$, denoted by $\R_{[k]}$, is a finite subdivision rule whose level-$0$ and level-$1$ complexes are equal to $\R^k(S_\R)$ and $\R^{k+1}(S_\R)$, respectively, and the subdivision map $f_{[k]}$ is equal to $f\from \R^{k+1}(S_\R) \to \R^k(S_\R)$.
\end{defn}

\begin{defn}[Powers]\label{defn:power}
For $k\ge 1$, the {\it $k^{th}$-power} of $\R$, denoted by $\R^k$, is a finite subdivision rule whose level-$0$ and level-$1$ complexes are equal to $S_\R$ and $\R^k(S_\R)$, respectively, and the subdivision map is $f^k\from \R^k(S_\R) \to S_\R$.
\end{defn}

We note that if $\R$ has a polynomial growth rate of edge subdivisions, then its shifts $\R_{[k]}$ and powers $\R^k$ also have polynomial growth rates of edge subdivisions.

\section{Separated recurrence}\label{sec:DualRecSklton}

In this section, we introduce the notion of separated recurrence for finite subdivision rules (Definition~\ref{defn:SepRecur}) and show that the finite subdivision rules associated with rational maps satisfy separated recurrence (Theorem~\ref{thm:SepRecurr}).

\subsection*{Basic terminologies of dual graphs}
Consider the level-$n$ subdivision complex $\R^n(S_\R)$ for a finite subdivision rule $\R$ and $n\ge0$.

We define the {\it level-$n$ dual $1$-skeleton} $G_n$ as a graph in $S^2$ as follows. Each level-$n$ tile $t^n$ of $\R^n(S_\R)$ contains exactly one dual vertex $V(t^n)\in \V(G_n)$ in its interior. Each level-$n$ edge $e^n\in \Edge(\R^n(S_\R))$ has a dual edge $E(e^n)\in \Edge(G_n)$ so that (a) $E(e^n)$ joins the two vertices $V(t^n_1)$ and $V(t^n_2)$ where $t^n_1$ and $t^n_2$ are level-$n$ tiles incident to $e^n$ and (b) $E(e^n)$ intersects $\R^n(S_\R)$ transversely at only one point, which lies in the interiors of $e^n$ and $E(e^n)$. If $t^n_1=t^n_2$ is doubly incident to $e^n$, then $E(e^n)$ is a loop-edge. We call $V(t^n)$ and $E(e^n)$ the {\it dual vertex} and the {\it dual edge} of $t^n$ and $e^n$, respectively.


\begin{defn}[Links of vertices]\label{defn:Links}
	Fix $n\ge0$ and a level-$n$ vertex $v^n\in \V(\R^n(S_\R))$. The {\it link} $L^n(v^n)$ of $v^n$ is a connected graph consisting of the dual vertices $V(t^n)$ and edges $E(e^n)$ of all level-$n$ tiles $t^n$ and edges $e^n$ incident to $v$, respectively. We consider $L^n(v^n)$ as a subgraph of the level-$n$ dual skeleton $G_n$ of $\R^n(S_\R)$. 
\end{defn}

Consider a periodic Julia vertex $v$, which is a level-$0$ vertex. Since $v$ is also a level-$n$ vertex for every $n\ge 0$, the link $L^n(v)$ is well-defined and we call it the {\it level-$n$ link of $v$}. In this case, by Lemma \ref{lem:Edge_inc_PerVert}, $L^n(v)$ consists of duals of level-$n$ recurrent edges.

Fix a vertex $v\in \V(\R^n(S_\R))$. The link $L^n(v)$ is homeomorphic to a circle if and only if its every vertex has degree $2$. This is also equivalent to that every level-$n$ tile $t^n$ incident to $v$ contains exactly two edges incident to $v$. Then, either (i) $t^n$ is singly incident to $v$ or (ii) $t^n$ is a bigon doubly incident to $v$ such that both edges of $t^n$ incident to $v$ are loop-edges. From these observations, it is easy to show the following lemma. The proof is left to the reader.

\begin{lem}\label{lem:Link Circle}
    Consider $v \in \V(\R^n(S_\R))$, If every level-$n$ tile incident to $v$ is singly incident to $v$. Then, the link $L^n(v)$ is homeomorphic to a circle.
\end{lem}

The dual edge $E(e^n)$ of a level-$n$ edge $e^n$ is decomposed into two subedges by the intersection point $E^n\cap e^n$. We call them {\it dual half-edges} of $e^n$.

\begin{defn}[Separated recurrence]\label{defn:SepRecur}
    Suppose that $\R$ is a finite subdivision rule. We say that $\R$ {\it has separated recurrence at level $n$} if every non-empty intersection of a level-$n$ tile with the level-$n$ recurrent skeleton $N_n$ is either
    \begin{enumerate}
        \item the union of two adjacent dual half-edges in the link $L^n(v)$ of a periodic Julia vertex $v$, or
        \item one dual half-edge of a level-$n$ recurrent edge that is not incident to any periodic Julia vertices.
    \end{enumerate}
    We say that $\R$ has {\it separated recurrence} if it does at every level $n\ge0$.
\end{defn}

The highlight of Section \ref{sec:DualRecSklton} is the following theorem.

\begin{thm}\label{thm:SepRecurr}
	Let $\R$ be a finite subdivision rule associated with a \pcf rational map. Then, $\R$ has separated recurrence at level-$n$ for all sufficiently large $n$. Hence, for any sufficiently large $n$, the $n^{th}$-shift $\R_{[n]}$ has separated recurrence (at every level). 
\end{thm}

We will prove Theorem \ref{thm:SepRecurr} at the end of Section \ref{subsec:Dual Rec Skel of FSR of PCF}.

\begin{rem}
    The only property of \pcf rational maps used in Theorem~\ref{thm:SepRecurr} is the expansion property of the canonical conformal metric. Hence, Theorem~\ref{thm:SepRecurr} extends to finite subdivision rules associated with B\"{o}ttcher expanding maps, in the sense of~\cite{BartholdiDudko_expanding}. 
\end{rem}

\subsection{Adjacency of dual edges and corner bands}
When edges subdivide polynomially fast, recurrent subedges and recurrent subbands are unique at each level, as stated in the following lemma.

\begin{lem}\label{lem:UniqRecurBand}
	Let $\R$ be a finite subdivision rule.
	\begin{enumerate}
		\item Suppose $e$ is a level-$0$ recurrent edge that subdivides polynomially fast. Then, for each $n\ge 0$, there is a unique level-$n$ subedge of $e$ that is recurrent.
		\item Suppose $b:=\band$ is a level-$0$ recurrent band such that $e_1$ and $e_2$ are level-$0$ recurrent edges that subdivide polynomially fast. Then, for any $n\ge 0$, there exists a unique level-$n$ subband of $b$ that is recurrent.
        \item If $h_{top}(f|_{S_\R})>0$, then there exists a level-$0$ recurrent edge $e$ that has more than one level-$n$ recurrent subedges at any sufficiently large $n>0$.
	\end{enumerate}
\end{lem}
\begin{proof}
	Statements (1) and (2) are straightforward from Lemma \ref{lem:UniqRecurPath}. If $h_{top}(f|_{S_\R)})>0$, then, by Proposition \ref{prop:entpysubexpedgesubdiv}, there exists a level-$0$ edge $e$ such that the vertex $[e]$ of the directed graph $\Ecal_\R$ of edge subdivisions is contained in two different cycles. Then, for each $n>0$, there exist at least two distinct level-$n$ recurrent subedges of $e$. 
\end{proof}

For a recurrent level-$0$ edge $e$ that subdivides polynomially fast, the corresponding vertex $[e]\in \V(\Ecal_\R)$ is contained in a unique directed cycle in $\Ecal_\R$. We call the length of the cycle the {\it period of recurrence} of the edge $e$. We define the periods of recurrence for recurrent bands in a similar way.

\begin{lem}[Adjacency of dual edges]\label{lem:bandwithperiodicsides}
	Suppose $\R$ is a finite subdivision rule. Then, the following properties hold.
	\begin{itemize}
		\item [(1)] Every level-$n$ band bijectively corresponds to the adjacency of the dual edges of the level-$n$ subedges of level-$0$ edges. This includes the case when the dual edge is a loop-edge, which is adjacent to itself.
		\item [(2)] Under the bijection in (1), the level-$n$ recurrent bands (possibly not-surjectively) correspond to the adjacencies of the dual edges of recurrent level-$n$ subedges of level-$0$ edges.
		\item [(3)] If $\mathcal{R}$ has polynomial growth in the number of edge subdivisions, then the correspondence in (2) is surjective (and hence bijective) for all sufficiently large $n$.
	\end{itemize}
		
\end{lem}

\begin{proof}
	(1) is immediate and (2) follows from a fact that side edges of a level-$n$ recurrent band are level-$n$ recurrent subedges.
	
	Let us show (3). Suppose that $e_1$ and $e_2$ are level-$0$ recurrent edges. Since $e_1$ and $e_2$ subdivide polynomially fast, for each $i\in \{1,2\}$ and $n>0$, there exists a unique level-$n$ recurrent subedge $e^n_i$ of $e_i$. Let $p$ be the least common multiple of the periods of recurrence of $e_1$ and $e_2$. To show (3), it suffices to show the following: Suppose that there exists an integer $N\ge 2p$ such that the dual edges $E(e^N_1)$ and $E(e^N_2)$ of $e^N_1$ and $e^N_2$ are adjacent. Then, the level-$N$ band $b^N=(t^N;(e^N_1,s^N_1),(e^N_2,s^N_1))$, which exists by (1), is recurrent.
    
    To show that $b^N$ is recurrent, it suffices to show that $b^0$ is recurrent, where $b^0 = (t; (e_1,s_1), (e_2,s_2))$ is the level-$0$ band containing $b^N$ as a subband. Indeed, suppose that $b^0$ is recurrent. Then its side edges $e_1$ and $e_2$ are recurrent. Since $b^0$ is recurrent, it has a unique level-$N$ recurrent subband $b'^N$ (Lemma \ref{lem:UniqRecurBand}). The side edges of $b'^N$ are recurrent subedges of $e_1$ and $e_2$, and hence they must be $e_1^N$ and $e_2^N$. It follows that $b^N = b'^N$, and therefore $b^N$ is recurrent.

    For the rest of the proof, we show $b^0$ is recurrent. Recall $N\ge 2p$. For each integer $m$ with $0 \le m\le 2p$, we define $b^m=\bandm$ as the level-$m$ band that contains $b^N$ as a subband. By the definition of $p$, for any $i,j\in\{1,2\}$, the edge $e^{jp}_i$ is of type $e_i$. Below, we address the case $e_1 \neq e_2$; the case $e_1 = e_2$ is similar.    
    
    \vspace{5pt}
    \noindent{\it Case 1: $b^0$ is the only level-$0$ band having $e_1$ and $e_2$ as side edges.}

    In this case, $b^p$ is of type $b^0$ and $b^0$ is recurrent.

    \noindent{\it Case 2: There exists a level-$0$ band $b'^0=(t';(e_1,s_1'),(e_2,s_2'))$ such that $t \neq t'$.}

    Since $b^0$ and $b'^0$ are the only level-$0$ bands having $e_1$ and $e_2$ as their side edges, each of $b^p$ and $b^{2p}$ is of type either $b_0$ or $b_0'$.

    If $b^p$ or $b^{2p}$ is of type $b^0$, then $b^0$ is recurrent.

    \begin{figure}
    \centering
        \def\svgwidth{0.7\textwidth}
        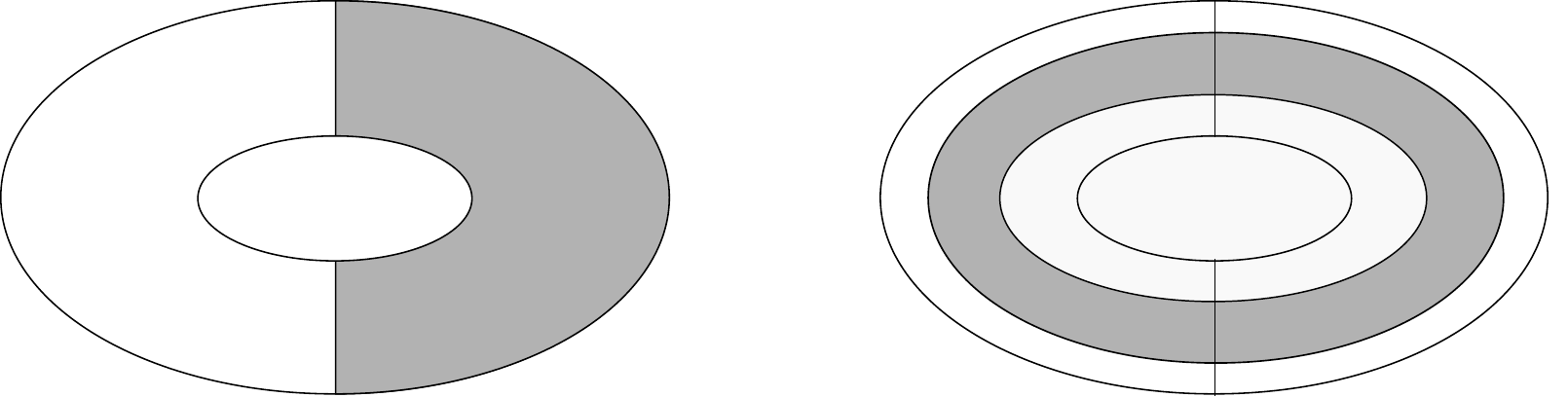
    \caption{Violation of orientation preservation of $f^p$}
    \label{fig:BandsShareSides}
    \end{figure}

    If both $b^{p}$ and $b^{2 p}$ are of type $b'^0$, then $b'^0$ is a recurrent band whose period of recurrence is $p$. Let $b'^p=(t'^p;(e^p_1,s_1'^p),(e^p_2,s_2'^p))$ be the level-$p$ recurrent subband of $b'$. Then, $b^{p}$ and $b'^{p}$ share the side edges $e^p_1$ and $e^p_2$, which are the unique level-$p$ recurrent subedges of $e_1$ and $e_2$. Since both $b^{p}$ and $b'^{p}$ are of type $b'^0$, i.e., $f^p(b^p)=f^p(b'^p)=b'^0$, the $p^{th}$-iteration $f^{p}$ cannot be orientation-preserving on $t^{p} \cup t'^{p}$, which is a contradiction. See Figure \ref{fig:BandsShareSides}. 
\end{proof}

\begin{defn}[Corner band]\label{defn:CornerBand}
	A level-$n$ band $b^n=\bandn$ is a {\it corner band} if the two side edges $e^n_1$ and $e^n_2$ are the images of the two consecutive edges of the polygon $\bft$ under the characteristic map $\phi_{t^n}\from \bft \to t^n$. Thus, $e^n_1$ and $e^n_2$ have a common endpoint $v^n$, which is called the {\it corner vertex} of the corner band $b^n$. For $m>n$, a level-$m$ subband $b^m=\bandm$ of a corner band $b^n$ is a {\it corner subband} if $b^m$ is also a corner band. See Figure \ref{fig:CornerBand}. 
\end{defn}

In Definition \ref{defn:CornerBand}, if $e^n_1=e^n_2$, then the corner vertex $v$ is incident to only one level-$n$ edge, which equals $e^n_1(=e^n_2)$. 

Recall the definition of corners at vertices in Section \ref{sec:FSR}. For a fixed level $n$ and a level-$n$ vertex $v^n$, there is a natural bijection between corners at $v^n$ and level-$n$ corner bands with the corner vertex $v^n$.
\begin{figure}
	\centering
    \def\svgwidth{0.3\textwidth}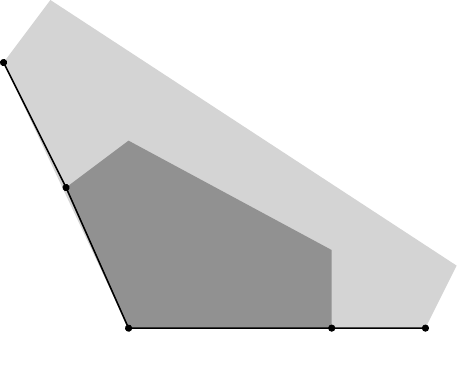
	\caption{Corner band/subband}
	\label{fig:CornerBand}
\end{figure}   
\begin{prop}\label{prop:CornerBand}
	The following properties about corner bands are satisfied.
	\begin{itemize}
		\item [(1)] For all $m>n$ and every level-$n$ corner band $b^n$, there exists at most one level-$m$ corner subband $b^m$ of $b^n$. If such a $b^m$ exists, then $b^n$ and $b^m$ share the same corner vertex.
		\item [(2)] If a level-$n$ corner band $b^n$ is recurrent, then its corner vertex is a periodic point of $f$.
		\item [(3)] Let $v\in \V(S_\R)$ be a Fatou vertex. Then, there exists at least one level-$0$ corner band $b^0$ having $v$ as the corner vertex such that $b^0$ has no level-$n$ corner subbands for any sufficiently large $n>0$.
		\item [(4)] Let $v\in \V(S_\R)$ be a periodic Julia vertex. Then, for every $n>0$, every level-$0$ corner band $b^0$ with the corner vertex $v$ has a level-$n$ corner subband $b^n$, which is unique by (1). Moreover, $b^n$ is a recurrent subband of $b^0$.
	\end{itemize}
\end{prop}
\begin{proof}
	(1) and (2) are immediate from definitions; (3) Since $v$ is a Fatou vertex, its forward image contains a periodic critical point. Note that the number of level-$n$ corner bands whose corner vertex is $v$ is equal to $\deg_{\R^n(S_\R)}(v)$. By Equation \eqref{eqn:Deg Vertex SubdivCplx}, the number $\deg_{\R^n(S_\R)}(v)$ grows exponentially fast with $n$. Hence we obtain (3); (4) It follows from Equation \eqref{eqn:Deg Vertex SubdivCplx}, we have $\deg(f^n(v))=\deg(v)$ for every $n\ge0$. Then every level-$0$ corner band with the corner vertex $v$ has a corner subband at every level. The uniqueness follows from (1), and the recurrence follows from the fact that $v$ is periodic.
\end{proof}

\subsection{Finite subdivision rules associated with \pcf rational maps}\label{subsec:Dual Rec Skel of FSR of PCF}

In this subsection, we investigate properties of finite subdivision rules that are associated with \pcf rational maps; recall Definition \ref{defn:FSR Asso to PCF}. 

\begin{lem}\label{lem:NoFatou_CornerV}
    For a finite subdivision rule $\R$ associated with a \pcf rational map, if the corner vertex $v$ of a corner band is periodic, then $v$ is a Julia vertex.
\end{lem}
\begin{proof}
    This immediately follows from the fact that the first return map around a periodic Fatou vertex is conjugate to $z^d$ for some $d>1$. 
\end{proof}
Lemma \ref{lem:NoFatou_CornerV} should be compared with Proposition \ref{prop:CornerBand}.

\begin{lem}\label{lem:RecurrentBand=CornerBand}
   Consider a finite subdivision rule $\R$ associated with a \pcf rational map $f$. Suppose that $b^l=(t^l;(e_1^l,s_1^l),(e_2^l,s_2^l))$ is a level-$l$ recurrent band. Then, $b^l$ and all its level-$n$ recurrent subbands $b^n=\bandn$ with $n>l$ are corner bands. 
\end{lem}
\begin{proof}
    By using the $l^{th}$-shift $\R_{[l]}$, we can assume that $l=0$.

    Suppose $b^0=\band$ is recurrent. Then, there exists a level-$m$ subband $b^m=\bandm$ of $b^0$ for some $m>0$ such that $b^m$ is of type $b^0$, i.e., $f^m\from b^m(\subset b^0)\to b^0$. By taking its backward iterates, we obtain a sequence of level-$km$ bands $\{b_{km}=(t^{km};(e_1^{km},s_1^{km}),(e_2^{km},s_2^{km}))\}_{k\ge0}$ such that for each $k\ge0$, the band $b_{(k+1)m}$ is a recurrent subband of $b_{km}$ and of type $b^0$.

    Consider the canonical conformal orbifold metric $\mu$ of $f$. Let $\gamma_0$ be a $\mu$-rectifiable curve in $t^0$ that joins two interior points of $e^0_1$ and $e^0_2$. Define $\gamma_{km}$ as the lifting of $\gamma_0$ through $f^{km}$ such that $\gamma_{km}\subset t^{km}$. By the backward-contraction property (Equation \eqref{eqn:Contract_Conf.metric}) with $C:=l_\mu (\gamma_0)$, there exist $D=D(C)\in (0,1)$ such that
    \[
        l_\mu(\gamma_{km})<D^{km}\, l_\mu(\gamma_0),
    \]
    and thus $l_\mu(\gamma_{km})\to 0$ as $k\to \infty$. Recall that the two endpoints of $\gamma_{km}$ lie on the edges $e^{km}_1 \subset e^0_1$ and $e^{km}_2 \subset e^0_2$, and that $\gamma_{km} \subset t^{km} \subset t^0$. By considering the lifts of $\{\gamma_{km}\}_{k\ge1}$ via the characteristic map $\phi_{t^0} \colon \bft \to t^0$, we conclude from $\lim_{k\to\infty} l_\mu(\gamma_{km}) = 0$ that $b^0$ is a corner band.
\end{proof}

Consider the case when a level-$n$ tile $t^n$ is multiply incident to a vertex $v$. This possibility is the reason that we considered the lifts of $\{\gamma_{km}\}_{k\ge 1}$ via the characteristic map at the end of the proof of Lemma~\ref{lem:RecurrentBand=CornerBand}.
Suppose no loop-edge incident to $v$ exists. Then, there exist at least four distinct pairs $\{(e^n_i,s^n_i)\}_{i=1,2,3,4}$ with $s^n_i\in {\rm Side}(e^n_i)$ and $t^n={\rm Tile}(e^n_i,s^n_i)$ such that $(t^n;(e^n_1,s^n_1),(e^n_2,s^n_2))$ and $(t^n;(e^n_3,s^n_3),(e^n_4,s^n_4))$ are corner bands with the corner vertex $v$. Lemma \ref{lem:FourEdges} implies that there exist $i\in \{1,2\}$ and $j\in \{3,4\}$ such that $(t^n;e^n_{v,i},e^n_{v,j})$ is not a recurrent band, 

\begin{lem}\label{lem:FourEdges}
    Consider a finite subdivision rule $\R$ associated with a \pcf rational map. Let $t^n$ be a level-$n$ tile. Suppose that $\{(e^n_i,s^n_i):1\le i \le 4\}$ is a set of four pairwise distinct pairs with ${\rm Tile}(e^n_i,s^n_i)=t^n$ such that all $e^n_i$'s are incident to $v$. Then it is impossible that the six bands $(t^n,(e^n_i,s^n_i),(e^n_j,s^n_j))$'s with $i\neq j$ are all recurrent. 
\end{lem}
\begin{proof}
    By Lemma \ref{lem:RecurrentBand=CornerBand}, recurrent bands are corner bands. Let $\phi_{t^n}\from \bft \to t^n$ be the characteristic map of $t^n$. Recall that $(t^n,(e^n_{v,i},s^n_{v,i}),(e^n_{v,j},s_{v,j}^n))$ is a corner band if and only if the lifts of $(e^n_{v,i},s^n_{v,i})$ and $(e^n_{v,j},s_{v,j}^n)$ via $\phi_{t^n}$ are adjacent edges of the polygon $\bft$. Then the lemma follows from the fact that four edges of a polygon cannot be pairwise adjacent.
\end{proof}

We say that an edge $e\in \Edge(S_\R)$ is {\it periodic} if $f^n\from e\to e$ is a homeomorphism for some $n>0$. 

\begin{lem}\label{lem:No_Per_JJEdge}
    Suppose $\R$ is a finite subdivision rule associated with a \pcf rational map $f$. Then, any $JJ$-edge $e\in \Edge(S_\R)$ cannot be periodic. In particular, for every loop-edge $e\in \Edge(S_R)$ incident to a Julia vertex $v$, the edge $e$ subdivides into (strictly smaller) subedges in all sufficiently high levels. 
\end{lem}
\begin{proof}
    Suppose that $f^n\from e\to e$ is a homeomorphism for some $n>0$ and a $JJ$-edge $e\in \Edge(S_\R)$.
    By Equation \ref{eqn:Contract_Conf.metric}, we have $l_\mu(e)<D^nl_\mu(e)$ for some $D\in(0,1)$, where $\mu$ is the canonical orbifold metric. This is a contradiction.

    Suppose that $e\in \Edge(S_\R)$ is a loop edge incident to a Julia vertex $v$ such that $\R^n(e)$ consists of a single edge for every $n\ge 0$. Then, for every $n\ge 0$, the $n^{th}$-forward image $f^n(e)$ is a loop-edge in $S_\R$ such that $f\from f^n(e)\to f^{n+1}(e)$ is a homeomorphism. Hence, there exists $m>0$ such that $f^m(e)$ is a periodic $JJ$-edge, which is a contradiction.
\end{proof}



\begin{defn}[Consecutive boundary edges]
    Let $t^n$ be a level-$n$ tile of a finite subdivision rule $\R$ such that $\phi_{t^n}\from \bft \to t^n$ is its characteristic map. Edges $e^n_1,e^n_2,\dots,e^n_k$ are called {\it consecutive boundary edges} of $t^n$ if they are $\phi_{t^n}$-images of consecutive boundary edges of $\bft$.
\end{defn}

\begin{lem}\label{lem:Middle Edge}
    Consider a level-$n$ subtile $t^n$ of a level-$m$ tile $t^m$, where $n>m$. Suppose that $e^m_1,e^m_2,e^m_3$ are boundary edges of $t^m$ such that $e^n_1,e^n_2,e^n_3$ are their respective subedges in the boundary of $t^n$.
    \begin{enumerate}
        \item Suppose that $e^n_1,e^n_2,e^n_3$ are consecutive boundary edges of $t^n$. Then $e^m_1,e^m_2,e^m_3$ are consecutive boundary edges of $t^m$ and $e^m_2=e^n_2$.
        \item Suppose that $t^n$ and $t^m$ are of the same type, and  $e_i^n$ and $e_i^m$ are of the same type for $i\in\{1,3\}$. If $e^m_1,e^m_2,e^m_3$ are consecutive edges of $t^m$, then $e^n_1,e^n_2,e^n_3$ are consecutive edges of $t^n$ and $e^n_2=e^m_2$.
    \end{enumerate}
\end{lem}
\begin{proof}
    The proof follows easily from consecutiveness. For (2), to see the level-$n$ subdivision complex within $t^m$, we can consider $\R^{n-m}(\bft)$, where $\bft$ is the tile type of $t^n$ and $t^m$. We omit the details.
\end{proof}

We remark that Lemma \ref{lem:Middle Edge} covers the case when $e_2$ or $e^n_2$ is a loop-edge.

\begin{lem}\label{lem:Separated Recur VV}
Let $\R$ be a finite subdivision rule associated with a \pcf rational map. For all sufficiently large $n\ge0$ and any level-$n$ tile $t^n$,
\begin{enumerate}
    \item $t^n$ cannot be multiply incident to a periodic Julia vertex, and
    \item $t^n$ cannot be incident to two distinct periodic Julia vertices. 
\end{enumerate}
Furthermore, (1) implies that the level-$n$ links $L^n(v)$ of periodic Julia vertices $v$ are homeomorphic to circles, and (2) implies those links are disjoint.
\end{lem}
\begin{proof}
    We provide a proof of (1) only; the proof of (2) follows similarly by replacing the corner band $b^n_2=(t^n;(e^n_{v,3},s^n_{v,3}),(e^n_{v,4},s^n_{v,4}))$ below with a corner band $b^n_2:=(t^n;(e^n_{w,1},s^n_{w,1}),(e^n_{w,2},s^n_{w,2}))$ whose corner vertex is a periodic Julia vertex $w\neq v$.

    Suppose that $v$ is a periodic Julia vertex such that for each $n\ge0$, there exists a level-$n$ tile $t^n$ that is multiply incident to $v$. Recall that $\deg_{\R^n(S_\R)}(v)=\deg_{S_\R}(v)$ for every $n\ge0$ (Lemma \ref{lem:DegVertices}). Hence, we may assume that $t^0\supset t^1\supset t^2\supset \cdots$. Since $t^n$'s are multiply incident to $v$, for each $n\ge0$, there exist four pairs $\{(e^n_{v,i},s^n_{v,i})\}_{i=1,2,3,4}$ such that (i) $b^n_1:=(t^n;(e^n_{v,1},s^n_{v,1}),(e^n_{v,2},s^n_{v,2}))$ and $b^n_2:=(t^n;(e^n_{v,3},s^n_{v,3}),(e^n_{v,4},s^n_{v,4}))$ are corner bands with the corner vertex $v$, which are recurrent by Proposition \ref{prop:CornerBand}, and (ii) for each $k>l$ and $i\in\{1,2\}$, the band $b^k_i$ is a recurrent subband of $b^l_i$.

    We show that the pairs $(e^k_{v,i},s^k_{v,i})$ for $i=1,2,3,4$ are pairwise distinct for every $k\ge0$. By the condition (ii) above, it suffices to show the statement for $k=0$. By definition of bands, $(e^0_{v,i},s^0_{v,i})
    \neq(e^0_{v,j},s^0_{v,j})$ if $\{i,j\}=\{1,2\}\,{\rm or}\,\{3,4\}$.
    
    For a contradiction, assume without loss of generality that \begin{equation}\label{eqn:MidEdge}
        (e^0_{v,2},s^0_{v,2})=(e^0_{v,3},s^0_{v,3}).
    \end{equation}
    Recall that the side edges of a corner band is the image of two consecutive edges of a polygon $\bft$ under the characteristic map $\phi_{t^0}\from \bft \to t^0$. It follows that $e^0_{v,1}, e^0_{v,2}(=e^0_{v,3}), e^0_{v,4}$ are three consecutive edges of $t^n$. Since $v$ is a periodic vertex, there exists $N>0$ such that $b^{N}_1$ and $b^{N}_2$ are of type $b^0_1$ and $b^0_2$, respectively. 
    By Lemma \ref{lem:Middle Edge}, we have $(e^{N}_{v,2},s^{N}_{v,2})=(e^{N}_{v,3},s^{N}_{v,3})$ and $e^0_{v,2}=e^N_{v,2}$. Since $e^N_{v,2}$ is of type $e^0_{v,2}$, it follows that $e^0_{v,2}$ is a periodic $JJ$-edge. This contradicts Lemma \ref{lem:No_Per_JJEdge}.

    Therefore, for each $n\ge 0$, $b_{i,j}^n:=(t^n;(e^n_{v,i},s^n_{v,i}),(e^n_{v,j},s^n_{v,j}))$ is a recurrent subband of $(t^0;(e^0_{v,i},s^0_{v,i}),(e^0_{v,j},s^0_{v,j}))$ for every $i\neq j\in \{1,2,3,4\}$ such that $b^n_{i,j}$'s are pairwise distinct. This contradicts Lemma \ref{lem:FourEdges}.

    By Lemma \ref{lem:Link Circle}, (1) implies that $L^n(v)$ is homeomorphic to a circle. The disjointness of links follows immediately from (2).
\end{proof}

\begin{lem}\label{lem:Separated Recur VE}
Consider a finite subdivision rule $\R$ associated with a \pcf rational map. For all sufficiently large $n\ge0$, any level-$n$ tile cannot be incident to both a periodic Julia vertex and a level-$n$ recurrent edge that is not incident to any periodic Julia vertices.
\end{lem}
\begin{proof}
    Suppose that for every $n>0$, there exists a level-$n$ tile that is incident to a periodic Julia vertex and a level-$n$ recurrent edge not incident to any periodic Julia vertices.
    Note that (a) the number of level-$n$ tiles incident to each periodic Julia vertex is uniformly bounded (Lemma \ref{lem:DegVertices}), and (b) if a level-$n$ tile $t^n$ is incident to a periodic Julia vertex $w$, then for each $0\le m<n$, the unique level-$m$ tile $t^m$ with $t^m\supset t^n$ is also incident to $w$. Therefore, we can assume that there exist a periodic Julia vertex $v$, a nested sequence of tiles $t^0\supset t^1\supset \cdots$, and recurrent edges $e^0\supset e^1 \supset \cdots$ such that (i) $t^n$ is a level-$n$ tile that is incident to $e^n$ and $v$ and (ii) $e^n$ is a level-$n$ edge that is not incident to any periodic Julia vertices.
    
    For each $n$, we can choose one edge $e'^n$ of $t^n$ that is incident to $v$ and $e'^n\neq e^n$ such that $e'^0\supset e'^1\supset \cdots$. Then, there exists a band $b^n$ within $t^n$ whose side edges are $e^n$ and $e'^n$. Since $v$ is periodic, its incident edge $e'^n$ is a recurrent subedge of $e'^0$. It follows that $b^n$ is a level-$n$ recurrent band. By Lemma \ref{lem:RecurrentBand=CornerBand}, $b^n$ is a corner band. Hence, $e'^n$ and $e^n$ are incident to the corner vertex $w$ of $b^n$. By Proposition \ref{prop:CornerBand}, the vertex $w$ is a periodic Julia vertex, and this contradicts the assumption that $e^n$ is not incident to any periodic Julia vertices.
\end{proof}

\begin{proof}[Proof of Theorem \ref{thm:SepRecurr}]
    We take $N>0$ sufficiently large such that for every $n$ with $n>N$, the Lemmas and Propositions cited in this proof work.
    
    Recall that there is a bijection between the adjacencies of duals of level-$n$ recurrent edges and level-$n$ recurrent bands (Lemma \ref{lem:bandwithperiodicsides}).
    By Lemmas \ref{lem:NoFatou_CornerV} and \ref{lem:RecurrentBand=CornerBand}, the set of level-$n$ recurrent bands is equal to the set of level-$n$ recurrent corner bands around periodic Julia vertices. Hence, every connected component of $N_n$ is either the level-$n$ link $L^n(v)$ of a periodic Julia vertex $v$ or the dual of a level-$n$ recurrent edge that is not incident to any periodic Julia vertices.

    Then the separated recurrency follows from Lemmas \ref{lem:Separated Recur VV} and \ref{lem:Separated Recur VE}.
\end{proof}

\section{Asymtotic {\it p}-conformal energy
and Ahlfors-regular conformal dimension}\label{sec:AsympConfEng}

In this section, we first discuss the theory of conformal energies (Section \ref{sec:AsympConfEngReview}) and its application in estimating conformal dimensions of Julia sets (Section \ref{sec:ConfDimJulia}). Then, we establish the monotonicity of conformal energies under the decomposition of branched coverings (Theorem \ref{thm:MonoConfEngDim}).

\subsection{Asymtotic {\it p}-conformal energies}\label{sec:AsympConfEngReview}
Let us review the theory of conformal energies introduced by D.\@ Thurston in \cite{ThuElaGraph,ThuPosChac}.

\subsubsection{Conformal graphs and energies}$ $

\subsection*{Conformal graphs.}
Let $G$ be a finite graph. We always assume that a linear structure is given to each edge, i.e., we fix a homeomorphism from each edge $e$ to $[0,1]$ so that piecewise-linear (PL) maps between finite graphs are well-defined.

For $p\in (1,\infty]$, a {\it $p$-length} $\alpha$ on $G$ is a function $\alpha\from \Edge(G)\to \Rbb_{>0}$. When $p=\infty$, we exceptionally allow $\alpha(e)=0$ for $e\in \Edge(G)$. For $p\in (1,\infty]$, a {\it $p$-conformal graph}, denoted by $(G,\alpha)$, is a graph $G$ equipped with a $p$-length $\alpha$. In particular, if $\alpha$ is an $\infty$-length, then $(G,\alpha)$ is also called a {\it length graph}. For the case $p=1$, we assign to $G$ a {\it weight function} $w\from \Edge(G)\to \Rbb_{\ge 0}$, which is also called a {\it $1$-length}. We call $(G,w)$ a {\it weighted graph} or {\it $1$-conformal graph}.

We note that, at this point, the parameter $p$ does not play any role. When considering the energies of maps between conformal graphs, different values of $p$ will matter.

\subsection*{Singular and regular values.} Let $G$ and $H$ be finite graphs. Suppose $\phi\from H \to G$ is a PL-map. A point $x\in H$ is a {\it singular point} if it is a vertex of $H$, a preimage of a vertex of $G$, a point where the linearity of $\phi$ breaks, or $\phi$ is constant on a small neighborhood of $x$. Otherwise, $x\in H$ is a {\it regular point}. The image of a singular point is a {\it singular value}, and a point $y\in G$ is a {\it regular value} if it is not a singular value. There are only finitely many singular values.

\vspace{5pt}
Let $(G,\alpha)$ be a $p$-conformal graph, and let $(H,\beta)$ be a $q$-conformal graph with $p \ge q$. For any PL map $\phi\from (H,\beta)\to(G,\alpha)$, we define the conformal energies $E^q_p(\phi)$ and  $E^q_p[\phi]$. In this paper, we only consider two cases: $p=q$ or $1=q<p$. See \cite{ThuElaGraph} for general cases.

\subsection*{Conformal energies (when $p=q$)}
When $p=q\in (1,\infty)$, we define
\[
    \Fill^p_p(\phi)(y)= \sum\limits_{x\in\phi^{-1}(y)} |\phi'(x)|^{p-1},
\]
where $y\in G$ is a regular value. Here, $\phi'$ is evaluated by considering $\alpha$ and $\beta$ as actual lengths of edges. Although $\Fill^p_p(\phi)$ is not defined on singular values, we simply say that $\Fill^p_p(\phi)\from H \to \Rbb_{\ge0}$ is a function defined on $H$.

When $p=q=1$, the functions $\alpha$ and $\beta$ be weight functions, which we denote by $w_G$ and $w_H$, respectively. Let $\phi\from (H,w_H)\to(G,w_G)$ be a PL map between two weighted graphs. The function $\Fill^1_1(\phi)\from H\to \Rbb_{\ge0}$ is defined by
\[
    \Fill^1_1(\phi)(y)=\frac{\sum_{x\in\phi^{-1}(y)} w_H(x)}{w_G(y)},
\]
where $w_H(x)$ means the weight of the edge of $H$ containing $x$, and similar for $w_G(y)$.

The {\it $p$-conformal energy} of $\phi$, denoted by $E^p_p(\phi)$, is defined by
\begin{equation}\label{eq:ConfEngFill}
    E^p_p(\phi):= \esssup_{y\in G} \left( \Fill^p_p(\phi)(y) \right)^{1/p}
\end{equation}
for $p\in [1,\infty)$, and
\[
      E^\infty_\infty(\phi):= \esssup_{x\in H} |\phi'(x)|.
\]
We often consider the infimum in the homotopy class
\[
    E^p_p[\phi]:= \inf_{\psi \sim \phi} E^p_p(\psi),
\]
which is referred to as the $p$-conformal energy of $[\phi]$.

\subsection*{Conformal energies when $1=q<p$.}  Let $(H,w)$ be a weighted graph and $(G,\alpha)$ be a $p$-conformal graph for $p>1$. Let $\phi\from H \to G$ be a PL map. We first define the {\it multiplicity function} $n_\phi\from G \to \Rbb_{\ge 0}$ by
\[
    n_\phi(y)=\sum_{x\in \phi^{-1}(y)} w(x).
\]
Then, the conformal energy $E^1_p(\phi)$ of $\phi$ is defined by
\[
    E^1_p(\phi):=\|n_\phi\|_{p^\vee} =\left(\int_G {n_{\phi}}^{p/(p-1)}\right)^{(p-1)/p},
\]
where $p^\vee:=p/(p-1)$.

We often use the infimum in the homotopy class
\[
    E^1_p[\phi]:=\inf_{\psi\sim \phi} E^1_p(\psi).
\]
Sometimes we also consider the homotopy class $[\phi\from (H,A)\to (G,B)]$ relative to $A\subset \V(H)$, where $B=\phi(A)$.

\begin{prop}[Sub-multiplicativity {\cite[Proposition A.12]{ThuElaGraph}}]\label{prop:SubMulti}
For $1 \le p\le q\le r \le \infty$, let $F$, $G$, and $H$ are $p$-, $q$-, and $r$-conformal graphs respectively. Suppose that there are PL maps $f\from F\to G$ and $g\from G\to H$. Then, we have
\[
    E^p_r[g \circ f]\le E^q_r[g] E^p_q[f].
\]
\end{prop}

\subsection*{Energy minimizers and taut maps.} Let $\phi\from H\to G$ be a PL map where $H$ is a $q$-conformal graph and $G$ is a $p$-conformal graph with $q\le p$. The map $\phi$ is an {\it energy minimizer} if $E^q_p[\phi]=E^q_p(\phi)$. An energy minimizer always exists \cite[Theorem 6]{ThuElaGraph}.

Suppose that $(H,w)$ is a weighted graph. For an interior point $y$ of an edge of $G$, we define
\[
    n_{[\phi]}(y)=\inf\limits_{\psi\sim \phi} n_\psi(y),
\]
where the infimum is taken over $\psi$ for which $y$ is a regular value. From the definition, $n_{[\phi]}$ is constant on the interior of each edge. A map $\phi\from H \to G$ is {\it taut} if $n_{[\phi]}(y)=n_\phi(y)$ for every $y$ in the interior of an edge of $G$. It is immediate from definitions that a taut map is an $E^1_p$-energy minimizer. A taut map always exists in each homotopy class \cite[Theorem 3]{ThuElaGraph}.

\begin{prop}[{\cite[Theorem 6]{ThuElaGraph}}]\label{prop:PPEFrac1P}
For every $p\in[1,\infty]$ and any continuous map $f\from G\to H$ between two $p$-conformal graphs, we have
\[
    E^p_p[f]=\sup\limits_{c\from  W \to G} \frac{E^1_p [f \circ c]}{E^1_p [c]}.
\]
The supremum is taken over all weighted graph $W$ and PL maps $c\from W\to G$.
\end{prop}

The next Proposition imply that the conformal energy is not changed by the extension of range.

\begin{prop}\label{prop:EngEqualities}
Fix $p\in[1,\infty]$. Let $G$ be a $p$-conformal graph. Let $H$ be a $p$-conformal homotopically non-trivial subgraph of $G$ such that $\iota\from H \hookrightarrow G$ is the embedding. Then, for every weighted graph $W$ and continuous map $f\from W\to H$, we have that
\begin{enumerate}
    \item \cite[Lemma 3.43]{ThuElaGraph} the map $f$ is taut if and only if $\iota \circ f$ is taut,
    \item $E^p_p[\iota]=1$, and
    \item $E^1_p[f]=E^1_p[\iota \circ f]$.
\end{enumerate}
Hence, for every $p$-conformal graph $F$ and continuous map $g\from F \to H$, we have
\begin{enumerate}[resume]
    \item $E^p_p[g]=E^p_p[\iota \circ g]$.
\end{enumerate}
\end{prop}
\begin{proof}
We refer to \cite[Lemma 3.43]{ThuElaGraph} for (1).

Let us show (2) and (3). Since a taut map always exists in the homotopy class, we may assume that $f\from W \to H$ is a taut map. It follows from (1) that $\iota \circ f$ is also taut. Then, we have
\begin{equation}\label{eqn:ineqn,E^1_p}
    E^1_p(\iota\circ f)=E^1_p[\iota \circ f]\le E^p_p[\iota] E^1_p[f]=E^p_p[\iota] E^1_p(f).
\end{equation}
We have $E^1_p(\iota\circ f)=E^1_p(f)\neq 0$ from the definition of $E^1_p$ and the assumption that $H$ is homotopically non-trivial. Thus $1\le E^p_p[\iota]$. Since $E^p_p[\iota]\le E^p_p(\iota)=1$, we obtain (2). Therefore, we have $E^p_p[\iota]=1$ and $E^1_p(\iota\circ f)=E^1_p(f)\neq 0$, and thus Equation \eqref{eqn:ineqn,E^1_p} implies (3).

Let us now show (4). By Theorem \ref{prop:PPEFrac1P}, we have
\[
    E^p_p[\iota \circ g]
=\sup\limits_{c\from  W \to F} \frac{E^1_p [\iota \circ g \circ c]}{E^1_p [c]}
=\sup\limits_{c\from  W \to F} \frac{E^1_p [g \circ c]}{E^1_p [c]}
=E^p_p[g].
\]
We use (3) for the middle equality.
\end{proof}

\subsubsection{Graph virtual endomorphisms and asymptotic conformal energies.}

\begin{defn}[Graph virtual endomorphism]
A {\it graph virtual endomorphism} is a quadruple $\Gcal = (G_0,G_1,\pi,\phi)$ where $G_0$ and $G_1$ are graphs, $\pi\from G_1 \to G_0$ is a covering map of finite degree on each connected component, and $\phi\from G_1 \to G_0$ is a continuous function. We often write  $\pi,\phi\from G_1\to G_0$ to indicate the graph virtual endomorphism. Moreover, we say that
\begin{itemize}
    \item $\Gcal$ is {\it connected} if $G_0$ and $G_1$ are connected, and
    \item $\Gcal$ is {\it recurrent} if $\phi_*\from \pi_1(G_1)\to \pi_1(G_0)$ is surjective.
\end{itemize}
\end{defn}

In this paper, all but the graph virtual endomorphisms of stable multicurves, which will be introduced in Section \ref{sec:MonoBrCov}, are assumed to be connected. All the graph virtual endomorphisms discussed in this article are recurrent.

\subsection*{Iterates of virtual endomorphisms.} Let $\pi,\phi\from G_1\to G_0$ be a graph virtual endomorphism. By pulling back the covering map $\pi\from G_1\to G_0$ through $\phi\from G_1\to G_0$, we obtain a covering space $G_2$ of $G_1$ with a covering map $\pi^2_1\from G_2\to G_1$. Lifting the continuous map $\phi\from G_1 \to G_0$ through the coverings $\pi$ and $\pi^2_1$, we have $\phi^2_1\from G_2\to G_1$. Let $\pi^1_0=\pi$ and $\phi^1_0=\phi$. By iterating the process, for every $n \ge 1$, we have $G_n$ and $\pi^n_{n-1}, \phi^n_{n-1}\from G_n \to G_{n-1}$. For every $n > m \ge 0$, we define
\[
    \begin{array}{cc}
        \pi^n_m:=\pi^{m+1}_m\circ \cdots \circ \pi^{n-1}_{n-2}\circ \pi^n_{n-1}\from G_n \to G_m & \mathrm{and} \\
        \phi^n_m:=\phi^{m+1}_m\circ \cdots \circ \phi^{n-1}_{n-2}\circ \phi^n_{n-1}\from G_n \to G_m & 
    \end{array}
\]

Assume that $G_0$ is equipped with a $p$-length $\alpha$ for $p\in[1,\infty]$. For every $n\ge1$, we equip $G_n$ with a $p$-length that is the lift of $\alpha$ through $\pi^n_0$. Then we can evaluate $E^p_p[\phi^n_0]$ for each $n\ge1$.

Note that we use superscripts for the levels of domains and subscripts for the levels of ranges.

\begin{defn}[Asymptotic conformal energies]\label{def:AsympEng}
Let $\Gcal=(G_0,G_1,\pi,\phi)$ be a graph virtual endomorphsim. For $p\in[1,\infty]$, the {\it asymptotic $p$-conformal energy} $\overline{E}^p(\Gcal)$ of $\Gcal$ is defined by
\[
    \overline{E}^p(\Gcal):= \lim\limits_{n\to \infty} \sqrt[n]{E^p_p[\phi^n_0]}.
\]
\end{defn}
The existence of the limit and the independence of the choice of $p$-lengths on $G_0$ in Definition \ref{def:AsympEng} follows from \cite[Proposition 5.6]{ThuPosChac}.

\subsection*{Homotopy invariance of conformal energies}
\begin{defn}[Homotopy morphisms]\label{defn:HomotopyMor}
Let $\Gcal=(G_0,G_1,\pi_G,\phi_G)$ and $\Hcal=(H_0,H_1,\pi_H,\phi_H)$ be two graph virtual endomorphisms. A {\it homotopy morphism from $\Gcal$ to $\Hcal$} is a pair of maps $\theta_0\from G_0\to H_0$ and $\theta_1\from G_1\to H_1$ satisfying $\theta_0 \circ \pi_G = \pi_H \circ \theta_1$ and $\theta_0 \circ \phi_G \sim \phi_H \circ \theta_1$, where $\sim$ means a homotopy. We write $\theta\from \Gcal \to \Hcal$ to indicate the homotopy morphism.
\end{defn}

Let $\theta\from \Gcal\to \Hcal$ be a homotopy morphism as in Definition \ref{defn:HomotopyMor}. By the homotopy lifting property, we have a sequence of maps $\{\theta_n\from G_n\to H_n\}_{n\ge0}$ such that for every $m>n\ge0$ we have
\[
	\theta_n\circ (\phi_G)^m_n \sim (\phi_H)^m_n \circ \theta_m.
\]

\begin{defn}[Homotopy equivalence]
Two graph virtual endomorphisms $\Gcal=(G_0,G_1,\pi_G,\phi_G)$ and $\Hcal=(H_0,H_1,\pi_H,\phi_H)$ are {\it homotopy equivalent} if there exist morphisms $\theta\from \Gcal\to \Hcal$ and $\eta\from \Hcal\to \Gcal$ such that $\theta_0 \circ \eta_0 \sim id_{H_0}$ and $\eta_0 \circ \theta_0 \sim id_{G_0}$.
\end{defn}

The conformal energies are invariant under homotopy equivalence. More precisely, the terms in the limit in Definition \ref{def:AsympEng} differ by constant multiplications and additions, which vanish as $n$ tends to the infinity.

\begin{prop}[Homotopy invariance of conformal energies {\cite[Proposition 5.7]{ThuPosChac}}]\label{prop:HomoInvConfEng}
Let $\Gcal=(G_0,G_1,\pi_G,\phi_G)$ and $\Hcal=(H_0,H_1,\pi_H,\phi_H)$ be homotopically equivalent graph virtual endomorphisms. Then, for every $p \in [1,\infty]$, we have $\overline{E}^p(\Gcal)=\overline{E}^p(\Hcal)$.
\end{prop}

\subsection{Julia sets of virtual endomorphisms}\label{sec:ConfDimJulia}
Let $\Gcal=(G_0,G_1,\pi,\phi)$ be a connected graph virtual endomorphism, i.e., $G_0$ and $G_1$ are connected. We say that $\Gcal$ is {\it forward-expanding (or backward-contracting)} if $\overline{E}^\infty(\Gcal)<1$. We define the {\it Julia set} $\Jcal_\Gcal$ of $\Gcal$ as the inverse limit of the system
\[
    \{\phi^n_m\from G_n\to G_m \mid n>m\ge0\}.
\]
If $\Gcal$ is recurrent and forward-expanding, then the Julia set $\Jcal_\Gcal$ is locally connected and has a visual metric $d_{vis.}$, which is Ahlfors-regular and defined up to quasi-symmetry \cite[Theorem F]{PilThu_ARConfDim}. More precisely, the visual metric is determined up to snowflake equivalence, and any snowflake map is a quasi-symmetry. 

\begin{defn}[Ahlfors-regular gauge of Julia set]
The quasi-symmetric class of Ahlfors-regular metrics on the Julia set $\Jcal_\Gcal$ containing visual metrics is called the {\it Ahlfors-regular gauge} of $\Jcal_\Gcal$.
\end{defn}

The covering map $\pi$ induces a self-covering map of $\Jcal_\Gcal$, which we also denote by $\pi$. If $\Gcal=(G_0,G_1,\pi_G,\phi_G)$ and $\Hcal=(H_0,H_1,\pi_H,\phi_H)$ are homotopically equivalent graph endomorphisms, then their self-coverings maps on the Julia sets $\pi_\Gcal\from \Jcal_\Gcal\righttoleftarrow$ and $\pi_\Hcal\from \Jcal_\Hcal\righttoleftarrow$ are topologically conjugate. Furthermore, the topological conjugacy preserves the Ahlfors-regular gauges of $\Jcal_\Gcal$ and $\Jcal_\Hcal$.

\begin{defn}[Ahlfors-regular conformal dimension of Julia set]
For a connected recurrent forward-expanding graph virtual endomorphism $\Gcal$, we define the {\it Ahlfors-regular conformal dimension of the Julia set} $\Jcal_\Gcal$, denoted by $\ARC(\Jcal_\Gcal)$, to be the infimum of the Hausdorff dimensions with respect to the metrics in the Ahlfors-regular gauge of $\Jcal_\Gcal$.
\end{defn}

\begin{thm}[Pilgrim-D.\@Thurston \cite{PilThu_ARConfDim}]\label{thm:ARCdim}
Suppose $\Gcal=(G_0,G_1,\pi,\phi)$ is a connected recurrent forward-expanding graph virtual endomorphism. For $p_*=\ARC(\Jcal_\Gcal)$, we have $\overline{E}^{p_*}(\Jcal_\Gcal)=1$.
\end{thm}

\subsection{Shift and power}
\begin{defn}[Shift]
Let $\Gcal=(G_0,G_1,\pi,\phi)$ be a graph virtual endomorphism. For an integer $k\ge0$, the {\it $k^{th}$-shift $\Gcal_{[k]}$ of $\Gcal$} is defined by $\Gcal_{[k]}=(G_k,G_{k+1},\pi^{k+1}_k,\phi^{k+1}_k)$.
\end{defn}

\begin{prop}\label{prop:Energy Inv under Shift}
Let $\Gcal$ be a graph virtual endomorphism and $\Gcal_{[k]}$ be its $k^{th}$-shift for some integer $k\ge0$. For every $p\in[1,\infty]$, we have 
\[
    \overline{E}^p(\Gcal)=\overline{E}^p(\Gcal_{[k]}).
\]
\end{prop}
\begin{proof}
For all $n,k >0$ we have $E^p_p[\phi^{n+k}_k]\le E^p_p[\phi^{n}_0]$ because every homotopy of $\phi^{n}_0$ can be lifted through $f^k$ to a homotopy of $\phi^{n+k}_k$ . Hence for any $n,k>0$
\[
    E^p_p[\phi^{n+k}_0]\le E^p_p[\phi^{n+k}_k]\,E^p_p[\phi^k_0] \le E^p_p[\phi^n_0]\, E^p_p[\phi^k_0].
\]
Because
\[
    \overline{E}^p(\Gcal_{[k]})=\lim\limits_{n\to \infty} \sqrt[n]{E^p_p[\phi^{n+k}_k]}\,,
\]
we can obtain the conclusion from the above inequalities.
\end{proof}

\begin{defn}[Power]
Let $\Gcal=(G_0,G_1,\pi,\phi)$ be a graph virtual endomorphism. For an integer $k\ge0$, the {\it $k^{th}$-power $\Gcal^k$ of $\Gcal$} is defined by $\Gcal^k=(G_0,G_k,\pi^k_0,\phi^k_0)$.
\end{defn}

The following proposition is immediate from the definitions.

\begin{prop}\label{prop:ConfEngPower}
Let $\Gcal$ be a graph virtual endomorphism and $\Gcal^k$ be the $k^{th}$-power of $\Gcal$ for some integer $k\ge 1$. For every $p\in[1,\infty]$, we have
\[
    \overline{E}^p(\Gcal)=\left(\overline{E}^p(\Gcal^k)\right)^{1/k}.
\]
\end{prop}

\subsection{Monotonicity of conformal energies: Graph virtual endomorphisms}

\begin{thm}[Monotonicity of conformal energies of graph virtual endomorphisms]\label{thm:MonoGraph}
Let $\Gcal=(G_0,G_1,\pi_G,\phi_G)$ and $\Hcal=(H_0,H_1,\pi_H,\phi_H)$ be graph virtual endomorphisms. Suppose that there is a homotopy morphism $\theta\from \Hcal \rightarrow \Gcal$ such that $\theta_0\from H_0\to G_0$ is an embedding (so that $\theta_n\from H_n\to G_n$ is also an embedding for every $n \ge 0$). Then, for every $p\in [1,\infty]$, we have
\[
    \overline{E}^p(\Hcal) \le \overline{E}^p(\Gcal).
\]
\end{thm}

\begin{proof}
Recall that $\overline{E}^p$ is independent of the choice of $p$-length. Assign a $p$-length $\alpha$ to $G_0$. Then, for each $n\ge1$, pull $\alpha$ back to $G_n$ via $(\pi_G)^n_0$, and then to $H_n$ via $\theta_n$.

Since $\theta_n\from H_n \to G_n$ is an embedding, $E^p_p[\theta_n]=1$ for every $n\in\Nbb$ by Proposition \ref{prop:EngEqualities}. It follows that
\[
   E^p_p[(\phi_G)^n_0 \circ \theta_n] \le E^p_p[(\phi_G)^n_0] E^p_p[\theta_n]=E^p_p[(\phi_G)^n_0].
\]
On the other hand, it also follows from Proposition \ref{prop:EngEqualities} that
\[
    E^p_p[(\phi_H)^n_0]=E^p_p[\theta_0 \circ (\phi_H)^n_0]=E^p_p[ (\phi_G)^n_0\circ \theta_n].
\]
Therefore, for every $n\ge 0$, we have $E^p_p[(\phi_H)^n_0] \le E^p_p[(\phi_G)^n_0]$.
\end{proof}

\subsection{Virtual endomorphisms of \pcf branched coverings}

Suppose that $\TmapA$ is a \pcf hyperbolic-type branched covering. Let $G_0$ be a finite graph that is a {\it spine} of $(S^2,A)$, i.e., $G_0 \subset S^2 \setminus A$ is a deformation retract of $S^2 \setminus A$. Define $G_1:=f^{-1}(G_0)$. Then, $G_1$ is a spine of $(S^2,f^{-1}(A))$. Since $f(A) \subset A$, we have $A \subset f^{-1}(A)$. The inclusion $S^2\setminus f^{-1}(A) \hookrightarrow S^2\setminus A$ defines a homotopy class of continuous maps $[\phi\from G_1\to G_0]$, where $\phi$ is a representative. Then, the quadruple $\Gcal_f:=(G_0,G_1,f,\phi)$ is a {\it virtual endomorphism of $f$}.

There are different choices of the spine $G_0$ and the representative $\phi$, but all the choices give rise to the same graph virtual endomorphism up to homotopy equivalence.

For $\Gcal_f$ given as above, the $n^{th}$-iterate $G_n$ is easily obtained as the $n^{th}$-preimage $f^{-n}(G_0)$. The homotopy class $[\phi^n_0]$ is induced from the inclusion $S^2\setminus f^{-n}(A) \hookrightarrow S^2\setminus A$.


The next proposition implies that, for any hyperbolic-type \pcf branched covering $\TmapA$, the asymptotic conformal energy is invariant under the choice of the set of marked points $A$. Hence, we simply denote by $\overline{E}^p(f)$ the asymptotic $p$-conformal energy of $\TmapA$.

\begin{prop}[Independence of the choice of marked points]
Suppose $\TmapA$ is a hyperbolic-type \pcf branched covering, with $f(A) \cup P_f \subset A$, and $\Gcal$ is its graph virtual endomorphism. Let $\Hcal$ be a graph virtual endomorphism of $f\from (S^2,P_f)\righttoleftarrow$. Then $\overline{E}^p(\Gcal)=\overline{E}^p(\Hcal)$ for every $p\in[1,\infty]$.
\end{prop}
\begin{proof}
Since $P_f \subset A$, we can take $\Gcal$ and $\Hcal$ in such a way that, for every $n\ge0$, $H_n\subset G_n$. It follows from Theorem \ref{thm:MonoGraph} that $\overline{E}^p(\Gcal) \ge \overline{E}^p(\Hcal)$.

Since $A$ is a set of Fatou points, there exists an integer $k\ge0$ such that $P_f \subset A \subset f^{-k}(P_f)$. Then $\Gcal$ and $\Hcal$ can be chosen so that there is an embedding of $G_0$ into $H_k$. This yields a morphism from $\Gcal$ to the $k^{th}$-shift $\Hcal_{[k]}$ of $\Hcal$. Then by Proposition \ref{prop:Energy Inv under Shift} and Theorem \ref{thm:MonoGraph} we have
\[
\overline{E}^p(\Gcal) \le \overline{E}^p(\Hcal_{[k]}) = \overline{E}^p(\Hcal). \qedhere
\]
\end{proof}



If $f$ is a \pcf hyperbolic rational map, then the self-covering $f\from \Jcal_f \righttoleftarrow$ is topologically conjugate to the self-covering $f\from \Jcal_{\Gcal_f}\righttoleftarrow$. Moreover, the Ahlfors-regular gauge of $\Jcal_{\Gcal_f}$ contains the pull-back of the spherical metric on $\Jcal_f$ through the topological conjugacy \cite[Sections 3--4]{HP_CXC}, \cite[Section 3]{PilThu_ARConfDim}.

\subsection{Monotonicity of conformal energies: Branched coverings}\label{sec:MonoBrCov}$ $

\vspace{5pt}
\paragraph{\bf{Decomposition along multicurves}} Suppose that $\TmapA$ is a \pcf branched covering. Let $\Gamma$ be a multicurve of $(S^2,A)$ that is completely $f$-invariant up to isotopy, i.e., the collections of the homotopy classes of essential components appearing in $f^{-1}(\Gamma)$ and in $\Gamma$ are the same. We split the sphere into small spheres by pinching it along $\Gamma$. More precisely, we remove $\Gamma$ from $S^2$, compactify each connected component $X$ with boundary circles, and then we 
 collapse each boundary circle to a point. We refer to the compactified spheres as the {\it small spheres} of the $\Gamma$-decomposition.

The dynamics $\TmapA$ descends to the dynamics on the union of small spheres as follows: Let $X$ be a connected component of $S^2 \setminus \Gamma$ and $\hat{X}$ be the sphere that is the compactification of $X$. There exists a unique connected component $Y'$ of $S^2\setminus f^{-1}(\Gamma)$ such that $X$ and $Y'$ are homotopic relative to $A$. Let $Y$ be a connected component of $S^2 \setminus \Gamma$ such that $Y=f(Y')$ and $\hat{Y}$ be the sphere that is the compactification of $Y$. Using the map $f|_{Y'}\from Y'\to Y$ and the homotopy between $X$ and $Y'$, we have a branched covering $\hat{X}\to \hat{Y}$. In this way, we have an induced map among the small spheres. 

Let $F$ be the induced map among the small spheres of the decomposition by $\Gamma$. A small sphere $S$ is {\it periodic} if $F^k(S)=S$ for some $k>0$. Let $\{S^2(i)~|~i=1,2,\dots,n\}$ be the set of periodic small spheres. The {\it first return time $\tau_i$} of a periodic small sphere $S^2(i)$ is the smallest positive integer $k$ satisfying $F^k(S^2(i))=S^2(i)$. For each $i$, we define the {\it first return map $f_i$} of $S^2(i)$ by
\[
    f_i:=F^{\tau_i}|_{S^2(i)}\from S^2(i)\righttoleftarrow.
\]

To each periodic small sphere $S^2(i)$, we assign a set of marked points $A(i)$ as follows: Let $U(i)$ be the connected component of $S^2\setminus \Gamma$ such that $\overline{U(i)}$ is the compactification of $U(i)$ with boundary circles and there is a quotient map $q_i\from \overline{U(i)}\to S^2(i)=\overline{U(i)}/\sim$ which collapses every boundary circle to a point. The set $A(i)\subset S^2(i)$ is defined as the union of $q_i(A \cap U(i))$ and $q_i(\partial \overline{U(i)})$. Thus every point in $A(i)$ comes from $A$ or $\Gamma$. It is easy to show that $f_i$ is post-critically finite, $f_i(A(i)) \subset A(i)$, and $P_{f_i}\subset A(i)$. Hence, $f_i\from (S^2(i),A(i))\righttoleftarrow$ is also a \pcf branched covering, and we call it a {\it small branched covering of the $\Gamma$-decomposition of $f$}.

If a small branched covering $f_i \from (S^2(i),A(i))\righttoleftarrow$ has a Thurston obstruction $\Gamma_i$, then $\Gamma_i$ lifts to a Thurston obstruction of $\TmapA$. It follows that if $f\from (\hCbb,A)\righttoleftarrow$ is a \pcf rational map, then its small branched coverings are combinatorially equivalent to rational maps. In this case, after taking a homotopy and giving an appropriate complex structure on $S^2(i)$, we may assume that the small branched coverings are \pcf rational maps $f_i\from (\hCbb(i),A(i))\righttoleftarrow$, which we refer to as {\it small rational maps} of the $\Gamma$-decomposition of $f$.

We refer the reader to \cite{Pilgrim_Combination} for details of the decomposition of branched coverings.

\begin{lem}\label{lem:SmallMapHyp}
Let $\TmapA$ be a \pcf branched covering of hyperbolic-type. Suppose $\Gamma$ is a multicurve which is completely $f$-invariant up to isotopy such that $\Gamma$ does not contain any Levy cycles. Then, every small branched covering $f_i\from (S^2(i),A(i))\righttoleftarrow$ of the $\Gamma$-decomposition of $f$ is also of hyperbolic-type.
\end{lem}
\begin{proof}
Let $f_i\from (S^2(i),A(i))\righttoleftarrow$ be a small branched covering of $(f,\Gamma)$. Let $p\in A(i)$. If $p$ comes from $U(i)\cap A$, then $p$ is a Fatou point because every point in $A$ is a Fatou point of $\TmapA$. Assume $p$ is the quotient of an element of $\Gamma$. If $p$ is periodic under $f_i$, then the periodic cycle containing $p$ is a critical cycle because $\Gamma$ does not contain any Levy cycles. Hence every element of $A(i)$ is a Fatou point of $f_i\from (S^2(i),A(i))\righttoleftarrow$.
\end{proof}


\begin{prop}\label{prop:SmallBrEmb}
Let $f\from (S^2,A)\righttoleftarrow$ be a \pcf hyperbolic-type branched covering. Suppose $\Gamma$ is a multicurve of $(S^2,A)$ that is completely $f$-invariant up to isotopy such that $\Gamma$ does not contain any Levy cycles. Let $f_i\from (S^2,A(i))\righttoleftarrow$ be a small branched covering of the $\Gamma$-decomposition of $f$ with the first return time $\tau_i$. Then, for any $p\in [1,\infty]$, we have
\[
	\left(\overline{E}^p(f_i)\right)^{1/\tau_i}\le \overline{E}^p(f).
\]
\end{prop}
\begin{proof}
	Let us take a spine $H_0$ of $(S^2(i),A(i))$. For the corresponding component $U(i)$ of $S^2\setminus \Gamma$, $H_0$ lifts to a spine of $U(i)$, which we also denote by $H_0$. We extend $H_0$ to a spine $G_0$ of $(S^2,A)$ so that there is an embedding $\iota_0\from H_0 \to G_0$. Then we have graph virtual endomorphisms $\Gcal=(G_0,G_1,f,\phi)$ of $\TmapA$ and $\Hcal=(H_0,H_1,f_i,\phi_i)$ of $f_i\from (S^2(i),A(i))\righttoleftarrow$ such that $\iota\from \Hcal \to \Gcal^{\tau_i}$ is a morphism with $\iota_0\from H_0 \to G_0$ being an embedding. Then the inequality follows from Proposition \ref{prop:ConfEngPower} and Theorem \ref{thm:MonoGraph}.
\end{proof}

\subsection*{Virtual endomorphism of stable multicurves.} Let $\TmapA$ be a hyperbolic-type \pcf branched covering and $\Gamma$ be an $f$-stable possibly peripheral multicurve of $(S^2,A)$. i.e., we allow a component of $\Gamma$ to be peripheral to $A$. Denote $\Gamma_0=\Gamma$ and let $\Gamma_1$ be the collection of connected components of $f^{-1}(\Gamma)$ that are essential or peripheral relative to $A$. The restriction $f_\Gamma:=f|_{\Gamma_1}\from \Gamma_1\to \Gamma_0$ is a covering map on each connected component. Since $\Gamma$ is $f$-stable, every connected component $\gamma'$ of $\Gamma_1$ is isotopic relative to $A$ to a connected component $\gamma$ of $\Gamma_0$. Define $\phi_\Gamma\from \Gamma_1\to \Gamma_0$ in such a way that the $\gamma'$ is mapped to $\gamma$ through the isotopy. Then the quadruple $\Gcal_\Gamma=(\Gamma_0,\Gamma_1,f_\Gamma,\phi_\Gamma)$ is the graph virtual endomorphism of $\Gamma$. This is the only graph virtual endomorphism in this article whose graphs are disconnected.

Recall the $p$-linear transformation $f_{p,\Gamma}$ and its leading eigenvalue $\lambda_{p,\Gamma}$ in Section \ref{sec:Obstructions}.

\begin{prop}\label{prop:EngOfMultiCurve}
Let $\TmapA$ be a \pcf hyperbolic-type branched covering. Let $\Gamma$ be an $f$-stable possibly peripheral multicurve and $\Gcal_\Gamma=(\Gamma_0,\Gamma_1,f_\Gamma,\phi_\Gamma)$ be its graph virtual endomorphism. Then, for every $p \in [1,\infty]$, we have
\[
    \overline{E}^p(\Gcal_\Gamma)=(\lambda_p(\Gamma))^{1/p},
\]
where $\lambda_p(\Gamma)$ is the maximal eigenvalue of the $p$-transformation of $\Gamma$.
\end{prop}
\begin{proof}
Since $\overline{E}^p$ is independent of the length of edges, we may assume that all the circles in $\Gamma$ have the same length. Assume that $\phi_\Gamma$ has constant derivative on every circle.

Let $\Gamma=\{\gamma_1,\gamma_2,\dots,\gamma_n\}$ and consider it as an ordered basis of $\Rbb^\Gamma$. Let $A_{p,\Gamma}$ be the matrix representative of the $p$-transformation $f_{p,\Gamma}$ with respect to this ordered basis. For every $k>0$ and $x\in \gamma_i$, the $\Fill^p_p((\phi_\Gamma)^k_0\from \Gamma_k\to \Gamma_0)(x)$ is equal to the $i^{th}$-row sum of $(A_{p,\Gamma})^k$. It follows that $E^p_p[(\phi_\Gamma)^k_0]$ is the $p^{th}$-root of the maximal row sum of $(A_{p,\Gamma})^k$. Then the conclusion follows from the fact that the maximal row sum of $(A_{p,\Gamma})^k$ grows exponentially fast with $k$ with base $\lambda_p(\Gamma)$. 
\end{proof}

We note that the definitions of decomposition along $\Gamma$ and small branched coverings still work even if $\Gamma$ has a peripheral component. Recall that if $\Gamma$ does not have any Levy cycles, then $\lambda_p(\Gamma)$ strictly decreases about $p$ so that the critical exponent $Q(\Gamma)$ is determined by $\lambda_{Q(\Gamma)}(\Gamma)=1$.

\begin{thm}[Monotonicity of conformal energies]\label{thm:MonoConfEngDim}
Let $\TmapA$ be a \pcf hyperbolic-type branched covering without any Levy cycles. Let $\Gamma$ be a multicurve which possibly contains a simple closed curve peripheral to $A$. Suppose that $\Gamma$ is completely $f$-invariant up to isotopy such that
\[
    \{f_i\from (S^2(i),A(i))\righttoleftarrow:i=1,2,\dots,n\}
\]
is the set of small branched coverings of the $\Gamma$-decomposition with the first return times $\{\tau_i\}_{i=1,2,\dots,n}$. Then, for every $p\in[0,\infty]$, we have
\[
    \overline{E}^p(f) \ge \max\left\{ \left(\overline{E}^p(f_i)\right)^{1/\tau_i}\mid 1\le i \le n\right\} ~{\rm and}~(\lambda_p(\Gamma))^{1/p}.
\]
Moreover, we have $\ARC(\Jcal_f) \ge Q(\Gamma)$. 
\end{thm}
\begin{proof}
The inequality $\overline{E}^p(f) \ge \left( \overline{E}^p(f_i)\right)^{1/\tau_i}$ follows from Proposition \ref{prop:SmallBrEmb}. Let us define a spine $G_0$ of $(S^2,A)$ by extending $\Gamma$. Even when $\Gamma$ contains peripheral loops, since $\TmapA$ is of hyperbolic-type, $\Gamma$ can be extended to a spine. Then the inequality $\overline{E}^p(f) \ge \left( \lambda_p(\Gamma) \right)^{1/p}$ follows from Theorem \ref{thm:MonoGraph} and Proposition \ref{prop:EngOfMultiCurve}. Since $\Gamma$ does not contain any Levy cycles, $\lambda_p(\Gamma)$ is strictly decreasing. Then the last inequality follows from Theorem \ref{thm:ARCdim}. 
\end{proof}

\section{Extensions of invariant graphs and preorders on edge sets}\label{Sec:GraphExt and Preorder}

Suppose that $f$ is a crochet map, i.e., there is a connected $f$-invariant graph $G$ with $f(\V(G))\cup P_f \subset \V(G)$ and $h_{top}(f|_G)=0$. For the purpose of this article, we want $G$ to satisfy four additional properties (P1)--(P4), which are defined in Section \ref{sec:P1-P4}. To do this, we note that in \cite{DHS_Decomp}, the graph $G$ is constructed by using three methods (E1)--(E3) of extending $f$-invariant graphs, which will be discussed in Sections \ref{subsec:GraphExtensions} and \ref{subsec:EntZeroGraph_Crochet}. Therefore, in Theorem \ref{thm:TotalPreorderGraphs}, we show that the properties (P1)--(P4) are preserved by the three extension methods (E1)--(E3).

\subsection*{The vertex set $\V(G)$} The properties (P1)--(P4) depend not only on the underlying topology of a graph $G$ but also on its combinatorial structure, meaning its edges and vertices. Therefore we should carefully consider the vertex set $\V(G)$; there may be many vertices of degree $2$. For each point $x\in G$, we define the {\it degree} of $x$ in $G$, denoted by $\deg_G(x)$, to be the number of connected components in $U\setminus \{x\}$ for a small neighborhood $U$ of $x$ in $G$. For each point $x\in G$, we have $\deg_G(x)\ge0$, which is zero when the point $x$ itself is a connected component of $G$. We always assume that every point $x\in G$ with $\deg_G(x)\neq 2$ is contained in $\V(G)$.

We will only consider graphs $G$ with $P_f\subset \V(G)$, where a connected component of $G$ may be a single vertex.

\subsection{Preorder extension}\label{sec:PreorderExt}
The property (P4) is about a preorder extension on the edge set $\Edge(G)$, which is essential in our proof of Theorem \ref{theorem:main}. In this subsection, we prove Theorem \ref{thm:PreorderExt}, which is a version of extension of preorder relative to an equivalence relation.

A {\it binary relation} of a set $X$ is a subset of $X\times X$. For a binary relation $R$, we use a notation $xRy$ for an element $(x,y)\in R\subset X\times X$.

\begin{defn}[Preorder]\label{defn:preorder}
For a set $X$, a {\it preorder} is a binary relation $\preceq$ satisfying the following.
\begin{enumerate}[leftmargin=20pt]
    \item[] (Reflexivity) For every $x\in X$, we have $x\preceq x$.
    \item[] (Transitivity) For $x,y,z\in X$, if $x\preceq y$ and $y\preceq z$, then $x\preceq z$.
\end{enumerate}
If $x\preceq y$ and $y \preceq x$, we say that $x$ and $y$ are {\it equivalent} with respect to $\preceq$ and write $x \simeq y$ . If $x$ and $y$ are not equivalent, we write $x \nsimeq y$. We write $x\precneq y$ if $x\preceq y$ but not $y \preceq x$.
A preorder $\preceq$ is a {\it total preorder} if it additionally satisfies the following.
\begin{enumerate}[leftmargin=20pt]
    \item[] (Strong connectivity) For every $x,y\in X$, either $x\preceq y$ or, possibly non-exclusively, $y \preceq x$.
\end{enumerate}
\end{defn}

When we adorn $\preceq$ with ${~}^*$ and $\bar{~}$ to denote preorders relevant to $\preceq$, we apply the same adornments to the symbols $\precneq$ and $\simeq$. 

\begin{defn}[Preorder extension]\label{defn:PreorderExt}
    Let $X$ be a set. A preorder $\preceq^*$ on $X$ is an {\it extension} of a preorder $\preceq$ on $X$ if
    \begin{enumerate}
        \item $x\preceq y$ implies $x \preceq^* y$, and
        \item $x \precneq y$ implies $x \precneq^* y$.
    \end{enumerate}
    Suppose that $\sim$ is an equivalence relation on $X$. We say that an extension $\preceq^*$ of $\preceq$ is a {\it (preorder) extension of $\preceq$ relative to $\sim$} if $x\sim y$ implies $x\simeq^*y$.
\end{defn}

Let $X$ be a set. As a corollary of the linear extension theorem \cite{Szpilrajn_PartialOrderExt}, every preorder $\preceq$ on $X$ can be extended to a total preorder $\preceq^*$, as also proven in \cite[Lemma 3]{Hansson_PreOrderExt}.

In our setting, however, an equivalence relation $\sim$ is additionally given on $X$, and a total preorder extension of $\preceq$ relative to $\sim$ may not exist. In Theorem~\ref{thm:PreorderExt}, we characterize the conditions under which a preorder can be extended relative to an equivalence relation.

\begin{defn}[Virtually ordered sequence/cycle]\label{defn:VirtlyOrdSeqCyc}
Let $X$ be a finite set, $\preceq$ a preorder on $X$, and $\sim$ an equivalence relation on $X$.

\begin{itemize}
    \item An ordered finite sequence $\{ x_i\}_{i=1,2,\dots k}$ in $X$ with $k\ge2$ is called a {\it virtually ordered sequence (resp.\@ cycle) with respect to $(\preceq,\sim)$} if there exists a sequence $\{x_i'\}_{i=1,2,\dots,k}$ in $X$ such that $x_i'\sim x_i$ for every $i\in \{1,2,\dots,k\}$ and $x_i' \preceq x_{i+1}$ for every $i \in \{1,2,\dots,k-1\}$ (resp.\@ $i\in \{1,2,\dots, k \mod k\}$). The elements $x_i$ and $x_i'$ may be equal.
    \item We say that the virtually ordered cycle $\{x_i\}_{i=1,2,\dots,k}$ {\it virtually contains a pair of strict order $x \precneq y$} if there exists $x,y\in X$ with $x\sim x_i$, $y \sim x_j$ for some $i$ and $j$, possibly $i=j$, and $x \precneq y$.
\end{itemize}
\end{defn}

\begin{defn}[Induced preorders on quotient sets]\label{defn:PreorderQuotient}
    Let $X$ be a set. Suppose that $\sim$ is an equivalence relation on $X$ and $\overline{X}:=X/\sim$ is the associated quotient set, whose equivalence classes are written as $[x]$ for $x\in X$.
    \begin{enumerate}
        \item (Quotient) For a preorder $\preceq$ on $X$, we define a {\it quotient} preorder $\bar{\preceq}$ on $\overline{X}$ as the transitive closure of the relation $R$ given by $[x] R [y]$ if and only if $x' \preceq y'$ for some $x'\sim x$ and $y'\sim y$. If $X$ is finite, then one can easily show that $[x]\bar{\preceq}[y]$ if and only if there is a virtually ordered sequence $x_1,x_2,\dots,x_k$ with $x\sim x_1$ and $y\sim x_k$.
        \item (Lift) For a preorder $\bar{\preceq}$ on $\overline{X}$, we define its {\it lift} as the preorder $\preceq$ on $X$ defined by $x \preceq y$ if and only if $[x] \bar{\preceq} [y]$, which can be easily shown to be a preorder. 
    \end{enumerate}
\end{defn}

Consider a finite set $X$ with a preorder $\preceq$ and an equivalence relation~$\sim$. By using Definition \ref{defn:PreorderQuotient}, we can consider the quotient preorder $\bar{\preceq}$ on $\overline{X}:=X/\sim$, and then we can lift $\bar{\preceq}$ to a preorder $\preceq^*$ on $X$. It is easy to show that the preorder $\preceq^*$ is the smallest transitive (binary) relation containing both $\preceq$ and $\sim$, as subsets of $X \times X$. However, $\preceq^*$ may not be a preorder extension of $\preceq$ when it violates the condition (2) in Definition \ref{defn:PreorderExt}.

\begin{thm}[Preorder extension relative to an equivalence relation] \label{thm:PreorderExt}
    Let $X$ be a finite set with a preorder $\preceq$. Let $\sim$ be an equivalence relation on $X$ and $\overline{X}:=X/\sim$. The following are equivalent.
    \begin{enumerate}
        \item Every virtually ordered cycle with respect to $(\preceq,\sim)$ does not virtually contain any pairs of strict order.
        \item For the quotient preorder $\bar{\preceq}$ on $\overline{X}$ (Definition \ref{defn:PreorderQuotient}) and all $x,y\in X$, $x \precneq y$ implies $[x] \bar{\precneq} [y]$.
        \item There exists a total preorder extension $\preceq^*$ of $\preceq$ on $X$ such that $x \sim y$ implies $x\simeq^*y$.
    \end{enumerate}
\end{thm}
\begin{proof}
    $\sim$(2)$\Rightarrow$$\sim$(1) Consider $x,y\in X$ with $x\precneq y$. By definition we have $[x]\bar{\preceq}[y]$. Suppose $[x]\bar{\simeq}[y]$. Then, there is a virtually ordered cycle $x_1,x_2,\dots,x_k$ with $x_1\sim x$ and $x_l \sim y$ for some $l \in \{2,\dots,k\}$, which virtually contains a pair of strict order $x \precneq y$. This contradicts (1).

    (2)$\Rightarrow$(3) We extend the preorder $\bar{\preceq}$ on $\overline{X}$ to a total preorder $\bar{\preceq}^*$ by applying \cite[Lemma 3]{Hansson_PreOrderExt}, which is a corollary of \cite{Szpilrajn_PartialOrderExt}. We obtain the desired preorder extension $\preceq^*$ as the lift of $\bar{\preceq}^*$ (Definition \ref{defn:PreorderQuotient}).

    $\sim$(1)$\Rightarrow$$\sim$(3) Suppose that $x_1,x_2,\dots,x_k$ is a virtually ordered cycle that virtually contains a pair of strict order $x\precneq y$. For a contradiction, suppose that there exists a preorder extension $\preceq^*$ of $\preceq$ such that $x\sim y$ implies $x\simeq^*y$. Then, we obtain $x_1\simeq^*x_2\simeq^*\dots \simeq^*x_k$ such that $x\simeq^*y$. This contradicts the condition (2) in Definition \ref{defn:PreorderExt}. 
\end{proof}

\subsection{Properties (P1)--(P4)}\label{sec:P1-P4}
Let $f$ be a \pcf rational map. Suppose that $G$ is a finite $f$-invariant, possibly disconnected, graph with $P_f \cup f(\V(G)) \subset \V(G)$. We allow a single vertex as a connected component of $G$.

We adopt the terminologies used for finite subdivision rules, whose $1$-skeletons of the subdivision complexes are connected graphs. For example, Fatou and Julia vertices, and the sets of them $\V_F(G)$ and $\V_J(G)$, together with $FF$-, $FJ$-, and $JJ$-edges, and the sets $\Edge_{FF}(G)$, $\Edge_{FJ}(G)$, and $\Edge_{JJ}(G)$, are defined in the same way for any $f$-invariant graph $G$, even if it is disconnected. We also define the {\it level-$n$ subdivision graph} $G^n$ of $G$ by $\V(G^n)=f^{-n}(\V(G))\cap G$, and define level-$n$ subedges, edge types, and recurrent edges accordingly.

The properties (P1) and (P2) are defined as below.
\begin{enumerate}
    \item[(P1)] Every edge is an $FJ$-edge.
    \item[(P2)] For every $FJ$-edge $e$ with the Julia endpoint $v$, we have $v\notin f^n({\rm int}(e))$ for all $n\ge0$.
\end{enumerate}

\begin{lem}\label{lem:P1P2RecFF Edge}
    Consider an $f$-invariant graph $G$ with $P_f \cup f(\V(G))\subset \V(G)$ that satisfies (P1) and (P2). Suppose $e^0$ is a level-$0$ recurrent edge with the Julia vertex $v^0$. Then the following hold.
    \begin{enumerate}
        \item $v^0$ is a periodic point of $f$.
        \item For all $n>0$, the edge $e^0$ has a unique level-$n$ recurrent edge $e^n$, which is incident to $v^0$.
        \item $h_{top}(f|_G)=0$.
    \end{enumerate} 
\end{lem}
\begin{proof}
    Fix $n>0$. Suppose that $e^n$ is a level-$n$ recurrent subedge of $e^0$ whose Julia vertex is $v^n$. There exists a level-$m$ recurrent subedge $e^m$ of $e$ with $m>n$ such that $e^m\subset e^n$ and $e^m$ is of type $e$.

    We first show that $v^0=v^n$, which implies (2). Suppose $v^n \in {\rm int}(e)$. Since $e^m\subset e^n$, the Julia vertex $v^m$ of $e^m$ lies in ${\rm int}(e)$. It follows that $v=f^m(v^m)\in f^m({\rm int}(e))$, and this contradicts (P2). Since $v^n$ is a Julia vertex, and $e$ is an $FJ$-edge, it follows that $v^n=v^0$.

    The previous paragraph for level-$n$ also applies for level-$m$. Hence, $v^m=v^0$. Since $e^m$ is of type $e^0$, we have $f^m\from v^m(=v^0) \mapsto v^0$, which proves (1).

    By Proposition \ref{prop:entpysubexpedgesubdiv} and Lemma \ref{lem:UniqRecurBand}, (2) implies (3).

\end{proof}

\begin{defn}[Preorder and equivalence relation on $\Edge(G)$]\label{defn:PreorderEqrel}
Suppose $f$ is a \pcf rational map, and $G$ is an $f$-invariant, possibly disconnected, graph satisfying $P_f\cup f(\V(G))\subset \V(G)$ and (P1).
\begin{itemize}
    \item A preorder $\preceq$ on $\Edge(G)$ is defined in such a way that $e_1 \preceq e_2$ if and only if $e_2\subset f^n(e_1)$ for some $n\ge0$. Equivalently, $e_1 \preceq e_2$ if and only if there exists a directed path from $[e_1]$ to $[e_2]$ in the directed graph $\Ecal_\R$ of edge subdivisions.
    \item An equivalence relation $\sim$ on $\Edge(G)$ is defined by declaring $e_1 \sim e_2$ if and only if $e_1$ and $e_2$ have a common Julia vertex. 
\end{itemize}    
\end{defn}

The transitivity of $\sim$ follows from the absence of $JJ$-edges, which in turn is a consequence of (P1). Even if $JJ$-edges are present, one could instead consider the transitive closure of $\sim$, but we do not pursue it in this article.

Assume that $G$ satisfies (P1). Let $\preceq$ and $\sim$ be the preorder and equivalence relation on $\Edge(G)$ defined in Definition~\ref{defn:PreorderEqrel}. We now define the properties (P3) and (P4) for $f$-invariant graphs $G$.
\begin{enumerate}
     \item[(P3)]  Suppose $e_1,e_2,\dots,e_k$ is a virtually ordered cycle with respect to $(\preceq,\sim)$ on $\Edge(G)$. More precisely, suppose that for all $i \mod k$, there exists $e_i'$ with $e_i\sim e_i'$ such that $e_{i+1}\subset f^{n_i}(e_i')$ for some $n_i>0$. Let $v_i$ denote the Julia vertex of $e_i$. Then, $f^{n_i}(v_i)=v_{i+1}$, and thus $v_i$'s lie in the same periodic cycle.  
    \item [(P4)] There exists a total preorder extension of $\preceq$ relative to $\sim$.
\end{enumerate}

\begin{lem}\label{lem:(P2)(P3)Imply(P4)}
    (P4) follows from (P1),(P2), and (P3).
\end{lem}
\begin{proof}
    Suppose that an $f$-invariant graph $G$ satisfies (P1), (P2), and (P3). By (P1), every edge is an $FJ$-edges. Suppose $e_1,e_2,\dots,e_k$ is a virtually ordered cycle. We use the notations $e_i', v_i, n_i$ in the definition of (P3). By Theorem \ref{thm:PreorderExt}, it suffices to show that the cycle does not virtually contain any pair of strict order.
    
    Let $E_i$ be the set of edges incident to $v_i$. Define $E:=E_1\cup E_2\dots \cup E_k$.
    
    We first show the following: (a) For each $e_i''\in E_i$ and $n\ge0$, the forward image $f^n(e_i'')$ covers at most one edge in $E$ at most once, i.e., $e_i''$ contains at most one of the $f^n$-preimages of edges in $E$, and (b) if $e_j'' \subset f^n(e_i'')$ for $e_i''\in E_i$ and $e_j'' \in E_j$, then $f^n(v_i)=v_j$. In (b), we allow $E_i=E_j$, in which case we consider $e_i''$ and $e_j''$ as possibly the same edges in $E_i = E_j$.
    
    To prove (a), suppose that $f^n(e_i'')$ covers two edges of $E$ or one edge of $E$ twice. Since edges in $E$ are $FJ$-edges, there exists $v_j$ with $v_j \in f^n({\rm int}(e_i''))$. Since $v_i$ and $v_j$ are in the same periodic cycle by (P3), it follows that $v_i \in f^m({\rm int}(e_i''))$ for some $m>0$. This contradicts (P2). For (b), let us assume $e_j''\subset f^n(e_i'')$. By the argument in the proof of (a), $f^n({\rm int}(e_i''))$ does not contain $v_j$. Hence $f^n(v_i)=v_j$.

    Next, we suppose that $e_i'' \preceq e_j''$ for some $e_i''\in E_i$ and $e_j'' \in E_j$, and we will show that $e_j''\preceq e_i''$. Since $v_l$'s are periodic Julia vertices, they are not critical points of $f$. Hence, $f\from G\to G$ is locally homeomorphism near $v_l$'s. Then, there exists $N>0$ such that $e\subset f^N(e)$ for every $e\in E$ and $f^N(v_l)=v_l$ for all $l$. Since $e_i'' \preceq e_j''$, there exists $m>0$ such that $e_j''\subset f^m(e_i'')$. It follows from (b) that $f^m(v_i)=v_j$. We can assume $N>m$, by replacing $N$ with its larger multiple if necessary. Then $f^{N-m}(e''_j)\subset f^N(e_i'')$ and $f^{N-m}(v_j)=f^N(v_i)=v_i$. It follows that $f^{N-m}(e''_j)$ contains an edge $e_i'''$ in $E_i$, i.e., $e_j''\preceq e_i'''$, and so $e_i'''\subset f^N(e_i'')$. Note that $e_i''$ is contained in $f^N(e_i'')$, and hence, by (a), it is the only edge of $E$ contained in $f^N(e_i'')$. Therefore $e_i''=e_i'''$, and this proves $e_j'' \preceq e_i''$. 

    The previous paragraph demonstrates that the virtually ordered cycle $e_1,e_2,\dots,e_k$ cannot virtually contain any pair of strict order. Therefore, by Theorem \ref{thm:PreorderExt}, the graph $G$ satisfies (P4).
\end{proof}

\subsection{Modifications (E1), (E2), (E3), and (R) of invariant graphs}\label{subsec:GraphExtensions}
Let $f$ be a post-critically finite rational map. Suppose $G$ is an $f$-invariant graph with $P_f \cup f(\V(G))\subset \V(G)$. In this subsection, we discuss the following four methods of modifying $G$: (E1) Extension by preimages, (R) Restrictions to subgraphs, (E2) Extension by internal rays, and (E3) Extension by arcs almost in the preimages.

\subsection*{(E1) Extension by preimages} Let $n\ge1$. Since $G$ is $f$-invariant, we have $f^{-n}(G)\supset G$. We consider $f^{-n}(G)$ as a graph with $\V(f^{-n}(G))=f^{-n}(\V(G))$. Since $P_f \subset \V(G)$, every edge of $f^{-n}(G)$ is mapped homeomorphically onto an edge of $G$ by $f^n$. Define $G^{(n)}$ to be the union of the connected components of $f^{-n}(G)$ that contain some connected components of $G$. We refer to $G^{(n)}$ as the {\it level-$n$ extension of $G$ by preimages}, or the {\it level-$n$ extension of $G$ by (E1)}. We note that some connected components of $G^{(n)}$ may contain more than one connected component of $G$.

It is clear that if $G$ satisfies (P1),  then so does $G^{(n)}$ for all $n>0$.

\subsection*{(R) Restriction to subgraphs} Suppose that $H$ is an $f$-invariant subgraph of $G$ with $P_f \cup f(\V(H))\subset \V(H)$. We then call the replacement of $G$ with $H$ the {\it restriction to a subgraph}. In our definition, a subgraph $H$ of $G$ satisfies $\V(H)\subset \V(G)$ and $\Edge(H)\subset \Edge(G)$, such that the interior of each edge of $H$ does not contain any vertex of $G$.
\vspace{5pt}

For a periodic Fatou component $U$, its {\it internal ray} is an arc corresponding to a radius of $\mathbb{D}$ under the B\"ottcher coordinate. If $U$ is pre-periodic, its internal ray is a preimage of an internal ray of a periodic Fatou component.

\subsection*{(E2) Extension by internal rays} Let $A=\{e_1,e_2,\dots,e_k\}$ be a set of internal rays of Fatou components that is $f$-invariant, i.e., $f(e_i)\in A$ for every $e_i\in A$. We additionally assume the following condition, which we call the {\it E2-Vertex condition}.
\begin{enumerate}[leftmargin=55pt]
    \item[(E2-V)] For each endpoint $v$ of an edge $e_i\in A$, if $v\in G$, then $v\in \V(G)$.
\end{enumerate}

Define $H=G\cup \bigcup_{i=1}^k e_i$ and $\V(H)=\V(G)\cup \bigcup_i \partial e_i$.
By (E2-V), we have $\Edge(G)\subset\Edge(H)$, and thus $G$ is a subgraph of $H$. Every edge in $\Edge(H)\setminus \Edge(G)$ is an $FJ$-edge. We call $H$ the {\it extension of $G$ by (an $f$-invariant set $A$ of) internal rays}, or the {\it extension of $G$ by (E2)}.

\subsection*{(E3) Extension by arcs almost in the preimages} This modification is more involved than the previous three. An {\it arc} $\gamma$ in $\hCbb$ refers to an embedding $\gamma \colon [0,1] \to \hCbb$, or, depending on the context, to its embedded image.

\begin{defn}[Arcs almost in the preimages]\label{defn:ArcsAlmostPreimag}
    Suppose that $f$ is a \pcf rational map and $G$ is an $f$-invariant graph. We say that an arc $\gamma\from [0,1] \to \hCbb$ is {\it almost in the preimages of $G$} if $\gamma(1) \in G$, $\gamma((0,1)) \cap G =\emptyset$, and $[t_n,1]=\gamma^{-1}(f^{-n}(G))$ for any $n\ge0$ such that $t_n$ is monotonically decreasing to $0$ as $n \to \infty$. We call $\gamma(0)$ the {\it exceptional point} of $\gamma$, which may or may not lie in $\bigcup_{n\ge0}f^{-n}(G)$.
\end{defn}

\begin{figure}
    \centering
	\def\svgwidth{0.4\textwidth}
	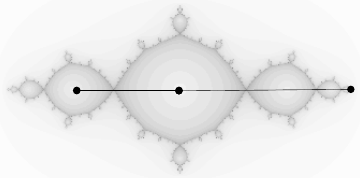
    \caption{The Basilica polynomial $f(z)=z^2-1$. The endpoints of the edge $e$ are the points $0$ and $-1$ in the critical $2$-cycle. Let $G$ be the graph consisting of a single edge $e$. The edge $e'$ joins $0$ with the $\beta$-fixed point $(1+\sqrt{5})/2$. Then, $e'$ is an arc almost in the preimages of $G$, so that $f(e')=e\cup e'$.}
    \label{fig:ArcAlmostinPreimages}
\end{figure}

\begin{lem}\label{lem:AlmostinPreimages} In the setting of Definition \ref{defn:ArcsAlmostPreimag}, the following hold.
\begin{enumerate}
    \item For all $t\in (0,1)$, we have $\gamma(t)\in \left(\bigcup_{n\ge0}f^{-n}(G)\right)\setminus G$.
    \item $\gamma^{-1}(G^{(n)})=\gamma^{-1}(f^{-n}(G))$ for all $n\ge0$.
    \item $\gamma(0)\in \Jcal_f$.
\end{enumerate}
\end{lem}
\begin{proof}
    Statements (1) and (2) are immediate from the definitions.
    Statement (3) follows from the fact that $\gamma(0)$ is an accumulation point of iterated preimages of $\V(G)$.
\end{proof}

\begin{defn}[Int-disjoint set of arcs almost in the preimages]\label{defn:CollectionArcsAlmostPreimag}
    Suppose that $f$ is a \pcf rational map and $G$ is an $f$-invariant graph. Consider a collection $A=\{\gamma_1, \gamma_2, \dots, \gamma_k\}$ of arcs almost in the preimages of $G$ and the set of its exceptional points $B:=\{\gamma_i(0):i\in\{1,2,\dots,k\}\}$. We say that $A$ is {\it int-disjoint} if ${\rm int}(\gamma_i)\cap {\rm int}(\gamma_j)$ for all $i\neq j$, and {\it $f$-invariant} if $f(B)\subset B$ and $f(H)\subset H$ where $H:=G \cup \bigcup_{i=1}^k \gamma_i$. For a technical reason, we additionally assume the following condition, which we refer to as the {\it E3-Vertex condition}.
    \begin{enumerate}[leftmargin=55pt]
        \item [(E3-V)] For each $v\in B$, if $v\in G$, then $v\in \V(G^{(m)})$ for any sufficiently large $m>0$.
    \end{enumerate}
\end{defn}


\begin{lem}\label{lem:CollectionAlmostinPreimages} Suppose that $\{\gamma_1,\gamma_2,\dots,\gamma_k\}$ is an $f$-invariant set of int-disjoint arcs almost in the preimages of $G$. For each $\gamma_i$, there exists a unique $\gamma_j$ such that $f(\gamma_i)\supset \gamma_j$. In this case, we have $f(\gamma_i(0))=\gamma_j(0)$.
\end{lem}
\begin{proof}
    One can easily find the arc $\gamma_j$ by examining the map $f$ near $\gamma_i(0)$ and $f(\gamma_i(0))$, which may be a branched covering, depending on whether $\gamma_i(0)$ is a critical point.
\end{proof}

\begin{eg}
An example of arcs almost contained in the preimages is provided by regulated paths, in the sense of \cite{Poi_HubTree}, within the filled Julia sets $\mathcal{K}_f$ of \pcf polynomials $f$ (see Figure~\ref{fig:ArcAlmostinPreimages}).

For non-polynomial rational maps, a similar but more intricate situation occurs, since the analogue of $\mathcal{K}_f$, called a cluster, need not be simply connected (see Figure~\ref{fig:RabbitandBasilica}). For further details on clusters, see Section~\ref{sec:crochet} or \cite{DHS_Decomp}.
\end{eg}

Suppose that $G$ satisfies (P1). We carefully define $\V(H)$ so that $H$ satisfies (P1). The vertex condition (E3-V) is used to include $B$ in $\V(H)$ with satisfying (P1).

Let $\{\gamma_1, \gamma_2, \dots, \gamma_k\}$ be an $f$-invariant set of int-disjoint arcs almost in the preimages of $G$. Choose $m>0$ sufficiently large such that for every $i\in\{1,2,\dots,k\}$, the following hold.
\begin{enumerate}
    \item $G^{(m)} \cap {\rm int}(\gamma_i)\neq \emptyset$.
    \item Either $\gamma_i(0)\in \V(G^{(m)})\cap G$ or $\gamma_i(0) \notin G$; see (E3-V).
\end{enumerate}
For each $i\in \{1,2,\dots,k\}$, let $x_i$ be the vertex in $\V(G^{(m)}) \cap {\rm int}(\gamma_i)$ that is closest to $\gamma_i(0)$. We define a set $V$ as the collection of $x_i$'s that are Julia vertices. Recall $B=\{\gamma_i(0):1\le i\le k\}$ is a set of Julia vertices (Lemma~\ref{lem:AlmostinPreimages}). We define the vertex set $\V(H)$ of $H:=G \cup \bigcup_{i=1}^k \gamma_i$ by
\[
    \V(H)= \left((\V(G^{(m)}) \cap H) \cup B\right)\setminus V. 
\]
Since $G^{(m)}\cap {\rm int}(\gamma_i)\neq \emptyset$, we have $\deg_H(x_i)=2$ for each $x_i$, even though possibly $\deg_{G^{(m)}}(x_i)>2$. Hence $x_i$ may be excluded from the vertex set of $H$.
We also remark that $\V(H)$ depends on the choice of $m>0$. It easily follows from the definition of $V$ that for all $x\in V$, if $f(y)=x$ for some $y\in H$, then $y\in V$ as well. Hence, $\V(H)$ is $f$-invariant, and $H$ satisfies (P1). We call this construction of $H$ the {\it extension of $G$ by (the $f$-invariant int-disjoint set  $\{\gamma_1, \gamma_2, \dots, \gamma_k\}$ of) arcs almost in the preimages of $G$}, or simply the {\it extension of $G$ by (E3)}.

Let us define an $f$-invariant subgraph $K$ of $H$ by removing, for every $i\in\{1,2,\dots,k\}$, the edge $e_i$ with $e_i\subset \gamma_i$ and $\gamma_i(0)\in e_i$. Then $K$ is a subgraph of $G^{(m)}$ for some large $m>0$. The following lemma is immediate by the above construction of $H$.

\begin{lem}\label{lem:E3} Suppose $G,H,$ and $K$ are defined as above. The extension $H$ of $G$ by (E3) is equivalent to first extending $G$ to $G^{(m)}$ using (E1), then restricting $G^{(m)}$ to $K$ via (R), and finally adding new $FJ$-edges that are arcs almost in the preimages of $K$. 
\end{lem}

\subsection*{(E3')-extension} For a technical reason, we are interested in a special case of (E3)-extensions. Recall the notations used in the definition of the (E3)-extension. We refer to the (E3)-extension of $G$ as an (E3')-extension if
\begin{enumerate}
    \item Each connected component $G_i$ is either fixed ($f(G_i)\subset G_i$) or prefixed ($f(G_j)\subset G_j$ with $f(G_i)\subset G_j$),
    \item $A=\{\gamma\}$ such that $f(\gamma)\supset \gamma$ and $f(\gamma(0))=\gamma(0)$, and 
    \item If $\gamma(0)\in G$, then $\gamma(0)$ and $\gamma(1)$ lie in different connected components $G_1$ and $G_2$ of $G$, which are fixed by $f$.
\end{enumerate}

\subsection{Preservation of (P1)--(P4) under extensions by (E1)--(E3)}
Recall that the properties (P3) and (P4) rely on the equivalence relation $\sim$, which was defined by using (P1).

\begin{thm}\label{thm:TotalPreorderGraphs}
    Let $f$ be a \pcf rational map, and let $G$ be an $f$-invariant graph, possibly disconnected, with $P_f \cup f(\V(G)) \subset \V(G)$. Suppose that $\preceq$ and $\sim$ are the preorder and equivalence relation on $\Edge(G)$ defined in Definition \ref{defn:PreorderEqrel}. If $G$ satisfies $(P1),(P2),(P3)$, and $(P4)$ and $h_{top}(f|_G)=0$, then all extensions $H$ of $G$ by (E1),(E2), and (E3') also satisfy $(P1),(P2),(P3),(P4)$, and $h_{top}(f|_H)=0$. More precisely, the following hold.
    \begin{enumerate}
        \item (E1)--(E3) and (R) preserve (P1).
        \item (E1),(E2), (E3'), and (R) preserve (P2).
        \item (P1) and (P2) implies $h_{top}=0$ (Lemma \ref{lem:P1P2RecFF Edge}).
    \end{enumerate}
     Below, we suppose that $G$ satisfies (P1) and (P2), so by (1) and (2), the extended graph $H$ also satisfies (P1) and (P2).
    \begin{enumerate}[resume]
        \item (E1) and (R) preserves (P3)
        \item (E2) and (E3') preserve (P3).
        \item (P1), (P2), and (P3) implies (P4) (Lemma \ref{lem:(P2)(P3)Imply(P4)}).
    \end{enumerate}
\end{thm}

\begin{proof}[Proof of Theorem \ref{thm:TotalPreorderGraphs}-(1)]
    The preservation of (P1) is clear for (E1) and (R), and has already been discussed for (E2) and (E3) in their respective constructions.
\end{proof}

\begin{eg}[Possibility that (E3) does not preserve (P2) and (P3)]\label{rem:selfjoin}
\begin{figure}
	\centering
    \begin{subfigure}[b]{0.3\textwidth}
    \def\svgwidth{\textwidth}
        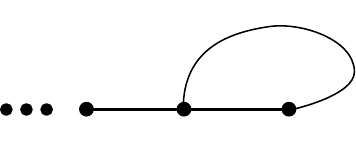
    \caption{}
    \end{subfigure}
    \hspace{20pt}
    \begin{subfigure}[b]{0.4\textwidth}
    \def\svgwidth{\textwidth}
        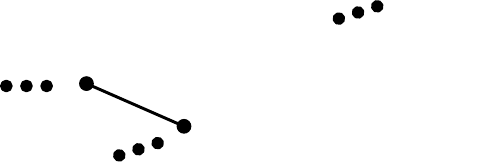
    \caption{}
    \end{subfigure}
    \caption{(E3) may not preserve (P2) and (P3). See Example \ref{rem:selfjoin}}
    \label{fig:selfjoin}
\end{figure}
See Figure \ref{fig:selfjoin}-(A). Suppose $G$ is an $f$-invariant graph, $e_1,e_2\in \Edge(G)$ with $f(e_1)=e_1$ and $f(e_2)=e_2$, and $p\in \V_J(G)$ is a fixed point. Suppose that $e_3$ is an arc almost in the preimages of $G$ such that $e_1,e_2\subset f(e_3)$. Then, $p\in f({\rm int}(e_3))$, and thus the graph $H:=G\cup e_3$, which is an (E3)-extension of $G$, does not satisfy (P2). We avoid this issue preventing a new arc almost in the preimages of $G$ from joining two vertices in the same connected components; see Condition (3) of (E3')-extension.

Consider Figure \ref{fig:selfjoin}-(B). Suppose $G_1$ and $G_2$ are connected components of an $f$-invariant graph $G$ such that for $i\in\{1,2\}$, $e_i\in \Edge(G_i)$ and $p_i\in \V_J(G_i)$. Assume $\{e_1',e_2'\}$ is an int-disjoint collection of arcs almost in the preimages of $G$ such that $f(G_1)\subset G_2$, $f(G_2)\subset G_1$, $e_2\subset f({\rm int}(e_1'))$, $e_2\subset f({\rm int}(e_1'))$, $f(p_1)\neq p_2$, and $f(p_2)\neq p_1$, where $p_1\in \V_J(G_1)$ and $p_2\in \V_J(G_2)$. Then, $e_1\sim e_1'$, $e_1'\preceq e_2$, $e_2\sim e_2'$, and $e_2'\preceq e_1$, so that $\{e_1,e_2\}$ is a virtually ordered cycle violating (P3). We avoid this situation by adding only one arc almost in the preimages of $G$ at each time, in (E3')-extensions.  
\end{eg}

Recall that (E3') is a special case of (E3).

\begin{proof}[Proof of Theorem \ref{thm:TotalPreorderGraphs}-(2)]
    It is clear that (R) preserve (P2).

    In (E2)-extensions, the newly added edges are internal rays of Fatou components and are therefore periodic or preperiodic, so (P2) is preserved.
    
    Let us show (E1) preserves (P2). Suppose $G^{(n)}$ is the level-$n$ extension of $G$ by preimages. Suppose that $e^n$ is an $FJ$-edge of $G^{(n)}$ with the Julia endpoint $v$. Then $f^n$ maps $e^n$ to an $FJ$-edge $e$ of $G$ homeomorphically, such that $f^n(v)$ is the Julia endpoint of $e$. If $v\in f^k({\rm int}(e^n))$ for some $k>0$, then $f^n(v) \in f^k(f^n({\rm int}(e^n)))=f^k({\rm int}(e))$. It contradicts (P2) for $G$, and so (E1) preserves (P2).

    Let us show that (P2) is preserved under extensions $H$ of $G$ by (E3'). Since we have proved that (E1) and (R) preserve (P2), by using Lemma \ref{lem:E3}, we may assume that edges in $\Edge(H)\setminus \Edge(G)$ are arcs almost in the preimages of $G$.
    
    Suppose that $H$ does not satisfy (P2). There exists $e\in \Edge(H)\setminus \Edge(G)$, whose  Julia endpoint is denoted by $v$, such that $v\in f^n({\rm int}(e))$ for some $n>0$. By definition of (E3'), we have $e\subset f(e)$ with $v=f(v)$. Let $w$ denote the Fatou Julia endpoint of $e$, which lie in $G$.

    By definition of (E3'), for every point $x\in {\rm int}(e)$, we have $x\notin G$ and $f^N(x)\in G$ for all sufficiently large $N>0$. Since $v=f^m(v)\in f^{n+m}({\rm int}(e))\subset e$ for all $m>0$, we have $v\in G$ such that $v$ and $w$ lie in the same connected component of $G$. This contradicts the assumption of (E3').
\end{proof}

\begin{proof}[Proof of Theorem \ref{thm:TotalPreorderGraphs}-(4)]
     Suppose $G$ satisfies (P1), (P2), and (P3). It is clear that (R) preserves (P3). We show the level-$n$ extension $G^{(n)}$ by (E1) satisfies (P3).
    
    Suppose $e^n_1,e^n_2,\dots, e^n_k$ is a virtually ordered cycle of $G^{(n)}$. More precisely, for each $i \mod k$, there exists $e'^n_i$ with $e^n_i\sim e'^n_i$ such that $e^n_{i+1}\subset f^{n_i}(e'^n_i)$ for some $n_i>0$. Let $v^n_i$ denote the Julia vertex of $e^n_i$. By the definition of (E1)-extension, $e_i:=f^n(e^n_i)$ is an edge of $G$. The map $f^n\from \Edge(G^{(n)})\to \Edge(G)$ preserves the preorders and equivalence relations. Hence $e_1,e_2,\dots,e_k$ are also a virtually ordered cycle on $\Edge(G)$ such that $e'_{i+1}:=f^n(e'^n_{i+1})\subset f^{n_i}(f^n(e^n_i))=f^{n_i}(e_i)$ and $e'_i\sim e_i$ for all $i\mod k$. The vertex $v_i:=f^n(v^n_i)$ is the common Julia vertex of $e_i$ and $e'_i$. By (P3) of $G$, we have $f^{n_i}(v_i)=v_{i+1}$ for all $i \mod k$. By (P1), we have $v_i=f^n(v^n_i)$ for all $i \mod k$.
    
    Next, we show that $f^{n_i}(v_i^n)=v^n_{i+1}$ for all $i \mod k$. For contradiction, without loss of generality, suppose that $v^n_2\neq f^{n_1}(v^n_1)$. Since $v^n_2\in f^{n_1}(e'^n_1)$ and $e^n_1$ is an FJ-edge, we have $v_2^n\in f^{n_i}({\rm int}(e'^n_1))$. Hence, by taking $f^n$, we obtain $v_2\in f^{n_i}({\rm int}(e'_1))$. By the conclusion of the previous paragraph, we have $v_1=f^{N}(v_2)$ where $N=n_2+n_3+\dots+n_k$. Hence, $v_1\in f^{n_1+N}({\rm int}(e'_1))$, which contradicts (P2) of $G$.
\end{proof}

We prove Theorem \ref{thm:TotalPreorderGraphs}-(5) by using an induction argument and the following lemma.

\begin{lem}\label{lem:P3 from P2}
    For a \pcf rational map $f$, let $H$ be an $f$-invariant graph satisfying (P1).
    Suppose $e_1,e_2,\dots,e_k$ is a virtually ordered cycle of $H$ such that for each $i\mod{k}$, there exist $n_i\ge 0$ and an edge $e_i'$ such that $e_i\sim e_i'$ and $e_{i+1}\subset f^{n_i}(e_i')$. Let $v_i$ denote the common Julia vertex of $e_i$ and $e_i'$.
    \begin{enumerate}
        \item If $k=1$ and $H$ satisfies (P2), then $f^{n_1}(v_1)=v_1$.
        \item Suppose $k\ge2$. If $f^{n_1}(v_1)=v_2$, then there exists $e_2''$ with $e_2\sim e_2''$ such that $e_2'',e_3,\dots,e_k$ is a virtually ordered cycle. More precisely, for $k>2$, we have $e_2''\subset f^{n_k+n_1}(e_k')$ and $e_{i+1}'\subset f^{n_i}(e_i')$ for $i=2,3,\dots,k$. For $k=2$, we have $e_2''\subset f^{n_2+n_1}(e_2')$. 
        \item In the case of (2), suppose $H$ satisfies (P2) and $f^{n_i}(v_i)=v_{i+1}$ for $i=2,3,\dots,k-1$. Then $f^{n_k}(v_k)=v_1$.
    \end{enumerate}
\end{lem}
\begin{proof}
    (1) Suppose $k=1$ and $H$ satisfies (P2). Then, $e_1\subset f^{n_1}(e_1')$. Hence, if $v_1\neq f^{n_1}(v_1)$, then $v_1\in f^{n_1}({\rm int}(e_1))$, which contradicts (P2).

    (2) Suppose $k\ge2$ and $f^{n_1}(v_1)=v_2$. Since $v_1\in e_1\subset f^{n_k}(e_k')$, we have $v_2\in f^{n_k+n_1}(e_k')$. It follows that there exists an edge $e_2''$ incident to $v_2$ such that $e_2''\subset f^{n_k+n_1}(e_k')$. Then, we obtain the desired virtually ordered cycle.

    (3) Suppose $f^{n_k}(v_k)\neq v_1$. Then $v_1\in f^{n_k}({\rm int}(e_k))$. It follows that for $N:=n_1+n_2+\dots+n_{k-1}$, we have $v_k=f^{N}(v_1)\in f^{N+n_k}({\rm int}(e_k))$, which contradicts (P2).
\end{proof}

\begin{proof}[Proof of Theorem \ref{thm:TotalPreorderGraphs}-(5)]
    Suppose $G$ satisfies (P3) and $H$ is an extension of $G$ by (E2) or (E3'). Suppose that $e_1,e_2,\dots,e_k$ is a virtually ordered cycle of $H$ but not of $G$. We use the notations $e_i',v_i,n_i$ defined in Lemma \ref{lem:P3 from P2}. By an induction argument using Lemma \ref{lem:P3 from P2}, it suffices to show that $f^{n_j}(v_j)=v_{j+1}$ for some $j$.
    
    Since $e_1,e_2,\dots,e_k$ is not a virtually ordered cycle of $G$, there exists $i$ such that $e_i$ or $e_i'$ does not lie in $G$.
    
    Consider the case $e_i\notin \Edge(G)$. Without loss of generality, assume $i=1$. Since $e_1\subset f^{n_k}(e_k')$, by $f$-invariance of $G$, we obtain $e_k' \notin \Edge(G)$. If $H$ is an (E2)-extension, then both $e_1$ and $e_k'$ are internal rays of Fatou components. Hence $e_1\subset f^{n_k}(e_k')$ implies $e_1= f^{n_k}(e_k')$ and $v_1 = f^{n_k}(v_k)$. If $H$ is an (E3')-extension, by Lemma \ref{lem:CollectionAlmostinPreimages}, one can deduce $v_1=f^{n_k}(v_k)$ from $e_1\subset f^{n_k}(e_k')$. 

    Consider the case $e_i'\notin \Edge(G)$. Without loss of generality, assume $i=1$. Note that $e_2\subset f^{n_1}(e_1')$.
    
    If $H$ is an (E2)-extension, then $e_1'$ is an internal ray of a Fatou component. Then, $e_2\subset f^{n_1}(e_1')$ implies $e_2$ is also an internal rays of Fatou component such that $e_2= f^{n_1}(e_1')$ and $v_2= f^{n_1}(v_1)$.
    
    Suppose that $H$ is an (E3')-extension. By Lemma \ref{lem:E3}, we may assume that $H$ is obtained by adding a single new edge that is an arc almost contained in the preimages of $G$. Since we assumed $e_1'\notin \Edge(G)$, the new edge is $e_1'$.

    By the argument for the case $e_1 \notin \Edge(G)$, we may assume that $e_i \in \Edge(G)$ for all $i$, and hence $v_i \in \V(G)$ for all $i$. By the definition of (E3'), we have $f(e_1') \supset e_1'$ and $f(v_1) = v_1$.

    Let $G_1$ be the connected component of $G$ containing the Fatou endpoint of $e_1'$. Since $f(e_1') \supset e_1'$, it follows that $f(G_1) \subset G_1$.

    If $v_2 = v_1 (= f^{n_1}(v_1))$, then we can apply the induction argument using Lemma \ref{lem:P3 from P2}. Suppose instead that $v_2 \neq v_1$. Since $v_2 \in e_2 \subset f^{n_1}(e_1')$, we have $v_2 \in f^{n_1}(\mathrm{int}(e_1'))$, and hence $v_2 \in G_1$. It follows that $e_2 \subset G_1$. Moreover, $e_2' \subset G$, and thus $e_2' \subset G_1$. Indeed, if $e_2' \notin \Edge(G)$, then since $e_1'$ is the only new edge, we must have $e_2' = e_1'$, which implies $v_1 = v_2$, contradicting our assumption.

    From $e_3 \subset f^{n_2}(e_2')$, it follows that $e_3 \in \Edge(G_1)$. By the same argument used to show $e_2' \in \Edge(G_1)$, we obtain $e_3' \in \Edge(G_1)$. Proceeding inductively, we conclude that $e_k, e_k' \in \Edge(G_1)$, and hence $e_1 \in \Edge(G_1)$. This implies $v_1 \in \V(G_1)$, contradicting Condition (3) of (E3') that the two endpoints of $e_1'$ cannot lie in the same connected component of $G$.
\end{proof}

\subsection*{Property (P5) and vertex conditions (E2-V) and (E3-V)}
Recall the (E2)- and (E3)-extensions relies on the vertex conditions (E2-V) and (E3-V). We will use the following property, referred to as (P5), to show that (E2-V) and (E3-V) are always satisfied in the setting of our application.
\begin{enumerate}
    \item[(P5)] If a point $x\in G$ is periodic under $f$, then $x\in \V(G)$.
\end{enumerate}
Property (P5) is trivial for Fatou periodic points, since they lie in $P_f$. Moreover, it is easy to see that $h_{top}(f|_G)>0$ implies that $G$ contains infinitely many periodic (Julia) points. Hence, (P5) relies on the condition $h_{top}(f|_G)=0$.

\begin{lem}\label{lem:(P5)Preserve}
    Suppose that $G$ is an $f$-invariant graph with $P_f\subset G$ satisfying (P1) and $h_{top}(f|_G)=0$. Then (E1), (E2), and (E3) preserve (P5).
\end{lem}
\begin{proof}
    The statement is straightforward for (E1). In our constructions of (E2) and (E3), new periodic Julia points are defined as vertices of the extended graphs.
\end{proof}

\begin{lem}\label{lem:(P5) (E2-V)}
    Suppose $G$ is an $f$-invariant graph with $P_f\cup f(\V(G))\subset f(\V(G))$ satisfying (P1) and (P5). Let $A:=\{e_1,e_2,\dots,e_k\}$ be an $f$-invariant collection of internal rays of Fatou components. Suppose that the level-$N$ (E1)-extension $G^{(N)}$ contains all Fatou endpoints of $e_i$'s for some $N>0$. Then, the condition (E2-V) for $A$ holds for $G^{(m)}$ for any sufficiently large $m>0$.
\end{lem}
\begin{proof}
    Recall that the set $B$ of Julia endpoints of $e_i$'s is $f$-invariant, and it consists of periodic cycles and their preimages. Every periodic point in $B\cap G$ lies in $\V(G)$ by (P5). Hence, every point $x\in B$ with $x\in G$ is contained in $\V(G^{(M)})\cap G$ for some large $M>0$. Then (E2-V) holds for $G^{(m)}$ for every $m\ge\max\{M,N\}$.
\end{proof}

\begin{lem}\label{lem:(P5) (E3-V)}
    Suppose that $G$ is an $f$-invariant graph with $P_f\cup f(\V(G))\subset \V(G)$ satisfying (P1) and (P5). Then, for every $f$-invariant set $A:=\{\gamma_1,\gamma_2,\dots,\gamma_k\}$ of int-disjoint arcs almost in the preimages of $G$, the condition (E3-V) for $A$ holds.
\end{lem}
\begin{proof}
    By Lemma \ref{lem:CollectionAlmostinPreimages}, $X:=\{\gamma_i(0)\}_{i=1,2,\dots,k}$ is an $f$-invariant set. The remainder of the proof is similar to that of Lemma \ref{lem:(P5) (E2-V)}.
\end{proof}

\section{Crochet maps}\label{sec:crochet}

In this section, we state Dudko--Hlushchanka--Schleicher's crochet decomposition theorem (Theorem~\ref{thm:DHScrochet}) and prove one direction of Theorem~\ref{theorem:main} (Corollary~\ref{coro:NotCrochetCDim}). Section~\ref{subsec:eg of crochet} presents examples, including polynomial matings (Proposition~\ref{prop:MatingCrochet}). In Section~\ref{subsec:EntZeroGraph_Crochet}, we review the construction of entropy-zero invariant graphs for crochet maps by Dudko--Hlushchanka--Schleicher and relate it to materials in Section~\ref{Sec:GraphExt and Preorder}.

\begin{defn}[Crochet maps]
	For a \pcf rational map $f$ with non-empty Fatou set, we say that $f$ is a {\it crochet map} if there exists an $f$-invariant connected graph $G$ with $P_f\subset G$ and $h_{top}(f|_G)=0$.
\end{defn}
\begin{defn}[Cantor multicurve]
	Let $\TmapA$ be a \pcf branched covering. A multicurve $\Gamma$ of $(S^2,A)$ is a {\it Cantor multicurve} if there is $\gamma\in \Gamma$ such that the number of connected components of $f^{-n}(\Gamma)$ that are isotopic to $\gamma$ relative to $A$ grows exponentially fast with $n$.
\end{defn}

It follows from the definitions that a multicurve $\Gamma$ is a Cantor multicurve if and only if the Perron-Frobenius eigenvalue of the linear 1-transformation $f_{1,\Gamma}\from \Rbb^\Gamma \to \Rbb^\Gamma$ is greater that one, i.e., $\lambda_1(\Gamma)>1$. Hence, if $\Gamma$ contains no Levy cycles, then $\Gamma$ is a Cantor multicurve if and only if $Q(\Gamma)>1$.

Cantor multicurves produces a surjection from $\mathcal{C}\times S^1$ onto subsets of Julia sets, where $\mathcal{C}$ is a Cantor set; see Figure \ref{fig:Neckless} and \cite[Section 7.3]{PilThu_ARConfDim}.

\begin{figure}
	\centering
	\def\svgwidth{0.4\textwidth}
	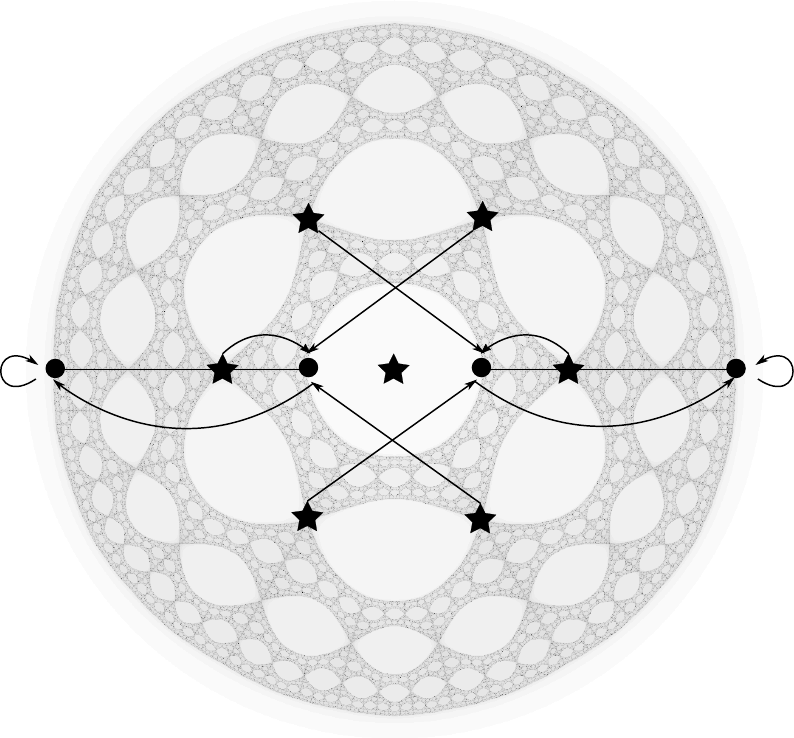
	\caption{$f(z)\approx 0.128262\, z^3+1/z^3$. This is an example in \cite[Figure 1-(e)]{DHS_Decomp}. The points $0$ and $\infty$ have degree $3$ and $f(0)=\infty$. The other six critical points are of degree 2 and prefixed points with preperiod 2. Let $\gamma$ be a simple closed curve that is the boundary of a small neighborhood of the Fatou component of $0$. Then $f^{-1}(\gamma)$ consists of two simple closed curves both isotopic to $\gamma$ relative to $P_f$. Thus $\{\gamma\}$ is a Cantor multicurve. The Julia set $\Jcal_f$ is the union of uncountably many simple closed curves homotopic to $\gamma$. The Fatou quotient $\hCbb/\!\sim_F$ can be identified with a line segment on the real line.}
\label{fig:Neckless}
\end{figure}

\begin{defn}[Fatou quotient]
	For a \pcf rational map $f$ with a non-empty Fatou set, we define an equivalence relation $\sim_F$ on $\hCbb$ as the closure, as a subset of $\hCbb\times \hCbb$, of the equivalence relation $\sim$ defined by
	\begin{center}
    	$x\sim y$ if and only if $x$ and $y$ are in the same Fatou component.
	\end{center}
	The quotient space $\hCbb/\!\sim_F$ is called the {\it Fatou quotient} of the Riemann sphere. Note that $\hCbb/\!\sim_F$ is equal to the Hausdorff quotient of $\hCbb/\!\sim$.
\end{defn}

\begin{thm}[Dudko--Hlushchanka--Schleicher \cite{DHS_Decomp}]\label{thm:DHScrochet}
Let $f$ be a \pcf rational map with a non-empty Fatou set. Then there exists a completely $f$-invariant multicurve $\Gamma$ so that any small rational map of the $\Gamma$-decomposition is either a \Sier map or a crochet map. Moreover, if $f$ is not a crochet map then either
\begin{itemize}
	\item $\Gamma$ is a Cantor multicurve, or possibly non-exclusively
	\item there exists a small rational map of the $\Gamma$-decomposition that is a \Sier map.
\end{itemize}
Also the following are equivalent.
\begin{itemize}
    \item [(1)] For every pair of Fatou components $U$ and $V$, their centers are connected by a curve $\gamma$ that intersects the Julia set $\Jcal_f$ in a countable set.
    \item [(2)] The Fatou quotient $\hCbb/\sim_{F}$ is a singleton.
    \item [(3)] $f$ is a crochet map.
\end{itemize}
\end{thm}

By \cite[Theorem C]{PilThu_ARConfDim}, if a \pcf hyperbolic rational map $f$ has a Sierpi\'{n}ski carpet Julia set, then $\overline{E}^1(f) > 1$.

\begin{coro}\label{coro:NotCrochetCDim}
Let $f$ be a \pcf hyperbolic rational map. If $f$ is not a crochet map, then $\ARC(\Jcal_f)>1$.
\end{coro}
\begin{proof}
By Theorem \ref{thm:DHScrochet}, there is a multicurve $\Gamma$ which is completely invariant up to isotopy such that either $\Gamma$ is a Cantor multicurve or there exists a small rational map of the $\Gamma$ decomposition, say $g$, such that the Julia set $\Jcal_g$ is a Sierpi\'{n}ski carpet.

If $\Gamma$ is a Cantor multicurve, then $Q(\Gamma)>1$. Thus $\ARC(\Jcal_f)>1$ by Theorem \ref{thm:MonoConfEngDim}.

Suppose that there exists a small rational map $g$ whose Julia set $\Jcal_g$ is a Sierpi\'{n}ski carpet. By Lemma \ref{lem:SmallMapHyp}, the map $g$ is a hyperbolic \pcf rational map. By Theorem \ref{thm:MonoConfEngDim} and \cite[Theorem C]{PilThu_ARConfDim}, we have
\[
	\overline{E}^1(f) \ge \left(\overline{E}^1(g)\right)^{1/\tau}>1,
\]
where $\tau$ is the first return time of $g$. Then $\ARC(\Jcal_f)>1$ follows from Theorem \ref{thm:ARCdim} and the fact that $p\mapsto \overline{E}^p(f)$ is monotonically decreasing (Theorem \ref{thm:PropertyOfComfEnergy}).
\end{proof}

\subsection{Examples of crochet maps}\label{subsec:eg of crochet}

Below is a list of families of crochet maps and their invariant graphs.
\begin{itemize}
	\item Post-critically finite polynomials and spiders \cite{spider}. More precisely, spiders define graphs that are forward $f$-invariant up to isotopy. To obtain $f$-invariant graphs, we can realize each edge of the spider graphs by a concatenation of one external ray and one internal ray.
	\item Critically fixed rational maps and Tischler graphs \cite{pilgrim_Classification_criticallyfixed, Hlushchanka_criticallyfixed}.
	\item Second iterates of critically fixed anti-rational maps and  Tischler graphs \cite{Geyer_CritFixAnti, LLM_KissCritFixAnti}.
	\item Post-critically finite Newton maps and their extended Newton graphs \cite{LMS_ClassfiPCFiniteNewton, DMRS_ClassifiPCFixedNewton}. We need a slight modification in this case. In these articles \cite{LMS_ClassfiPCFiniteNewton, DMRS_ClassifiPCFixedNewton}, invariant graphs, called {\it extended Newton graphs}, are defined as the extensions of Newton graphs by using Hubbard trees for renormalizable parts. Instead of using Hubbard trees, we extend the Newton graphs by using (pre-)periodic bubble rays as a replacement for external rays. This extension yields invariant graphs with topological entropy zero. 
	\item Matings of \pcf polynomials at least one of whose core entropy is zero (Proposition \ref{prop:MatingCrochet}).
\end{itemize}

\begin{prop}[Crochet matings]\label{prop:MatingCrochet}
	Let $f$ and $g$ be degree-$d$ \pcf polynomials. Suppose that the B\"{o}ttcher coordinate for the external Fatou components are fixed so that the mating of $f$ and $g$ is uniquely defined. Assume that $f$ and $g$ are mateable and $F$ is the corresponding degree-$d$ rational map. If $f$ (or $g$) has core entropy zero, then $F$ is a crochet map.
\end{prop}
\begin{proof}
	Let $S_g$ be the spider of $g$ and $\Theta(g)$ be the set of external angles whose external rays are in $S_g$. Let $M$ be the set of landing points of external rays of $f$ whose angles are in $-\Theta(g)$. Let $T_f$ be the $f$-invariant tree defined as the regulated hull of $M\cup P_f \cup \textup{Crit}(f)$. By \cite{Rees_semiconj, ShiShi_Rees}, $F$ is topologically conjugate to the quotient of the formal mating of $f$ and $g$ by the ray-equivalence classes. Hence, $S_g \cup T_f$ yields a connected $F$-invariant graph $G$ containing $P_f$. It is not hard to show that both $S_g$ and $T_f$ have entropy zero. Hence $G$ also has topological entropy zero.
\end{proof}

\begin{rem}
	Even if both polynomials $f$ and $g$ have positive core entropy, their mating may be a crochet map, e.g., the mating of the Airplane polynomial ($\approx z^2-1.75488$) and the Kokopelli polynomial ($\approx z^2-0.15652+ 1.03225 i$).
\end{rem}

\subsection{Entropy zero invariant graphs of crochet maps}\label{subsec:EntZeroGraph_Crochet}

In this subsection, we review the construction of entropy zero invariant graphs in \cite{DHS_Decomp}.

\subsection*{Levy arcs and adjacency of Fatou components} Let $f$ be a \pcf rational map. Recall that an arc of $(\hCbb,P_f)$ is an embedding $\alpha\from [0,1]\to \hCbb$ so that $\alpha((0,1)) \subset \hCbb \setminus P_f$ and $\alpha(\{0,1\})\subset P_f$. A {\it Levy arc} is an arc $\alpha$ satisfying the following: There exists a finite sequence of arcs $\alpha=\alpha_1,\alpha_2,\dots, \alpha_n$ such that $\alpha_{i+1}$ is homotopic relative to $P_f$ to a lifting of $\alpha_i$ through $f$ for every $i\mod{n}$.

By using the forward-expanding property of the canonical orbifold metric (Equation \eqref{eqn:Contract_Conf.metric}), we can show that each Levy arc $\alpha$ is homotopic to either (i) a single internal ray or (ii) the concatenation of two internal rays of Fatou components. Hence, we always assume that Levy cycles refer to these representatives.

Consider a \pcf\ rational map $f$. Then the Julia set $\Jcal_f$ is connected and locally connected, and every Fatou component of $f$ is simply connected.  We say that two Fatou components $U$ and $V$, possibly with $U = V$, are {\it adjacent} if $\overline{U} \cap \overline{V} \neq \emptyset$. A Fatou component $U$ is called {\it self-adjacent} if $U$ is not a Jordan domain. Moreover, the Julia set $\Jcal_f$ is a Sierpi\'{n}ski carpet if and only if no two Fatou components are adjacent and no Fatou component is self-adjacent \cite{Whyburn_TopChacSierCarpet}.

Proposition \ref{prop:NoLevyArciffCarpet} is proved for \pcf hyperbolic rational maps in \cite[Corollary 5.18]{Pil_Thesis}. Although the non-hyperbolic case is well-known to experts, we include it's proof here for completeness and future reference.

\begin{prop}\label{prop:NoLevyArciffCarpet}
	For a \pcf rational map $f$ with a non-empty Fatou set, the Julia set $\Jcal_f$ is a \Sier if and only if $f$ does not have a Levy arc of $(\hCbb,P_f)$.
\end{prop}
\begin{proof}
    By \cite[Proposition 1.3]{BartholdiDudko_expanding} and \cite[Theorem 5.1]{GHMZ_JCorveCarpetJulia}, we obtain one direction: if $\Jcal_f$ is a Sierpi\'{n}ski carpet, then there is no Levy arc.
    
    Suppose there exists a Levy arc $\alpha$ of $(\hCbb,P_f)$. Recall that $\alpha$ is (represented as) either (i) an arc connecting the centers of two periodic Fatou components or (ii) an internal ray of a periodic Fatou component.
    
    In Case (ii), $\alpha$ joins the centers of two adjacent Fatou components. Consider Case (i). Let $v$ be the Julia endpoint of $\alpha$. Since $v\in P_f$, there exist a critical point $w$ with $f^n(w)=v$ for some $n>0$. Then there exists more than one lifts of $\alpha$ via $f^n$ incident to $w$, and thus more than one Fatou components are adjacent to each other at $w$.
\end{proof}

\subsection*{Extension of entropy zero invariant graphs}
Suppose $\Jcal_f$ is not a \Sier so that there exist Levy arcs by Proposition \ref{prop:NoLevyArciffCarpet}. We use the graph extensions (E1),(E2), and (E3) discussed in Section \ref{subsec:GraphExtensions}.

{\it Step 0: initial graphs}\indent  Let $G$ be the union of the Levy arcs and $P_f$ with $\V(G)=P_f$. We allow a single vertex to be a connected component of $G$. We define each internal ray in the Levy arcs as an edge of $G$. Hence, every edge of $G$ is an $FJ$-edge.

{\it Step 1: merging components in backward iterates}\indent For any $n\ge0$, let $G^{(n)}$ be the level-$n$ extension of $G$ by preimages. We take a sufficiently large $n>0$ so that the number of connected components of $G^{(n)}$ is equal to that of $G^{(m)}$ for all $m>n$. We define $G^{(n)}$ as a new $G$.

{\it Step 2: merging components whose clusters intersect}\indent Let $G_1,G_2,\dotsm,G_k$ be connected components of $G$. For any $n\ge0$, we define $G^{(n)}_i$ as the connected component of the level-$n$ extension $G^{(n)}$ of $G$ by preimages that contains $G_i$.

For each $G_i$, we define the {\it level-$n$ cluster} $K_i^{(n)}$ and the {\it cluster} $K_i$ of $G_i$ as the closures of the unions of Fatou components intersecting $G^{(n)}_i$ and $\bigcup_{n\ge0} G^{(n)}_i$, respectively. For each $i$, we have $f(K_i)=K_j$ for some $j$. We say that $K_i$ is {\it periodic} if $f^m(K_i)$ for some $m>0$.

Suppose that $K_i \cap K_j \neq \emptyset$ with $i \neq j$.
If both $K_i$ and $K_j$ are periodic, then, by \cite[Lemma 4.3]{DHS_Decomp}, there exists a periodic point $p \in K_i \cap K_j$.

In general, there exists $m > 0$ such that $K_{i'} := f^m(K_i)$ and $K_{j'} := f^m(K_j)$ are periodic. If $i' \neq j'$, then there exists a periodic point $p' \in K_{i'} \cap K_{j'}$, and hence an $f^m$-preimage of $p'$ lies in $K_i \cap K_j$. If $i' = j'$, then by modifying the argument in \cite[Lemma 4.3]{DHS_Decomp}, we can find a periodic point $p' \in K_{i'}$ such that one of its $f^m$-preimages lies in $K_i \cap K_j$. In \cite[Lemma 4.3]{DHS_Decomp}, the authors consider Levy arcs between two periodic clusters; here, we instead consider Levy arcs that self-join a single cluster.

Hence, there exists a finite set $P \subset \bigcup_{i\neq j} K_i \cap K_j$ with $f(P)\subset P$ such that $P \cap (K_i \cap K_j) \neq \emptyset$ whenever $K_i \cap K_j \neq \emptyset$. We also include $P_f$ in $P$, i.e., $P_f\subset P$. For each $p\in P\cap K_i$ and some large $m>0$, we connect $G^{(m)}_i$ to $p$ with certain paths $\tilde{\gamma}_{i,p,l}$, where $l\in \{1,2,\dots,n_{i,p}\}$ such that the graph
\begin{equation}\label{eqn:GraphExtension}
    H:=\bigcup_i G^{(m)}_i \cup \bigcup_{i,p,l} \tilde{\gamma}_{i,p,l}
\end{equation}
is $f$-invariant. We briefly describe the construction of $\tilde{\gamma}_{i,p,l}$ for the case when $p$ and $K_i$ are fixed by $f$.

There are three possible cases.
\begin{enumerate}[leftmargin=*]
    \item If $p\in G_i^{(m)}$ for some $m>0$, then we do not create any new arcs $\tilde{\gamma}_{i,p,l}$; in this case, $p\notin \bigcup_{n>0} G_j^{(n)}$ for some $j$.
    \item Suppose $p\notin \bigcup_{n>0} G_i^{(n)}$ but $p$ is on the boundary of a Fatou component. Let $\{U_{i,p,l}:1\le l \le n_p\}$ be a periodic cycle of Fatou components in $K_i$ that are commonly incident to $p$. Let $v_{p,i,l}$ denote the center of $U_{i,p,l}$. Then, there exists $m>0$ such that for each $i$, the graph $G^{(m)}$ contains all $v_{p,i,l}$'s as vertices.
    
    Then, we can find an $f$-invariant set of internal rays of Fatou components $\{\tilde{\gamma}_{i,p,l}\}_{l=1,2,\dots,n_p}$.
 
    \item Suppose that $p\notin \bigcup_{n>0} G_i^{(n)}$ and $p$ is not on the boundary of any Fatou component in $K_i$. In \cite[Lemma 4.2]{DHS_Decomp}, the authors construct a path, possibly self-intersecting, that is periodic near $p$, by using iterative preimages of finitely many arcs in $\overline{G_i^{(1)}\setminus G_i}$. Then they use the linearization near the (repelling) fixed point $p$ in \cite[Lemma 4.2]{DHS_Decomp} to modify the path to a set of disjoint arcs $\{\tilde{\gamma}_{i,p,l}\}_{l=1,2,\dots,n_{i,p}}$ that are permuted by $f$ near $p$. Finally we take a sufficiently large $m>0$ so that $G_i^{(m)}$ intersects $\tilde{\gamma}_{i,p,l}$'s in the linearizable neighborhood of $p$. We note that $\tilde{\gamma}_{i,p,l}$'s can be chosen as arcs almost in the preimages of $G_i^{(m)}$, in the sense of Definition \ref{defn:ArcsAlmostPreimag}.
\end{enumerate}

Case (2) uses the extensions by internal rays (E2), and Case (3) uses the extension by arcs almost in the preimages (E3).

We iterate Steps~1 and~2, using the output graph $H$ from Step~2 as the input graph $G$ for the next iteration of Step~1. Each iteration reduces the number of connected components of $G$. After finitely many iterations, we obtain a graph $G$ for which a further iteration does not decrease the number of connected components. We then terminate the process and refer to the resulting graph $G$ as a {\it maximally extended graph}.

\begin{lem}\label{lem:ExtGraphConn iff Crochet}
    Let $f$ be a \pcf rational map with a non-empty Fatou set. The maximally extended graph is connected if and only if $f$ is a crochet map.
\end{lem}

The ``only if'' direction is clear. The ``if'' direction follows from \cite[Section 4.3]{DHS_Decomp}.

\subsection*{Vertex conditions}
Recall the definitions of extensions of graphs in Section \ref{subsec:GraphExtensions}.

In Case (2) of Step 2, the Fatou endpoint of an internal ray is contained in $G^{(m)}$ for all sufficiently large $m>0$. Since every $p\in P$ is either periodic or preperiodic, if $G$ satisfies (P5) and $p\in G$, then $p\in \V(G^{(m)})$ for all sufficiently large $m>0$. Therefore, for all sufficiently large $m>0$, the vertex conditions (E2-V) and (E3-V) are satisfied for $G^{(m)}$, where Cases (2) and (3) in Step 2 are regarded as (E2)- and (E3)-extensions; See Lemmas \ref{lem:(P5)Preserve}--\ref{lem:(P5) (E3-V)}.
\vspace{10pt}

\subsection*{Extension of graphs for a large iterate $f^N$}
A cluster $K_i$ is said to be {\it fixed} (by $f$) if $f(K_i)=K_i$ and {\it prefixed} if $K_j:=f(K_i)$ is fixed. 

There exists a sufficiently large $N>0$ such that, when iterating Steps 1 and 2 for $f^{N}$, we may additionally assume the following in Step 2.
\begin{enumerate}
    \item[(i)] Every cluster $K_i$ is fixed or prefixed by $f$.
    \item[(ii)] Every point $p\in P$ is fixed or prefixed by $f$.
    \item[(iii)] If $K_i$ and $K_j$ are two distinct clusters fixed by $f$ with $K_i\cap K_j\neq 0$, then there exists a fixed point $p\in K_i\cap K_j$ such that, in Case (3), there is an arc $\tilde{\gamma}_{i,p}\subset K_i$ with $f(\tilde{\gamma}_{i,p})\supset \tilde{\gamma}_{i,p}$.
\end{enumerate}
Then, in Case (3), we only add $\tilde{\gamma}_{i,p}$, which is an (E3')-extension.

Consider the case where $K_i \cap K_j \neq \emptyset$ with $i \neq j$, and $K_i$ is pre-fixed; that is, $f(K_i) = K_{i'}$ and $f(K_{i'}) = K_{i'}$. Suppose $f(K_j) = K_{j'}$, possibly with $j = j'$.

If $i' \neq j'$, then $G_{i'}$ and $G_{j'}$ are joined in Case (3) of Step 2 via an (E3')-extension; see (iii) above. Consequently, $G_i$ and $G_j$ are joined in a further iteration of Step 1.

If $i' = j'$, then there exists $p \in (K_i \cap K_j) \cap P$ such that $p' := f(p) \in K_{i'} \cap P$ is a fixed point of $f$. Hence, in Case (3) of Step 2, we add an arc along the preimage $\tilde{\gamma}_{i',p'}$, which yields an (E3')-extension. Then $G_i$ and $G_j$ are joined through $p$ by a further iteration of Step 1.

Therefore, we can assume that the (E3)-extension in Case~(3) of Step 2 is an (E3')-extension. Consequently, by Theorems \ref{thm:SepRecurr} and \ref{thm:TotalPreorderGraphs}, we obtain the following theorem.

\begin{thm}\label{thm:TotalPreorderCrochet}
    Suppose that $g$ is a crochet map. For $f:=g^N$ for some large $N>0$, we have a connected $f$-invariant graph $G$ with $P_f\subset G$ and $h_{top}(f|_G)=0$ such that
    \begin{enumerate}
        \item $G$ satisfies the properties (P1)--(P4) in Section \ref{Sec:GraphExt and Preorder}, and
        \item the finite subdivision rule $\R$ associated with $f$ with $S_\R^{(1)}=G$ has separated recurrence (Definition \ref{defn:SepRecur}).
    \end{enumerate}
    Let $\preceq^*$ denote the total preorder extension of the preorder $\preceq$ on $\Edge(S_\R)$, by using (P4). For level-$0$ edges $e_1$ and $e_2$, the following hold.
    \begin{enumerate}[resume]
        \item Suppose that $e_1$ has a level-$n$ subedge $e^n$ of type $e_2$. Then, $e_1 \preceq^* e_2$, and the equivalence $e_1 \simeq^* e_2$ holds if and only if $e^n$ is the level-$n$ recurrent subedge of $e_1$.
        \item If the Julia vertices of $e_1$ and $e_2$ lie in the same periodic cycle, then  $e_1 \simeq^* e_2$.
    \end{enumerate} 
\end{thm}
\begin{proof}
    The proof is straightforward from the discussion above. We recall that preorder extensions preserve strict orders (Definition \ref{defn:preorder}).
\end{proof}

It is worth noting that, when considering $f^N$, one may use an $f^N$-invariant proper subset of $P$, which is the one used for $f$. Even if the same set $P$ is used for $f^N$, the collection of edges and arcs added in Step 2 for $f^N$ may be smaller than that for $f$.

For example, in Figure \ref{fig:selfjoin}-(B), the set $P = \{p_1, p_2\}$ forms a $2$-cycle for $f$, and both $e_1'$ and $e_2'$ must be added simultaneously to preserve $f$-invariance. This is why (E3) may fail to preserve (P3). However, for $f^2$, the points $p_1$ and $p_2$ are fixed points, so it suffices to use only $p_1$ and add $e_1'$. Once $G_1$ and $G_2$ are joined via $e_1'$, adding $e_2'$ is no longer necessary and, indeed, is prohibited by Condition~(3) of (E3').


\section{Asymptotic conformal energies of crochet maps}\label{sec:AsympConfEngCro}
In Section \ref{sec:naturalmap}, we construct a graph virtual endomorphism $f,\psi\from H_1\to H_0$ from a finite subdivision rule $\R$. We then in Section \ref{sec:TotPreorderExpLen} define a $p$-length on $H_0$ with a certain property, called {\it $K\,\hat{}$-expanding}, by using the total preorder on $\Edge(G)$ obtained from the property (P4). In Section \ref{subsec:psi_within_tile}, we discuss a local picture of the map $\psi\from H_1 \to H_0$ restricted to each level-$0$ tile $t^0$. Following that, in Section \ref{sec:LocalDeformTrees}, we discuss a local deformation of maps $\psi$ between certain graphs to decrease $E^p_p(\psi)$. Finally in Section \ref{sec:PfCrochetImplyARCdimOne}, we complete the proof of Theorem \ref{theorem:main} by showing that the Julia sets of crochet maps have Ahlfors-regular conformal dimension one.

\subsection{Dual graph virtual endomorphisms of finite subdivision rules}\label{sec:naturalmap}

Let $\R$ be a finite subdivision rule with a subdivision map $f\from \R(S_\R)\to S_\R$ such that $P_f \subset \V_F(S_\R)$. Then $f\from (S^2,\V_F(S_\R))\righttoleftarrow$ is a \pcf branched covering of hyperbolic-type. We consider subdivision complexes $\R^n(S_\R)$'s of different levels $n$'s as different CW-complex structures on the same underlying sphere $S^2$. Recall the discussion in Section \ref{sec:DualRecSklton} on dual $1$-skeletons of finite subdivision rules.

Denote by $G_0$ the dual $1$-skeleton of $S_\R$. We can take the level-$n$ dual $1$-skeleton $G_n$ by $G_n=f^{-n}(G_0)$. There exists a unique homotopy class $[\phi^n_0]\from G_n \to G_0$ for which the diagram \eqref{eqn:phi} commutes up to homotopy. The vertical arrows in the diagram are embeddings, which are homotopy equivalences.
\begin{equation}\label{eqn:phi}
    \begin{tikzcd}
        S^2\setminus \V(\R^n(S_\R)) \arrow[hookrightarrow]{r}
        & S^2 \setminus \V(S_\R) \\
        G_n\arrow[hookrightarrow]{u}{\simeq}\arrow{r}[swap]{[\phi^n_0]}
        & G_0\arrow[hookrightarrow]{u}[swap]{\simeq}
    \end{tikzcd} 
\end{equation}
The restriction $f|_{G_1}\from G_1\to G_0$ is a covering. For $\phi:=\phi^1_0$, we call
the graph virtual endomorphism $\Gcal=(G_0,G_1,f,\phi)$ the {\it dual graph virtual endomorphism} of the finite subdivision rule $\R$. We refer to $G_n$ as the level-$n$ {\it dual graph}, or {\it dual skeleton}, of $\R$.

\subsection*{Natural representative $\phi^n_0\from G_n\to G_0$} Let $V$ be a vertex of $G_n$. There exists a level-$n$ tile $t^n$ that is dual to $V$. We write $V=V(t^n)$ and $t^n={\rm Tile}(V)$. Let $t$ be the level-$0$ tile with $t^0\supset t^n$ such that $V(t^0)\in \V(G_0)$. We define $\phi^n_0\from \V(G_n)\to V(G_0)$ by $\phi^n_0(V(t^n))=V(t^0)$.

Then, $\phi^n_0$ extends to $G_n$ in such a way that every edge is collapsed to a vertex or mapped onto an edge homeomorphically. To this end, it suffices to show that a pair of adjacent vertices of $G_n$ is mapped to a pair of adjacent vertices or the same vertex. Suppose $t^n$ and $t'^n$ are level-$n$ adjacent tiles. If $t^n,t'^n\subset t^0$ for some level-$0$ tile $t$, then $\phi^n_0\from V(t^n),V(t'^n)\mapsto V(t^0)$. If $t^n\subset t^0$ and $t'^n\subset t'^0$ with $t^0\neq t'^0$, then $t^0$ and $t'^0$ are also adjacent, and thus $\phi^n_0(V(t^n))=V(t^0)$ and $\phi^n_0(V(t'^n))=V(t'^0)$ are adjacent vertices in $G_0$.

We call this map $\phi^n_0\from G_n \to G_0$ the {\it natural representative} of the homotopy class $[\phi^n_0]$ of the dual graph virtual endomorphism of $\R$. The same construction applies to obtain a natural representative $\phi^n_m\from G_n \rightarrow G_m$ of $[\phi^n_m]$ for all $n>m\ge 0$.

Recall that we assume $P_f\subset \V_F(S_\R)$ and consider the branched covering $f\from (S^2,\V_F(S_\R))\righttoleftarrow$, which is of hyperbolic-type. If $\V(S_\R) \supsetneq \V_F(S_\R)$, the dual graph $G_0$ is not a spine of of the punctured sphere $S^2 \setminus \V_F(S_\R)$. Hence, $f,\phi\from G_1\to G_0$ is not a suitable graph virtual endomorphism of $f\from (S^2,\V_F(S_\R))\righttoleftarrow$. Next, we describe how to deform $G_0$ into a spine of $S^2 \setminus \V_F(S_\R)$ via the so-called {\it $\origboxtimes$-deformation}.

\subsection*{$\origboxtimes$-deformation and $\origboxtimes$-dual graphs}
Suppose that $\R$ is a finite subdivision rule such that every edge is an $FJ$-edge, i.e., $S_\R^{(1)}$ satisfies (P1).

Consider the dual graph $G_0$ of $S_\R$. For each Julia vertex $v\in \V_J(S_\R)$, we deform the level-$0$ link $L^0(v)(\subset G_0)$ as follow:  Consider a level-$0$ tile $t$ with $v\in \V(t)$ such that $\phi_t\from \bft \to t$ is its characteristic map. Since $\R$ has $FJ$-edges only, so does $\bft$. The dual $1$-skeleton, denoted by $G_0(\bft)$, of $\bft$ is a starlike tree with the center vertex $V(\bft)$, whose endpoints are placed at the midpoints of edges of $\bft$. Consider a Julia vertex $\bf{v}$ of $\bft$, and let $\bfe_1$ and $\bfe_2$ be its incident edges of $\bft$, which must be $FJ$-edges. We replace the two edges, denoted by $E(\bft,\bfe_1)$ and $E(\bft,\bfe_2)$, of $G_0(\bft)$ that are dual to $\bfe_1$ and $\bfe_2$ with a new edge $E({\bf t}, {\bf v})$ connecting $\bf{v}$ to the dual vertex $V(\bft)$ of $\bft$.

We perform this for each Julia vertex ${\bf v}$ of $\bft$. Then, $G_0(\bft)$ is changed into a starlike tree $H_0(\bft)$ whose every edge joins $V(\bft)$ to a Julia vertex ${\bf v}$. See Figures \ref{fig:Orig Hexa}--\ref{fig:FJFJ}.

\begin{figure}
	\centering
	\def\svgwidth{0.7\textwidth}
	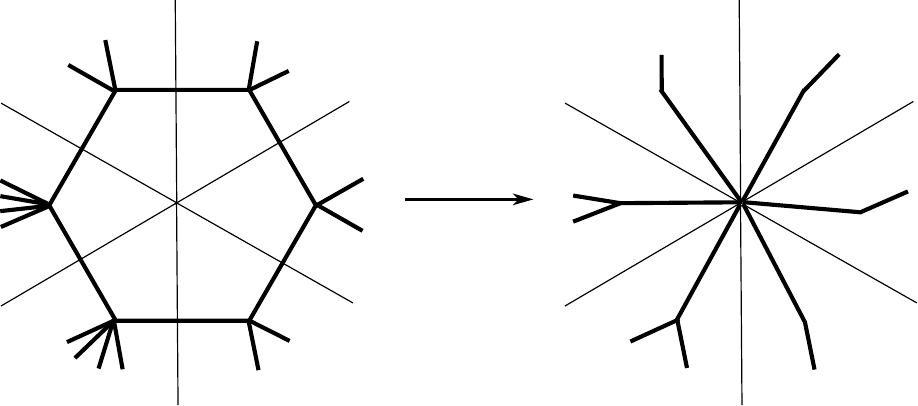
	\caption{The $\origboxtimes$-deformation near a Julia vertex $v$ with $\deg_{S
 _\R}(v)=6$.}
    \label{fig:Orig Hexa}
\end{figure}

We refer to this deformation as a {\it $\origboxtimes$-deformation}. If $\deg_{S_\R}(v)=4$ and $L^0(v)$ is an embedded square, this deformation changes the square to the cross in the symbol $\origboxtimes$. We emphasize that this deformation works in the absence of $JJ$-edges, which follows from (P1) of $S^{(1)}_\R$.

If the link $L^0(v)$ is homeomorphic to a circle, then the $\origboxtimes$-deformation replaces $L^0(v)$ with a starlike tree.

Let $H_0$ denote the graph obtained by applying the $\origboxtimes$-deformation to $G_0$. For any $n\ge 1$, we define $H_n:=f^{-n}(H_0)$, which is also equivalent to the $\origboxtimes$-deformation of $G_n$. We refer to $H_n$ as the level-$n$ {\it $\origboxtimes$-dual graph}, or {\it $\origboxtimes$-dual skeleton}, of $\R$.

It is easy to show that $H_0$ is a deformation retract of $\hCbb \setminus \V_F(S_\R)$. Hence there is a homotopy class $[\psi^n_0]$ of continuous maps $\psi^n_0 \from H_n \to H_0$ such that the following diagram commutes up to homotopy
\begin{equation}\label{eqn:DiagramDualGraphEndo}
    \begin{tikzcd}
        \hCbb\setminus \V_F(\R^n(S_\R)) \arrow[hookrightarrow]{r} & \hCbb \setminus \V_F(S_\R) \\
        H_n\arrow[hookrightarrow]{u}{\simeq}\arrow{r}[swap]{[\psi^n_0]} & H_0\arrow[hookrightarrow]{u}[swap]{\simeq}
    \end{tikzcd}
\end{equation}
where the vertical arrows are homotopy equivalences. The natural representative of $[\psi^n_0]$ will be discussed at the end of this subsection. See Figure \ref{fig:BasilicaFSR} for a finite subdivision rule of the Basilica polynomial and Figure \ref{fig:BasilicaFSRwithDualGraphwithLength} for its $\origboxtimes$-dual graphs $H_0$ and $H_1$.

\begin{rem}\label{rem:Homotopy Rel V_J}
The homotopy class of $[\psi^n_0 \from H_n \to H_0]$ can also be considered relative to $\V_J(S_\R) (\subset \V(H_0)\subset \V(H_n)$).    
\end{rem}

\begin{figure}
    \centering
        \def\svgwidth{0.8\textwidth}
        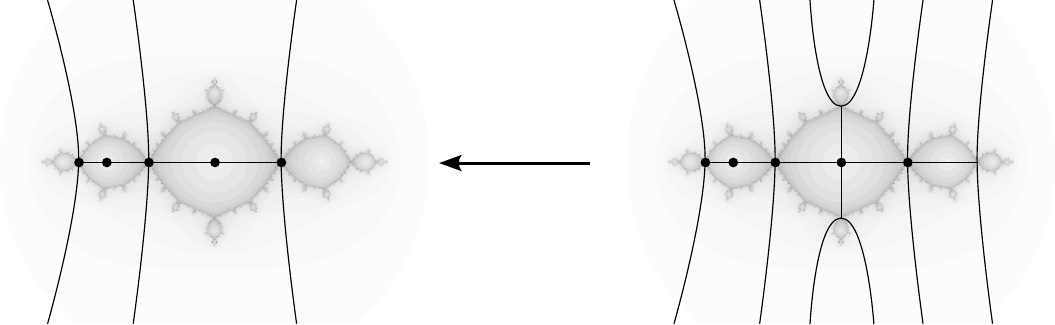
    \caption{A finite subdivision rule $\R$ of the Basilica polynomial $z^2-1$. Every edge is an $FJ$-edge, and $\R$ has separated recurrence.}
    \label{fig:BasilicaFSR}
\end{figure}

\begin{figure}
    \centering
        \def\svgwidth{0.9\textwidth}
        {\tiny
        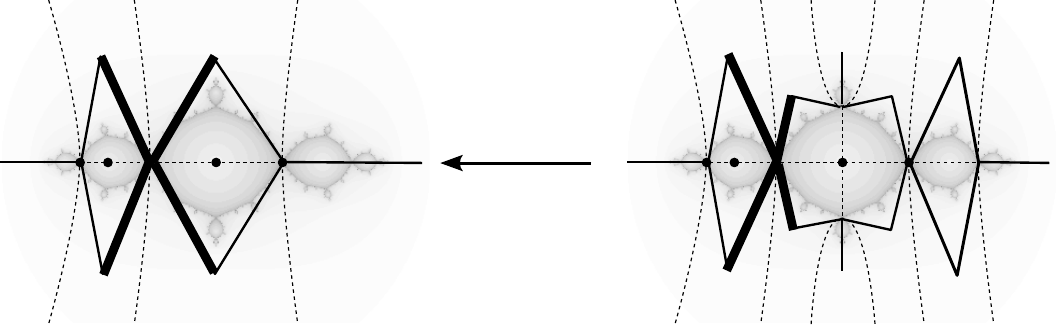}
    \caption{Level-$0$ and level-$1$ $\origboxtimes$-dual skeletons $H_0$ and $H_1$, with $K\,\hat{}$-expanding $p$-lengths labeled for $K=10$. The edges incident to periodic Julia vertices are depicted in bold. Under the natural representation $\phi^1_0$, the two tripods labeled with $10$'s are collapsed to vertices, and the rightmost tripod labeled with $10^2$'s is collapsed to the point at infinity.}
    \label{fig:BasilicaFSRwithDualGraphwithLength}
\end{figure}

\subsection*{Vertices and edges of $\origboxtimes$-dual graphs}
Let us discuss the vertices and edges of $H_n$ for $n\ge0$. Recall that, to perform the $\origboxtimes$-deformation, we assume (P1) such that every edge of $S_\mathbb{R}$ is an $FJ$-edge.

We define the vertex set $\V(H_n)$ by 
\[
    \V(H_n)=\V_J(S_\R)\cup V({\rm \{level\mhyphen n~Tiles\}}),
\]
where $V({\rm \{level\mhyphen n~Tiles\}})$ is the set of the dual vertices $V(t^n)$ of the level-$n$ tiles $t^n$. The edges of $H_n$ are defined accordingly. Then, for each edge $E$ in $H_n$, one of its endpoints is $V(t^n)$, where $t^n$ is a level-$n$ tile with $E\subset t^n$, and the other is a Julia vertex $v^n\in \V_J(\R^n(S_\R))$.

Recall the definitions of sides of edges and corners at vertices, which are discussed in Section~\ref{sec:FSR}. There is a natural bijection
\[
    \begin{array}{ccc}
    \left\{ (v^n,c) : v^n \in \V_J(\R^n(S_\R))~ {\rm and}~ c\in {\rm Corner}^n(v^n)\right\} & \leftrightarrow  & \left\{ {\rm edges~ of~} H_n\right\}\\
    (v^n,c) & \leftrightarrow & E(v^n,c)
    \end{array}
\]
that maps $(v^n,c)$ to a unique edge $E(v^n,c)$ of $H_n$ satisfying (i) its endpoints are $v^n$ and the dual vertex of ${\rm Tile}(v^n,c)$ and (ii) it intersects the connected component of $D(v^n)\setminus \R^n(S_\R)^{(1)}$ that is represented by $c$, where $D(v^n)$ is a small disk neighborhood of $v^n$.
Note that (i) is not enough to uniquely determine an edge of $H_n$ when ${\rm Tile}(v^n,c)$ is multiply incident to $v^n$.

\subsection*{Natural representatives of $\origboxtimes$-dual graphs}
Suppose that $\R$ is a finite subdivision such that $S_\R^{(1)}$ satisfies (P1)--(P4). Then, by Lemma \ref{lem:P1P2RecFF Edge}, we have $h_{\mathrm{top}}(f|_{S_\R^{(1)}})=0$.
Recall the natural representative $\phi^n_0\from G_n \to G_0$ for dual skeletons $G_n$'s of $\R$. The $\origboxtimes$-deformations naturally induces a map $\psi^n_0\from H_n \to H_0$, which maps each edge to either an edge homeomorphically or a vertex. We now describe this map $\psi^n_0$. The notation $\preceq^*$ denotes the total preorder introduced in Theorem \ref{thm:TotalPreorderCrochet}.
\medskip

\begin{figure}
    \centering
        \def\svgwidth{0.7\textwidth}
        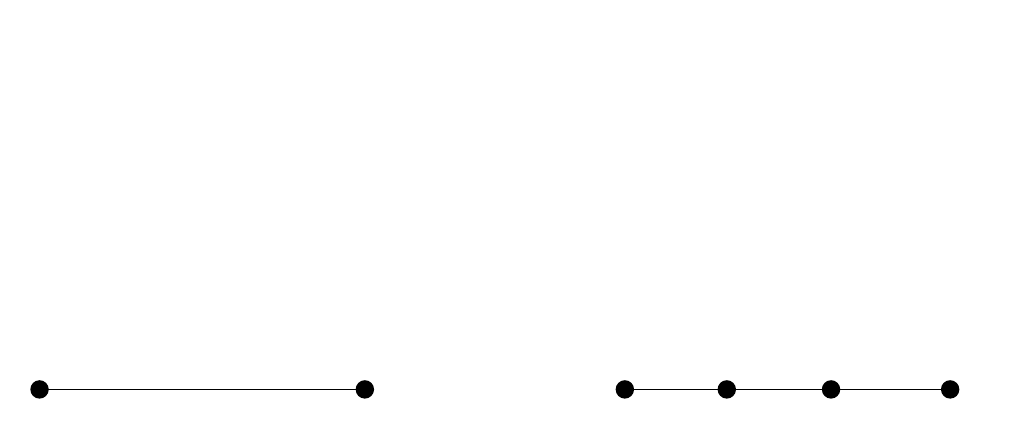
    \caption{$S_\R$ and $\R^n(S_\R)$ are shown in thin solid lines. Fatou and Julia vertices are labeled $F$ and $J$, respectively. $H_0$ and $H_n$ are shown in thick solid lines. The edges and vertices of $H_n$ contained in the area bounded by the dotted curve are mapped to the same vertex of $H_0$ (Case (J-0)). Four edges of Case (J-1) and two edges of Case (J-2) are drawn.}
    \label{fig:FJFJ}
\end{figure}

Consider the edge $E(v^n,c)$ of $H_n$ represented by $v^n \in \V_J(\R^n(S_\R))$ and $c\in {\rm Corner}^n(v^n)$. Let $t^n={\rm Tile}(v^n,c)$ and let $t^0$ be the level-$0$ tile with $t^0 \supset t^n$. Note that the vertex $V(t^n)\in \V(H_n)$, which is dual to $t^n$, is an endpoint of $E(v^n,c)$. We define $\psi^n_0(V(t^n))=V(t^0)\in \V(H_0)$. There are three cases. See Figure \ref{fig:FJFJ}.
\begin{itemize}
    \item[(J-0)] If $v^n \in {\rm int}(t^0)$, then $\psi^n_0( E(v^n,c) )=V(t^0) \in \V(H_0)$.
    \item[(J-1)] Suppose $v^n \in {\rm int}(e^0)$ where $e^0$ is a boundary edge of $t^0$. By (P1), $e^0$ is an $FJ$-edge. Let $v^0$ denote its Julia vertex $v^0$. Then $\psi^n_0(v^n)=v^0$ such that
        \[
            \psi^n_0\from E(v^n,c) \mapsto E(v^0,c'),
        \]
    where $c'\in {\rm Corner}^0(v^0)$ is the image of $c$ under
    \[
        {\rm Corner}^n(v^n) \twoheadrightarrow {\rm Side}(e^0) \hookrightarrow {\rm Corner}^0(v^0).
    \]
    Let $e^n$ be a level-$n$ edge incident to $v^n$ and $e'^0 := f^n(e^n)\in \Edge(S_\R)$. Then, $e^0 \precneq^* e'^0$. To see this, suppose for contradiction that $e^0 \simeq^* e'^0$. In that case, $e^n$ would be a recurrent subedge of $e^0$, and, by Lemma \ref{lem:P1P2RecFF Edge}, $v^n = v^0$ would be a periodic point, contradicting the assumption that $v^n \in {\rm int}(e^0)$.
    
    \item[(J-2)] Suppose $v^n =v^0 \in \V_J(S_\R)$. Then $\psi^n_0(v^n)=v^0=v^n$ such that
    \[
        \psi^n_0\from E(v^n,c)\mapsto E(v^0,c'),
    \]
    where $c'$ is the image of $c$ under
    \[
        {\rm Corner}^n(v^n) \twoheadrightarrow {\rm Corner}^0(v^0).
    \]
    Let $e^n$ be a level-$n$ edge incident to $v^n$ and $e'^0 := f^n(e^n)\in \Edge(S_\R)$. Then $e^0 \preceq^* e'^0$. The equivalence $e^0 \simeq^* e'^0$ holds if and only if $v^0=v^n$ is a periodic Julia vertex (Lemma \ref{lem:P1P2RecFF Edge}).
\end{itemize}

It is not difficult to show that the edge-wisely defined map $\psi^n_0\from H_n \to H_0$ is continuous and in the desired homotopy class in the diagram \eqref{eqn:DiagramDualGraphEndo}. We call the map $\psi^n_0$ defined above the {\it natural representative} in its homotopy class. We emphasize that the natural representative $\psi^n_0$ maps each edge either homeomorphically onto an edge or to a vertex. If $p$-lengths are assigned to $H_0$ and $H_n$, we define the natural representative so that it has a constant derivative on the interior of each edge.

\subsection{Total preorder on ${\rm Edge}(H_0)$ and $K\,\hat{}$-expanding $p$-lengths of $H_0$}\label{sec:TotPreorderExpLen}
Suppose that $\R$ is a finite subdivision rule such that $S_\R^{(1)}$ satisfies (P1)--(P4) and $h_{top}(f|_{S_\R^{(1)}})=0$
 
Consider the following diagram.
\begin{equation}\label{eqn:EdgeSets}
    \begin{tikzcd}
        & \Edge(S_\R)=\Edge_{FJ}(S_\R) \arrow[twoheadrightarrow]{d} \\
        \Edge(H_0)\arrow[twoheadrightarrow]{r}& \V_J(S_\R)
    \end{tikzcd}
\end{equation}
The surjection $\Edge_{FJ}(S_\R) \twoheadrightarrow \V_J(S_\R)$ sends each $FJ$-edge to its Julia vertex.

Let us begin with the total preorder $\preceq^*$ on $\Edge(S_\R)$ defined in Theorem \ref{thm:TotalPreorderCrochet}. By (4) of Theorem \ref{thm:TotalPreorderCrochet}, it induces a total preorder on $\V_J(S_\R)$. By considering the surjection $\Edge(H_0) \twoheadrightarrow \V_J(S_\R)$ as a quotient map and Definition \ref{defn:PreorderQuotient}, we obtain a total preorder on $\Edge(H_0)$, which we also denote by $\preceq^*$.   

For a real number $K>1$, we say that the $p$-length $\beta$ on $H_0$ is {\it $K\,\hat{}$-expanding} if it satisfies the following set of three conditions, denoted by ($K\,\hat{}$-exp):
\begin{enumerate}
    \item $K\le \beta(E)$ for any $E \in \Edge(H_0)$.
    \item $\beta(E(v_1,c_1))=\beta(E(v_2,c_2))$ if $v_1$ and $v_2$ lie in the same periodic cycle.
    \item $K\,\beta(E(v_1,c_1)) \le \beta(E(v_2,c_2))$ if $E(v_1,c_1) \precneq^* E(v_2,c_2)$.
\end{enumerate}
By ($K\,\hat{}$-exp)-(1), for each $E\in \Edge(H_0)$, we have $\lim_{K\to \infty}\beta(E)=\infty$.

Define a function $r\from \Edge(S_\R) \to \Zbb_{\ge0}$ as the minimal function satisfying: (a) $r(e)=r(e')$ if $e\simeq^* e'$, and (b) $r(e)<r(e')$ if $e\precneq^* e'$. We define a $p$-length $\alpha_K$ on $H_0$ by
\begin{align}\label{eqn:Alpha_K}
        \alpha_K(E(v,c))&:=K^{r(e_v)},
\end{align}
where $e_v$ is any edge in $S_\R$ incident to $v\in \V(S_\R)$. By Theorem \ref{thm:TotalPreorderCrochet}, the above inequalities are independent of the choice of edges $e_{v_i}$ incident to each vertex $v_i$. By definition, $\alpha_K$ is $K\,\hat{}$-expanding.

By abusing notation, we also denote by $\alpha_K$ the $p$-length on $H_n$ that is the lifting of $\alpha_K$ on $H_0$ through the covering map $f^n\from H_n \to H_0$.

\subsection{The restriction $\psi|_{H_1\cap t_0}$ to each level-$0$ tile $t_0$}\label{subsec:psi_within_tile}
Let $t^0$ be a level-$0$ tile. We first consider the case where $t^0$ is singly incident to its Julia vertices; the multiply incident case will be discussed at the end of this subsection. Then 
\[
H_0(t^0) := H_0 \cap t^0
\] 
is a starlike tree whose endpoints are the Julia vertices of $t^0$, and whose center vertex is the dual $V(t^0)$ of $t^0$. Its edges are of the form $E(v,c)$, where $v$ ranges over the Julia vertices of $t^0$ and $c^0 \in {\rm Corner}^0(v)$ satisfies ${\rm Tile}(v,c^0) = t^0$.

Next, consider the graph 
\[
H_1(t^0) := H_1 \cap t^0.
\] 
Its edges are of the form $E(v^1,c^1)$, where $v^1$ are the level-$1$ Julia vertices contained in $t^0$ and $c^1 \in {\rm Corner}(v^1)$ such that ${\rm Tile}(v^1,c^1)$ is a level-$1$ subtile of $t^0$. Recall the natural representative $\psi \colon H_1 \to H_0$ and the three cases (J-0), (J-1), and (J-2) discussed in Section~\ref{sec:naturalmap}.

We have $v^1 \in {\rm int}(t^0)$ if and only if $\psi(E(v^1,c^1)) = V(t^0)$.  

Suppose $v^1 \in \partial t^0$. Then 
\[
    \psi(E(v^1,c^1)) = E(v^0,c),
\] 
where $v^0$ is the Julia endpoint of a level-$0$ edge $e^0$ with $v^1 \in e^0$. Moreover, the corner $c^1$ maps to $c$ under the composition
\[
    {\rm Corner}^1(v^1) \to {\rm Side}(e^0) \to {\rm Corner}^0(v^0)
\]
when $v^1 \in {\rm int}(e^0)$ (Case~(J-1)), and under
\[
    {\rm Corner}^1(v^1) \to {\rm Corner}^0(v^0)
\]
when $v^1 = v^0$ (Case~(J-2)).

Now consider the $K\,\hat{}$-expanding $p$-length $\alpha_K$ defined in Section \ref{sec:TotPreorderExpLen}. Suppose $ \psi(E(v^1,c^1)) = E(v^0,c^0)$.

If $v^0 = v^1$ is a periodic Julia vertex, then
\[
    \frac{\alpha_K(E(v^1,c^1))}{\alpha_K(E(v^0,c^0))} = 1.
\]
Otherwise, let $e^1$ be a subedge of $e^0$ incident to $v^1 \in {\rm int}(e^0)$. Then $v' \notin \V(S_\R)$ and hence cannot be a periodic point of $f$. By Lemma~\ref{lem:P1P2RecFF Edge}, $e^1$ is not a recurrent subedge of $e^0$. Hence, if $e^1$ is of type $e'^0$, then $e'^0 \precneq^* e^0$, and therefore
\[
    \frac{\alpha_K(E(v^1,c^1))}{\alpha_K(E(v^0,c^0))} \le \frac{1}{K}.
\]

Suppose $t^0$ is multiply incident to edges or Julia vertices, and let $\phi_{t^0}\colon \bft^0 \to t^0$ denote the characteristic map of $t^0$. Then we can pull back $H_0(t^0)$ and $H_1(t^0)$ to the polygon $\bft^0$ via $\phi_{t^0}$, obtaining graphs $H_0(\bft^0)$ and $H_1(\bft^0)$ in $\bft^0$ that are analogous to the case when $t^0$ is singly incident to its Julia vertices. The same arguments can then be applied to $H_0(\bft^0)$ and $H_1(\bft^0)$.

\subsection{$(\epsilon,\epsilon')$-Deformation}\label{sec:LocalDeformTrees}
In this subsection, we discuss a model of the map $\psi \colon H_1(t^0) \to H_0(t^0)$, which is described in Section~\ref{subsec:psi_within_tile}, with the goal of constructing a deformation that reduces the $p$-conformal energy $E^p_p(\psi)$ to a value strictly less than one.

For a graph $G$, a {\it leaf} is a vertex $v$ with $\deg_G(v)=1$. We refer to an edge $e$ as a {\it boundary} edge is $e$ contains a leaf, and as {\it non-boundary} edge otherwise. The set of leaves is equal to the boundary $\partial G$ of $G$.

Consider a starlike tree $H$ with $\Edge(H)=\{e_0,e_1,\dots,e_k\}$ and the center vertex $v$. Denote by $w_i$ the leaf with $w_i\in e_i$.

Let $G$ be a connected finite graph such that its set of boundary edges is the union $E_0 \cup E_1 \cup \dots \cup E_k$, where $E_i=\{e_{i,1},e_{i,2},\dots,e_{i,p_i}\}$ with $p_i\ge1$ for $1\le i\le k$. Denote by $w_{i,j}$ and $v_{i,j}$ the endpoints of $e_{i,j}$ such that $w_{i,j}$ is a leaf of $G$. See Figure \ref{fig:deformation}.

A map $\psi\from (G,\partial G)\to (H,\partial H)$ is defined in such a way that (a) non-boundary edges are mapped to the center vertex $v$ and (b) each boundary edge $e_{i,j}$ is homeomorphically mapped onto $e_i$ with $\psi(v_{i,j})=v$ and $\psi(w_{i,j})=w_i$ for all $0\le i\le k$.

Fix $K,p\in(1,\infty)$. Suppose $\alpha$ and $\beta$ are $p$-lengths on $G$ and $H$ such that
\begin{enumerate}
    \item  $\alpha(e')\ge K$ and $\beta(e)\ge K$  for all $e'\in {\rm Edge}(G)$ and $e\in {\rm Edge}(H)$,
    \item $\alpha(e_{i,j})\ge K\, \beta(e_i)$ for all $i\ge 1$ and $j\in \{1,2,\dots p_i\}$,
    \item $\alpha(e_{0,j})\ge K\, \beta(e_0)$ for all $j\in \{2,\dots,p_0\}$, and
    \item $\alpha(e_{0,1})\ge \beta(e_0)$.
\end{enumerate}
We additionally assume that $\psi'$ is constant on ${\rm int}(e')$ for each $e'\in \Edge(G)$, with respect to $\alpha$ and $\beta$.

We consider the homotopy class of $[\psi\from (G,\partial G)\to (H,\partial H)]$ relative to the boundary $\partial G$.

\begin{lem}\label{lem:EpsilonDeform}
    In the setting above, there exists $K_0:=K_0(k,\max_i{p_i})>0$ such that for all $K>K_0$, we have
    \[
        E^p_p[\psi\from ((G,\partial G),\alpha)\to ((H,\partial H),\beta)]<1.
    \]
\end{lem}
\begin{proof}
Recall that for $p\in(1,\infty)$
\begin{align*}
    \textup{Fill}^p_p(\psi)(y) &= \sum_{x\in\psi^{-1}(y)}\psi'(x)^{p-1},~{\rm and}\\
    E^p_p(\psi) &=(\|\textup{Fill}^p_p(\psi)\|_{{\rm ess},\infty})^{1/p}.
\end{align*}

For the map $\psi\from G\to H$ defined above, we have
\[
    \textup{Fill}^p_p(\psi)(y)=\sum_{j=1}^{p_i} \frac{\beta(e_i)}{\alpha(e_{i,j})}\le 
    \left\{
    \begin{aligned}
         &\frac{p_i}{K}&& {\rm if}~y\in e_i ~{\rm and}~i>0 \\
        &\frac{p_1-1}{K}+\frac{\beta(e_0)}{\alpha(e_{0,1})} && {\rm if}~y\in e_0
    \end{aligned}
    \right.
\]
Recall $\alpha(e_{0,1}) \ge \beta(e_0)$. Since we are considering an upper bound for ${\rm Fill}^p_p$, we only need to consider the case when $\alpha(e_{0,1}) = \beta(e_0)$. In our application, we will have either $\alpha(e_{0,1}) = \beta(e_0)$ or $\alpha(e_{0,1}) \ge K\, \beta(e_0)$

Assume $\alpha(e_{0,1}) = \beta(e_0)$. Then $|\textup{Fill}^p_p(\psi)|_{{\rm ess},\infty} = 1 + \frac{p_1 - 1}{K}$. This value is slightly larger than 1 if $K$ is very large. We deform $\psi$ to $\psi_{(\epsilon,\epsilon')}$ so that $|\textup{Fill}^p_p(\psi_{(\epsilon,\epsilon')})|_{{\rm ess},\infty}<1$.

\begin{figure}
    \centering
        \def\svgwidth{0.8\textwidth}
        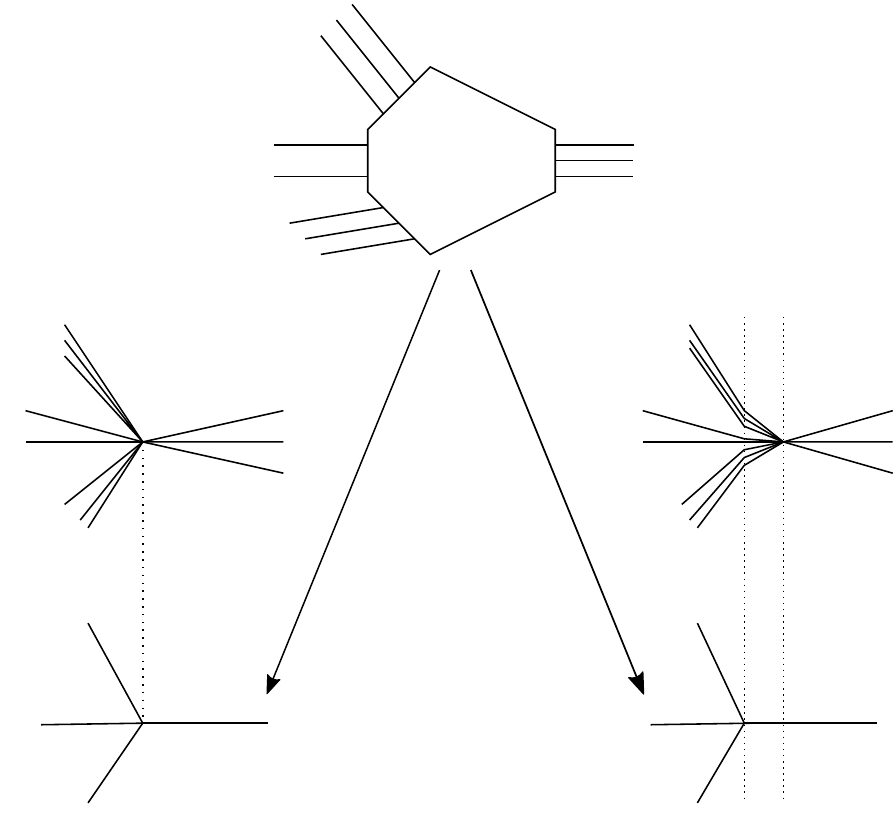
    \caption{The $(\epsilon,\epsilon')$-deformation. Graphs in the middle row are obtained from $G$ by collapsing $Z$ to a point.}
    \label{fig:deformation}
\end{figure}

Define $Z:=\psi^{-1}(v)$. The {\it $(\epsilon,\epsilon')$-deformation $\psi_{(\epsilon,\epsilon')}$} is, roughly speaking, pulling the map on $Z$, $\psi|_Z\from Z \to \{v\}$, toward the edge $e$. See Figure \ref{fig:deformation}. For a precise construction, let us take positive small real numbers $\epsilon,\epsilon'>0$. Let $v_\epsilon$ be the point in ${\rm int}(e_0)$ with $\beta(v,v_\epsilon)=\epsilon$ and $v_{i,j,\epsilon'}$ be the point in $e_{i,j}$ with $\alpha(v_{i,j},v_{i,j,\epsilon'})=\epsilon'$ for all $i\in \{1,2,\dots, k\}$ and $j\in\{1,2,\dots,p_i\}$. Here $p$-lengths $\alpha$ and $\beta$ are used as metrics on graphs. The deformation $\psi_{(\epsilon,\epsilon')}$ is defined in such a way that
\begin{enumerate}
    \item $\psi_{(\epsilon,\epsilon')}(Z)=v_\epsilon$,
    \item $\psi_{(\epsilon,\epsilon')}(v_{i,j,\epsilon'})=v$ for all $i\in \{1,2,\dots,k\}$, and
    \item $\psi_{(\epsilon,\epsilon')}'$ is locally constant on $G\,\setminus \left(Z\cup \{v_{i,j,\epsilon'}:1\le i\le k,\, 1\le j\le p_i\}\right)$.
\end{enumerate}

Let us evaluate $\textup{Fill}^p_p(\psi_{(\epsilon,\epsilon')})(y)$. Let $p_{\max}:=\max_i p_i$. For two points $w,w'$ on the same edge, we use interval notations $(w,w')$, $[w,w']$, etc., to denote the corresponding intervals embedded in the edge.

\noindent{\it Case 1: $y \in (v,v_\epsilon)$}
\[
    \textup{Fill}^p_p(\psi_{(\epsilon,\epsilon')})(y) = \sum\limits_{i=1}^k\sum\limits_{j=1}^{p_i} \left(\frac{\beta([v,v_\epsilon])}{\alpha([v_{i,j},v_{i,j,\epsilon'}])}\right)^{p-1}
    \le kp_{\max} \left(\frac{\epsilon}{\epsilon'}\right)^{p-1}.
\]

\noindent{\it Case 2: $y \in (v_\epsilon,v_0)$.}
For all $0<\epsilon<1-\frac{1}{\sqrt{2}}$, we have
\begin{align*}
    \textup{Fill}^p_p(\psi_{(\epsilon,\epsilon')})(y) & \le \left(\frac{\beta(e_0)-\epsilon}{\alpha(e_{0,1})}\right)^{p-1}+ \sum_{j=2}^{p_0}\left(\frac{\beta(e_0)-\epsilon}{\alpha(e_{0,j})}\right)^{p-1}\\
    & \le 1-(p-1)\frac{\epsilon}{\alpha(e_{0,1})}+|p-1||p-2|\left(\frac{\epsilon}{\alpha(e_{0,1})}\right)^2+\frac{p_{\max}}{K^{p-1}}.
\end{align*}
We used the following error bound of Taylor's expansion: For all $0<x< 1-\frac{1}{\sqrt{2}}$, which satisfies $(1-x)^{p-3}\le 2$ for all $p>1$, we have
\begin{align*}
    (1-x)^{p-1} &\le 1-(p-1)x+\frac{|(p-1)(p-2)|}{2}\max_{y\in [0,x]}|1-y|^{p-3}x^2\\
    &\le 1-(p-1)x+|(p-1)(p-2)| x^2. 
\end{align*}

\noindent{\it Case 3: $y\in {\rm int}(e_i)$ with $i>0$.} For all $0<\epsilon'<K/2$, we have
\[
    \textup{Fill}^p_p(\psi_{(\epsilon,\epsilon')})(y) = \sum\limits_{j=1}^{p_i} \left(\frac{\beta(e_i)}{\alpha(e_{i,j})- \epsilon'}\right)^{p-1}
    \le p_{\max}\frac{1}{(K/2)^{p-1}}.
\]
\vspace{5pt}

We can choose $\epsilon$, $\epsilon'$, and $K$ independently of $k$ and $p_{\max}$ such that $\textup{Fill}^p_p(\psi_{(\epsilon,\epsilon')})(y) < 1$ in all cases.
\end{proof}

\begin{rem}\label{rem:GenEpsilonDeform}
    The same argument works for the case when some leaves of $G$ and $H$ are identified, respectively. By the continuity $\psi$, if some leaves $w_{i_1,j_1}$ and $w_{i_2,j_2}$ of edges $e_{i_1,j_1}\in E_{i_1}$ and $e_{i_2,j_2}\in E_{i_2}$ with $i_1\neq i_2$ are identified in $H$, then we identify $v_{i_1}$ with $v_{i_2}$ in $G$ accordingly.
\end{rem}

\subsection{Proof of ARC.dim(the Julia set of a crochet map)=1}\label{sec:PfCrochetImplyARCdimOne}
Suppose $f$ is a crochet map. For every $N>0$, since $\Jcal_f=\Jcal_{f^N}$, we may replace $f$ with $f^N$. Let $G$ and $\R$ be an $f$-invariant graph and the finite subdivision rule with $S_\R=G$ whose subdivision map is $f$ obtained by Theorem \ref{thm:TotalPreorderCrochet}, i.e., we assume that
\begin{enumerate}
    \item $G=S_\R^{(1)}$ satisfies (P1)--(P4) in Section \ref{sec:P1-P4}, and 
    \item $\R$ has separated recurrence (Definition \ref{defn:SepRecur}, Theorem \ref{thm:SepRecurr}).
\end{enumerate}

Let $H_0$ and $H_1$ be the level-$0$ and level-$1$ $\origboxtimes$-dual skeletons of $\R$ such that $\psi\from H_1 \to H_0$ is the natural representative. Recall the description of the natural representative $\psi$ in Section~\ref{sec:naturalmap} and also in Section~\ref{sec:TotPreorderExpLen} regarding $K\,\hat{}$-expanding $p$-lengths. For each level-$0$ tile $t^0$, we will perform the $(\epsilon,\epsilon')$-deformation, which is described in Lemma \ref{lem:EpsilonDeform}, to the restriction $\psi|_{H_1\cap t_0}\from H_1\cap t_0\to H_0 \cap t_0$.

Let us first consider the case when $t^0$ is singly incident to every Julia vertex. Then $H_0 \cap t^0$ is a starlike tree, and the map $\psi \colon H_1 \cap t^0 \to H_0 \cap t^0$ satisfies the assumptions for $\psi \colon G \to H$ in Lemma \ref{lem:EpsilonDeform}, except that some leaves of edges in the same set $E_i$ (as in the setting of Lemma \ref{lem:EpsilonDeform}) may be identified. See Figure \ref{fig:FJFJ} and Remark \ref{rem:GenEpsilonDeform}.

If $t^0$ is multiply incident to some Julia vertices, then we are in the generalized situation mentioned in Remark \ref{rem:GenEpsilonDeform}, where some leaves of edges in $E_{i_1}$ and $E_{i_2}$ with $i_1\neq i_2$ are identified.

Therefore, for every fixed $p>1$, by applying Lemma \ref{lem:EpsilonDeform} to each restriction $\psi|_{H_1\cap t_0}\from H_1\cap t_0\to H_0 \cap t_0$, we can find $\psi_1\sim \psi$ such that $E^p_p(\psi_1)<1$. Note that, in this application of Lemma \ref{lem:EpsilonDeform}, the constants $k$ and $p_{\rm max}$ are determined by $\R$.
\hfill $\Box$




\subsection{Notes on the non-hyperbolic case}\label{sec:non-hyp}
Suppose $f$ is a non-hyperbolic crochet map, and let $V$ be a set of marked points of $f$ such that $V_F$ and $V_J$ are non-empty subsets of its Fatou and Julia points, respectively. We then consider a graph virtual endomorphism as follows, following the approach suggested in \cite[pp.\@ 10--11]{ThuPosChac}.

Suppose that $H_0$ is a graph in $\hat{\mathbb{C}}$ such that $V_J \subset H_0$ and $H_0$ is a deformation retract of $\hat{\mathbb{C}} \setminus V_F$ relative to $V_J$. For any $n \ge 1$, define $H_n := f^{-n}(H_0)$. By an argument similar to the hyperbolic case, we obtain a graph virtual endomorphism $\Gcal = (H_0, H_1, f, \psi)$. The differences are that $f$ is an orbigraph covering, with respect to certain orbigraph structures on $H_0$ and $H_1$, and that we consider the homotopy class $[\psi^n_0 \colon H_n \to H_0]$ relative to $V_J$; see Remark \ref{rem:Homotopy Rel V_J}.

The author expects that the basic properties of asymptotic conformal energies for hyperbolic cases continue to hold in the non-hyperbolic generalization. This generalization has not yet been rigorously developed in the literature, though it is mentioned in \cite[Section~3.4]{Thu_RubberBand}, \cite{ThuElaGraph}, and \cite[p.~10]{ThuPosChac}.

The construction of $\origboxtimes$-dual skeletons also applies to the non-hyperbolic case, yielding the desired graph virtual endomorphism; see Remark~\ref{rem:Homotopy Rel V_J}. Hence, the proof is expected to remain valid under this generalized definition of conformal energies.

Furthermore, the author anticipates that the monotonicity of the conformal energies in Theorem~\ref{thm:MonoConfEngDim} could also be established with minimal modification. Achieving this would require a suitable generalization of the results in \cite{PilThu_ARConfDim} to include non-hyperbolic rational maps.

\bibliography{Cdimone}
\bibliographystyle{amsalpha}

\end{document}

%% file: RabbitandBasilica.pdf_tex
\begingroup%
  \makeatletter%
  \providecommand\color[2][]{%
    \errmessage{(Inkscape) Color is used for the text in Inkscape, but the package 'color.sty' is not loaded}%
    \renewcommand\color[2][]{}%
  }%
  \providecommand\transparent[1]{%
    \errmessage{(Inkscape) Transparency is used (non-zero) for the text in Inkscape, but the package 'transparent.sty' is not loaded}%
    \renewcommand\transparent[1]{}%
  }%
  \providecommand\rotatebox[2]{#2}%
  \newcommand*\fsize{\dimexpr\f@size pt\relax}%
  \newcommand*\lineheight[1]{\fontsize{\fsize}{#1\fsize}\selectfont}%
  \ifx\svgwidth\undefined%
    \setlength{\unitlength}{450.71264072bp}%
    \ifx\svgscale\undefined%
      \relax%
    \else%
      \setlength{\unitlength}{\unitlength * \real{\svgscale}}%
    \fi%
  \else%
    \setlength{\unitlength}{\svgwidth}%
  \fi%
  \global\let\svgwidth\undefined%
  \global\let\svgscale\undefined%
  \makeatother%
  \begin{picture}(1,0.99841886)%
    \lineheight{1}%
    \setlength\tabcolsep{0pt}%
    \put(0,0){\includegraphics[width=\unitlength,page=1]{RabbitandBasilica.pdf}}%
  \end{picture}%
\endgroup%

%% file: BandsShareSides.pdf_tex
\begingroup%
  \makeatletter%
  \providecommand\color[2][]{%
    \errmessage{(Inkscape) Color is used for the text in Inkscape, but the package 'color.sty' is not loaded}%
    \renewcommand\color[2][]{}%
  }%
  \providecommand\transparent[1]{%
    \errmessage{(Inkscape) Transparency is used (non-zero) for the text in Inkscape, but the package 'transparent.sty' is not loaded}%
    \renewcommand\transparent[1]{}%
  }%
  \providecommand\rotatebox[2]{#2}%
  \newcommand*\fsize{\dimexpr\f@size pt\relax}%
  \newcommand*\lineheight[1]{\fontsize{\fsize}{#1\fsize}\selectfont}%
  \ifx\svgwidth\undefined%
    \setlength{\unitlength}{772.50576453bp}%
    \ifx\svgscale\undefined%
      \relax%
    \else%
      \setlength{\unitlength}{\unitlength * \real{\svgscale}}%
    \fi%
  \else%
    \setlength{\unitlength}{\svgwidth}%
  \fi%
  \global\let\svgwidth\undefined%
  \global\let\svgscale\undefined%
  \makeatother%
  \begin{picture}(1,0.25291788)%
    \lineheight{1}%
    \setlength\tabcolsep{0pt}%
    \put(0,0){\includegraphics[width=\unitlength,page=1]{BandsShareSides.pdf}}%
    \put(0.22104536,0.19554579){\color[rgb]{0,0,0}\makebox(0,0)[lt]{\lineheight{1.25}\smash{\begin{tabular}[t]{l}$e_1$\end{tabular}}}}%
    \put(0.22104536,0.03304142){\color[rgb]{0,0,0}\makebox(0,0)[lt]{\lineheight{1.25}\smash{\begin{tabular}[t]{l}$e_2$\end{tabular}}}}%
    \put(0.05356204,0.10973992){\color[rgb]{0,0,0}\makebox(0,0)[lt]{\lineheight{1.25}\smash{\begin{tabular}[t]{l}$b^0$\end{tabular}}}}%
    \put(0.33551745,0.10973968){\color[rgb]{0,0,0}\makebox(0,0)[lt]{\lineheight{1.25}\smash{\begin{tabular}[t]{l}$b'^0$\end{tabular}}}}%
    \put(0.60262907,0.10973968){\color[rgb]{0,0,0}\makebox(0,0)[lt]{\lineheight{1.25}\smash{\begin{tabular}[t]{l}$b^p$\end{tabular}}}}%
    \put(0.91598741,0.10973968){\color[rgb]{0,0,0}\makebox(0,0)[lt]{\lineheight{1.25}\smash{\begin{tabular}[t]{l}$b'^p$\end{tabular}}}}%
    \put(0,0){\includegraphics[width=\unitlength,page=2]{BandsShareSides.pdf}}%
    \put(0.48110103,0.1406278){\color[rgb]{0,0,0}\makebox(0,0)[lt]{\lineheight{1.25}\smash{\begin{tabular}[t]{l}$f^p$\end{tabular}}}}%
  \end{picture}%
\endgroup%

%% file: CornerBand.pdf_tex
\begingroup%
  \makeatletter%
  \providecommand\color[2][]{%
    \errmessage{(Inkscape) Color is used for the text in Inkscape, but the package 'color.sty' is not loaded}%
    \renewcommand\color[2][]{}%
  }%
  \providecommand\transparent[1]{%
    \errmessage{(Inkscape) Transparency is used (non-zero) for the text in Inkscape, but the package 'transparent.sty' is not loaded}%
    \renewcommand\transparent[1]{}%
  }%
  \providecommand\rotatebox[2]{#2}%
  \newcommand*\fsize{\dimexpr\f@size pt\relax}%
  \newcommand*\lineheight[1]{\fontsize{\fsize}{#1\fsize}\selectfont}%
  \ifx\svgwidth\undefined%
    \setlength{\unitlength}{219.18866731bp}%
    \ifx\svgscale\undefined%
      \relax%
    \else%
      \setlength{\unitlength}{\unitlength * \real{\svgscale}}%
    \fi%
  \else%
    \setlength{\unitlength}{\svgwidth}%
  \fi%
  \global\let\svgwidth\undefined%
  \global\let\svgscale\undefined%
  \makeatother%
  \begin{picture}(1,0.8306715)%
    \lineheight{1}%
    \setlength\tabcolsep{0pt}%
    \put(0,0){\includegraphics[width=\unitlength,page=1]{CornerBand.pdf}}%
    \put(0.21666695,0.01638208){\color[rgb]{0,0,0}\makebox(0,0)[lt]{\lineheight{1.25}\smash{\begin{tabular}[t]{l}$v$\end{tabular}}}}%
    \put(0.3332408,0.21236395){\color[rgb]{0,0,0}\makebox(0,0)[lt]{\lineheight{1.25}\smash{\begin{tabular}[t]{l}$t^m$\end{tabular}}}}%
    \put(0.14260266,0.59852835){\color[rgb]{0,0,0}\makebox(0,0)[lt]{\lineheight{1.25}\smash{\begin{tabular}[t]{l}$t^n$\end{tabular}}}}%
  \end{picture}%
\endgroup%

%% file: Basilica_regulated_path.pdf_tex
\begingroup%
  \makeatletter%
  \providecommand\color[2][]{%
    \errmessage{(Inkscape) Color is used for the text in Inkscape, but the package 'color.sty' is not loaded}%
    \renewcommand\color[2][]{}%
  }%
  \providecommand\transparent[1]{%
    \errmessage{(Inkscape) Transparency is used (non-zero) for the text in Inkscape, but the package 'transparent.sty' is not loaded}%
    \renewcommand\transparent[1]{}%
  }%
  \providecommand\rotatebox[2]{#2}%
  \newcommand*\fsize{\dimexpr\f@size pt\relax}%
  \newcommand*\lineheight[1]{\fontsize{\fsize}{#1\fsize}\selectfont}%
  \ifx\svgwidth\undefined%
    \setlength{\unitlength}{172.59455523bp}%
    \ifx\svgscale\undefined%
      \relax%
    \else%
      \setlength{\unitlength}{\unitlength * \real{\svgscale}}%
    \fi%
  \else%
    \setlength{\unitlength}{\svgwidth}%
  \fi%
  \global\let\svgwidth\undefined%
  \global\let\svgscale\undefined%
  \makeatother%
  \begin{picture}(1,0.49482392)%
    \lineheight{1}%
    \setlength\tabcolsep{0pt}%
    \put(0,0){\includegraphics[width=\unitlength,page=1]{Basilica_regulated_path.pdf}}%
    \put(0.31319123,0.18626839){\color[rgb]{0,0,0}\makebox(0,0)[lt]{\lineheight{1.25}\smash{\begin{tabular}[t]{l}$e$\end{tabular}}}}%
    \put(0.72753345,0.18627591){\color[rgb]{0,0,0}\makebox(0,0)[lt]{\lineheight{1.25}\smash{\begin{tabular}[t]{l}$e'$\end{tabular}}}}%
  \end{picture}%
\endgroup%

%% file: selfjoin.pdf_tex
\begingroup%
  \makeatletter%
  \providecommand\color[2][]{%
    \errmessage{(Inkscape) Color is used for the text in Inkscape, but the package 'color.sty' is not loaded}%
    \renewcommand\color[2][]{}%
  }%
  \providecommand\transparent[1]{%
    \errmessage{(Inkscape) Transparency is used (non-zero) for the text in Inkscape, but the package 'transparent.sty' is not loaded}%
    \renewcommand\transparent[1]{}%
  }%
  \providecommand\rotatebox[2]{#2}%
  \newcommand*\fsize{\dimexpr\f@size pt\relax}%
  \newcommand*\lineheight[1]{\fontsize{\fsize}{#1\fsize}\selectfont}%
  \ifx\svgwidth\undefined%
    \setlength{\unitlength}{170.51556565bp}%
    \ifx\svgscale\undefined%
      \relax%
    \else%
      \setlength{\unitlength}{\unitlength * \real{\svgscale}}%
    \fi%
  \else%
    \setlength{\unitlength}{\svgwidth}%
  \fi%
  \global\let\svgwidth\undefined%
  \global\let\svgscale\undefined%
  \makeatother%
  \begin{picture}(1,0.40204865)%
    \lineheight{1}%
    \setlength\tabcolsep{0pt}%
    \put(0,0){\includegraphics[width=\unitlength,page=1]{selfjoin.pdf}}%
    \put(0.05895109,0.17017946){\color[rgb]{0,0,0}\makebox(0,0)[lt]{\lineheight{1.25}\smash{\begin{tabular}[t]{l}$G$\end{tabular}}}}%
    \put(0.3087461,0.01228048){\color[rgb]{0,0,0}\makebox(0,0)[lt]{\lineheight{1.25}\smash{\begin{tabular}[t]{l}$e_1$\end{tabular}}}}%
    \put(0.63121788,0.01228048){\color[rgb]{0,0,0}\makebox(0,0)[lt]{\lineheight{1.25}\smash{\begin{tabular}[t]{l}$e_2$\end{tabular}}}}%
    \put(0.78891559,0.35976955){\color[rgb]{0,0,0}\makebox(0,0)[lt]{\lineheight{1.25}\smash{\begin{tabular}[t]{l}$e_3$\end{tabular}}}}%
    \put(0.48611801,0.01228049){\color[rgb]{0,0,0}\makebox(0,0)[lt]{\lineheight{1.25}\smash{\begin{tabular}[t]{l}$p$\end{tabular}}}}%
  \end{picture}%
\endgroup%

%% file: periodicjoin.pdf_tex
\begingroup%
  \makeatletter%
  \providecommand\color[2][]{%
    \errmessage{(Inkscape) Color is used for the text in Inkscape, but the package 'color.sty' is not loaded}%
    \renewcommand\color[2][]{}%
  }%
  \providecommand\transparent[1]{%
    \errmessage{(Inkscape) Transparency is used (non-zero) for the text in Inkscape, but the package 'transparent.sty' is not loaded}%
    \renewcommand\transparent[1]{}%
  }%
  \providecommand\rotatebox[2]{#2}%
  \newcommand*\fsize{\dimexpr\f@size pt\relax}%
  \newcommand*\lineheight[1]{\fontsize{\fsize}{#1\fsize}\selectfont}%
  \ifx\svgwidth\undefined%
    \setlength{\unitlength}{235.21421417bp}%
    \ifx\svgscale\undefined%
      \relax%
    \else%
      \setlength{\unitlength}{\unitlength * \real{\svgscale}}%
    \fi%
  \else%
    \setlength{\unitlength}{\svgwidth}%
  \fi%
  \global\let\svgwidth\undefined%
  \global\let\svgscale\undefined%
  \makeatother%
  \begin{picture}(1,0.33002668)%
    \lineheight{1}%
    \setlength\tabcolsep{0pt}%
    \put(0,0){\includegraphics[width=\unitlength,page=1]{periodicjoin.pdf}}%
    \put(0.09498949,0.05432034){\color[rgb]{0,0,0}\makebox(0,0)[lt]{\lineheight{1.25}\smash{\begin{tabular}[t]{l}$G_1$\end{tabular}}}}%
    \put(0,0){\includegraphics[width=\unitlength,page=2]{periodicjoin.pdf}}%
    \put(0.82346029,0.23246216){\color[rgb]{0,0,0}\makebox(0,0)[lt]{\lineheight{1.25}\smash{\begin{tabular}[t]{l}$G_2$\end{tabular}}}}%
    \put(0.29552455,0.11735246){\color[rgb]{0,0,0}\makebox(0,0)[lt]{\lineheight{1.25}\smash{\begin{tabular}[t]{l}$e_1$\end{tabular}}}}%
    \put(0.62044871,0.03927235){\color[rgb]{0,0,0}\makebox(0,0)[lt]{\lineheight{1.25}\smash{\begin{tabular}[t]{l}$e_1'$\end{tabular}}}}%
    \put(0.64250783,0.17161354){\color[rgb]{0,0,0}\makebox(0,0)[lt]{\lineheight{1.25}\smash{\begin{tabular}[t]{l}$e_2$\end{tabular}}}}%
    \put(0.33726234,0.25183112){\color[rgb]{0,0,0}\makebox(0,0)[lt]{\lineheight{1.25}\smash{\begin{tabular}[t]{l}$e_2'$\end{tabular}}}}%
    \put(0,0){\includegraphics[width=\unitlength,page=3]{periodicjoin.pdf}}%
    \put(0.34768432,0.01276536){\color[rgb]{0,0,0}\makebox(0,0)[lt]{\lineheight{1.25}\smash{\begin{tabular}[t]{l}$p_1$\end{tabular}}}}%
    \put(0.58132793,0.27779304){\color[rgb]{0,0,0}\makebox(0,0)[lt]{\lineheight{1.25}\smash{\begin{tabular}[t]{l}$p_2$\end{tabular}}}}%
  \end{picture}%
\endgroup%

%% file: Neckless.pdf_tex
\begingroup%
  \makeatletter%
  \providecommand\color[2][]{%
    \errmessage{(Inkscape) Color is used for the text in Inkscape, but the package 'color.sty' is not loaded}%
    \renewcommand\color[2][]{}%
  }%
  \providecommand\transparent[1]{%
    \errmessage{(Inkscape) Transparency is used (non-zero) for the text in Inkscape, but the package 'transparent.sty' is not loaded}%
    \renewcommand\transparent[1]{}%
  }%
  \providecommand\rotatebox[2]{#2}%
  \newcommand*\fsize{\dimexpr\f@size pt\relax}%
  \newcommand*\lineheight[1]{\fontsize{\fsize}{#1\fsize}\selectfont}%
  \ifx\svgwidth\undefined%
    \setlength{\unitlength}{381.01275071bp}%
    \ifx\svgscale\undefined%
      \relax%
    \else%
      \setlength{\unitlength}{\unitlength * \real{\svgscale}}%
    \fi%
  \else%
    \setlength{\unitlength}{\svgwidth}%
  \fi%
  \global\let\svgwidth\undefined%
  \global\let\svgscale\undefined%
  \makeatother%
  \begin{picture}(1,0.92849302)%
    \lineheight{1}%
    \setlength\tabcolsep{0pt}%
    \put(0,0){\includegraphics[width=\unitlength,page=1]{Neckless.pdf}}%
  \end{picture}%
\endgroup%

%% file: origboxtimes.pdf_tex
\begingroup%
  \makeatletter%
  \providecommand\color[2][]{%
    \errmessage{(Inkscape) Color is used for the text in Inkscape, but the package 'color.sty' is not loaded}%
    \renewcommand\color[2][]{}%
  }%
  \providecommand\transparent[1]{%
    \errmessage{(Inkscape) Transparency is used (non-zero) for the text in Inkscape, but the package 'transparent.sty' is not loaded}%
    \renewcommand\transparent[1]{}%
  }%
  \providecommand\rotatebox[2]{#2}%
  \newcommand*\fsize{\dimexpr\f@size pt\relax}%
  \newcommand*\lineheight[1]{\fontsize{\fsize}{#1\fsize}\selectfont}%
  \ifx\svgwidth\undefined%
    \setlength{\unitlength}{440.20330955bp}%
    \ifx\svgscale\undefined%
      \relax%
    \else%
      \setlength{\unitlength}{\unitlength * \real{\svgscale}}%
    \fi%
  \else%
    \setlength{\unitlength}{\svgwidth}%
  \fi%
  \global\let\svgwidth\undefined%
  \global\let\svgscale\undefined%
  \makeatother%
  \begin{picture}(1,0.44168282)%
    \lineheight{1}%
    \setlength\tabcolsep{0pt}%
    \put(0,0){\includegraphics[width=\unitlength,page=1]{origboxtimes.pdf}}%
  \end{picture}%
\endgroup%

%% file: BasilicaFSR.pdf_tex
\begingroup%
  \makeatletter%
  \providecommand\color[2][]{%
    \errmessage{(Inkscape) Color is used for the text in Inkscape, but the package 'color.sty' is not loaded}%
    \renewcommand\color[2][]{}%
  }%
  \providecommand\transparent[1]{%
    \errmessage{(Inkscape) Transparency is used (non-zero) for the text in Inkscape, but the package 'transparent.sty' is not loaded}%
    \renewcommand\transparent[1]{}%
  }%
  \providecommand\rotatebox[2]{#2}%
  \newcommand*\fsize{\dimexpr\f@size pt\relax}%
  \newcommand*\lineheight[1]{\fontsize{\fsize}{#1\fsize}\selectfont}%
  \ifx\svgwidth\undefined%
    \setlength{\unitlength}{506.34662778bp}%
    \ifx\svgscale\undefined%
      \relax%
    \else%
      \setlength{\unitlength}{\unitlength * \real{\svgscale}}%
    \fi%
  \else%
    \setlength{\unitlength}{\svgwidth}%
  \fi%
  \global\let\svgwidth\undefined%
  \global\let\svgscale\undefined%
  \makeatother%
  \begin{picture}(1,0.30730706)%
    \lineheight{1}%
    \setlength\tabcolsep{0pt}%
    \put(0,0){\includegraphics[width=\unitlength,page=1]{BasilicaFSR.pdf}}%
  \end{picture}%
\endgroup%

%% file: BasilicaFSRwithDualGraphwithLength.pdf_tex
\begingroup%
  \makeatletter%
  \providecommand\color[2][]{%
    \errmessage{(Inkscape) Color is used for the text in Inkscape, but the package 'color.sty' is not loaded}%
    \renewcommand\color[2][]{}%
  }%
  \providecommand\transparent[1]{%
    \errmessage{(Inkscape) Transparency is used (non-zero) for the text in Inkscape, but the package 'transparent.sty' is not loaded}%
    \renewcommand\transparent[1]{}%
  }%
  \providecommand\rotatebox[2]{#2}%
  \newcommand*\fsize{\dimexpr\f@size pt\relax}%
  \newcommand*\lineheight[1]{\fontsize{\fsize}{#1\fsize}\selectfont}%
  \ifx\svgwidth\undefined%
    \setlength{\unitlength}{506.88201664bp}%
    \ifx\svgscale\undefined%
      \relax%
    \else%
      \setlength{\unitlength}{\unitlength * \real{\svgscale}}%
    \fi%
  \else%
    \setlength{\unitlength}{\svgwidth}%
  \fi%
  \global\let\svgwidth\undefined%
  \global\let\svgscale\undefined%
  \makeatother%
  \begin{picture}(1,0.30697497)%
    \lineheight{1}%
    \setlength\tabcolsep{0pt}%
    \put(0,0){\includegraphics[width=\unitlength,page=1]{BasilicaFSRwithDualGraphwithLength.pdf}}%
    \put(0.32546136,0.15966524){\color[rgb]{0,0,0}\makebox(0,0)[lt]{\lineheight{1.25}\smash{\begin{tabular}[t]{l}$10^2$\end{tabular}}}}%
    \put(0.62629457,0.15718306){\color[rgb]{0,0,0}\makebox(0,0)[lt]{\lineheight{1.25}\smash{\begin{tabular}[t]{l}$10^2$\end{tabular}}}}%
    \put(0.6638218,0.1984695){\color[rgb]{0,0,0}\makebox(0,0)[lt]{\lineheight{1.25}\smash{\begin{tabular}[t]{l}$10^2$\end{tabular}}}}%
    \put(0.66328557,0.08911995){\color[rgb]{0,0,0}\makebox(0,0)[lt]{\lineheight{1.25}\smash{\begin{tabular}[t]{l}$10^2$\end{tabular}}}}%
    \put(0.95033443,0.15718315){\color[rgb]{0,0,0}\makebox(0,0)[lt]{\lineheight{1.25}\smash{\begin{tabular}[t]{l}$10^2$\end{tabular}}}}%
    \put(0.91926359,0.20157222){\color[rgb]{0,0,0}\makebox(0,0)[lt]{\lineheight{1.25}\smash{\begin{tabular}[t]{l}$10^2$\end{tabular}}}}%
    \put(0.91926538,0.08911986){\color[rgb]{0,0,0}\makebox(0,0)[lt]{\lineheight{1.25}\smash{\begin{tabular}[t]{l}$10^2$\end{tabular}}}}%
    \put(0.02704268,0.15570353){\color[rgb]{0,0,0}\makebox(0,0)[lt]{\lineheight{1.25}\smash{\begin{tabular}[t]{l}$10$\end{tabular}}}}%
    \put(0.081532,0.17176276){\color[rgb]{0,0,0}\makebox(0,0)[lt]{\lineheight{1.25}\smash{\begin{tabular}[t]{l}$10$\end{tabular}}}}%
    \put(0.76696004,0.19574613){\color[rgb]{0,0,0}\makebox(0,0)[lt]{\lineheight{1.25}\smash{\begin{tabular}[t]{l}$10$\end{tabular}}}}%
    \put(0.79962649,0.22820566){\color[rgb]{0,0,0}\makebox(0,0)[lt]{\lineheight{1.25}\smash{\begin{tabular}[t]{l}$10$\end{tabular}}}}%
    \put(0.81146341,0.19565367){\color[rgb]{0,0,0}\makebox(0,0)[lt]{\lineheight{1.25}\smash{\begin{tabular}[t]{l}$10$\end{tabular}}}}%
    \put(0.76707677,0.09956975){\color[rgb]{0,0,0}\makebox(0,0)[lt]{\lineheight{1.25}\smash{\begin{tabular}[t]{l}$10$\end{tabular}}}}%
    \put(0.79962445,0.06396608){\color[rgb]{0,0,0}\makebox(0,0)[lt]{\lineheight{1.25}\smash{\begin{tabular}[t]{l}$10$\end{tabular}}}}%
    \put(0.81158014,0.09947742){\color[rgb]{0,0,0}\makebox(0,0)[lt]{\lineheight{1.25}\smash{\begin{tabular}[t]{l}$10$\end{tabular}}}}%
    \put(0.08496517,0.10983485){\color[rgb]{0,0,0}\makebox(0,0)[lt]{\lineheight{1.25}\smash{\begin{tabular}[t]{l}$10$\end{tabular}}}}%
    \put(0.17504206,0.18781904){\color[rgb]{0,0,0}\makebox(0,0)[lt]{\lineheight{1.25}\smash{\begin{tabular}[t]{l}$10^3$\end{tabular}}}}%
    \put(0.11464588,0.21245115){\color[rgb]{0,0,0}\makebox(0,0)[lt]{\lineheight{1.25}\smash{\begin{tabular}[t]{l}$10^3$\end{tabular}}}}%
    \put(0.70767375,0.2191111){\color[rgb]{0,0,0}\makebox(0,0)[lt]{\lineheight{1.25}\smash{\begin{tabular}[t]{l}$10^3$\end{tabular}}}}%
    \put(0.74499882,0.17267308){\color[rgb]{0,0,0}\makebox(0,0)[lt]{\lineheight{1.25}\smash{\begin{tabular}[t]{l}$10^3$\end{tabular}}}}%
    \put(0.74614491,0.11997557){\color[rgb]{0,0,0}\makebox(0,0)[lt]{\lineheight{1.25}\smash{\begin{tabular}[t]{l}$10^3$\end{tabular}}}}%
    \put(0.70743328,0.07893032){\color[rgb]{0,0,0}\makebox(0,0)[lt]{\lineheight{1.25}\smash{\begin{tabular}[t]{l}$10^3$\end{tabular}}}}%
    \put(0.83048451,0.17324229){\color[rgb]{0,0,0}\makebox(0,0)[lt]{\lineheight{1.25}\smash{\begin{tabular}[t]{l}$10^3$\end{tabular}}}}%
    \put(0.83048451,0.11997545){\color[rgb]{0,0,0}\makebox(0,0)[lt]{\lineheight{1.25}\smash{\begin{tabular}[t]{l}$10^3$\end{tabular}}}}%
    \put(0.86599526,0.19691646){\color[rgb]{0,0,0}\makebox(0,0)[lt]{\lineheight{1.25}\smash{\begin{tabular}[t]{l}$10^3$\end{tabular}}}}%
    \put(0.86303603,0.08742355){\color[rgb]{0,0,0}\makebox(0,0)[lt]{\lineheight{1.25}\smash{\begin{tabular}[t]{l}$10^3$\end{tabular}}}}%
    \put(0.11895941,0.07955593){\color[rgb]{0,0,0}\makebox(0,0)[lt]{\lineheight{1.25}\smash{\begin{tabular}[t]{l}$10^3$\end{tabular}}}}%
    \put(0.17500605,0.10221995){\color[rgb]{0,0,0}\makebox(0,0)[lt]{\lineheight{1.25}\smash{\begin{tabular}[t]{l}$10^3$\end{tabular}}}}%
    \put(0.23816108,0.08766806){\color[rgb]{0,0,0}\makebox(0,0)[lt]{\lineheight{1.25}\smash{\begin{tabular}[t]{l}$10^2$\end{tabular}}}}%
    \put(0.24061263,0.20137685){\color[rgb]{0,0,0}\makebox(0,0)[lt]{\lineheight{1.25}\smash{\begin{tabular}[t]{l}$10^2$\end{tabular}}}}%
  \end{picture}%
\endgroup%

%% file: FJFJ.pdf_tex
\begingroup%
  \makeatletter%
  \providecommand\color[2][]{%
    \errmessage{(Inkscape) Color is used for the text in Inkscape, but the package 'color.sty' is not loaded}%
    \renewcommand\color[2][]{}%
  }%
  \providecommand\transparent[1]{%
    \errmessage{(Inkscape) Transparency is used (non-zero) for the text in Inkscape, but the package 'transparent.sty' is not loaded}%
    \renewcommand\transparent[1]{}%
  }%
  \providecommand\rotatebox[2]{#2}%
  \newcommand*\fsize{\dimexpr\f@size pt\relax}%
  \newcommand*\lineheight[1]{\fontsize{\fsize}{#1\fsize}\selectfont}%
  \ifx\svgwidth\undefined%
    \setlength{\unitlength}{495.60109056bp}%
    \ifx\svgscale\undefined%
      \relax%
    \else%
      \setlength{\unitlength}{\unitlength * \real{\svgscale}}%
    \fi%
  \else%
    \setlength{\unitlength}{\svgwidth}%
  \fi%
  \global\let\svgwidth\undefined%
  \global\let\svgscale\undefined%
  \makeatother%
  \begin{picture}(1,0.43383169)%
    \lineheight{1}%
    \setlength\tabcolsep{0pt}%
    \put(0,0){\includegraphics[width=\unitlength,page=1]{FJFJ.pdf}}%
    \put(-0.00136731,0.00792201){\color[rgb]{0,0,0}\makebox(0,0)[lt]{\lineheight{1.25}\smash{\begin{tabular}[t]{l}$J$\end{tabular}}}}%
    \put(0.55936874,0.00792319){\color[rgb]{0,0,0}\makebox(0,0)[lt]{\lineheight{1.25}\smash{\begin{tabular}[t]{l}$J$\end{tabular}}}}%
    \put(0.90557295,0.00792319){\color[rgb]{0,0,0}\makebox(0,0)[lt]{\lineheight{1.25}\smash{\begin{tabular}[t]{l}$F$\end{tabular}}}}%
    \put(0.78235786,0.00792319){\color[rgb]{0,0,0}\makebox(0,0)[lt]{\lineheight{1.25}\smash{\begin{tabular}[t]{l}$J$\end{tabular}}}}%
    \put(0.76450666,0.22358267){\color[rgb]{0,0,0}\makebox(0,0)[lt]{\lineheight{1.25}\smash{\begin{tabular}[t]{l}$J$\end{tabular}}}}%
    \put(0.6768043,0.00792319){\color[rgb]{0,0,0}\makebox(0,0)[lt]{\lineheight{1.25}\smash{\begin{tabular}[t]{l}$F$\end{tabular}}}}%
    \put(0.32779365,0.00792319){\color[rgb]{0,0,0}\makebox(0,0)[lt]{\lineheight{1.25}\smash{\begin{tabular}[t]{l}$F$\end{tabular}}}}%
    \put(0,0){\includegraphics[width=\unitlength,page=2]{FJFJ.pdf}}%
    \put(0.92443576,0.13635826){\color[rgb]{0,0,0}\makebox(0,0)[lt]{\lineheight{1.25}\smash{\begin{tabular}[t]{l}$J$\end{tabular}}}}%
    \put(0.7873086,0.13635826){\color[rgb]{0,0,0}\makebox(0,0)[lt]{\lineheight{1.25}\smash{\begin{tabular}[t]{l}$F$\end{tabular}}}}%
    \put(0.67170077,0.13333164){\color[rgb]{0,0,0}\makebox(0,0)[lt]{\lineheight{1.25}\smash{\begin{tabular}[t]{l}$J$\end{tabular}}}}%
    \put(0.92624497,0.36725488){\color[rgb]{0,0,0}\makebox(0,0)[lt]{\lineheight{1.25}\smash{\begin{tabular}[t]{l}$J$\end{tabular}}}}%
    \put(0.92624342,0.23072176){\color[rgb]{0,0,0}\makebox(0,0)[lt]{\lineheight{1.25}\smash{\begin{tabular}[t]{l}$F$\end{tabular}}}}%
    \put(0.33899017,0.39032753){\color[rgb]{0,0,0}\makebox(0,0)[lt]{\lineheight{1.25}\smash{\begin{tabular}[t]{l}$J$\end{tabular}}}}%
    \put(0,0){\includegraphics[width=\unitlength,page=3]{FJFJ.pdf}}%
  \end{picture}%
\endgroup%

%% file: Deformation.pdf_tex
\begingroup%
  \makeatletter%
  \providecommand\color[2][]{%
    \errmessage{(Inkscape) Color is used for the text in Inkscape, but the package 'color.sty' is not loaded}%
    \renewcommand\color[2][]{}%
  }%
  \providecommand\transparent[1]{%
    \errmessage{(Inkscape) Transparency is used (non-zero) for the text in Inkscape, but the package 'transparent.sty' is not loaded}%
    \renewcommand\transparent[1]{}%
  }%
  \providecommand\rotatebox[2]{#2}%
  \newcommand*\fsize{\dimexpr\f@size pt\relax}%
  \newcommand*\lineheight[1]{\fontsize{\fsize}{#1\fsize}\selectfont}%
  \ifx\svgwidth\undefined%
    \setlength{\unitlength}{428.6157611bp}%
    \ifx\svgscale\undefined%
      \relax%
    \else%
      \setlength{\unitlength}{\unitlength * \real{\svgscale}}%
    \fi%
  \else%
    \setlength{\unitlength}{\svgwidth}%
  \fi%
  \global\let\svgwidth\undefined%
  \global\let\svgscale\undefined%
  \makeatother%
  \begin{picture}(1,0.93857556)%
    \lineheight{1}%
    \setlength\tabcolsep{0pt}%
    \put(0.49920007,0.74150204){\color[rgb]{0,0,0}\makebox(0,0)[lt]{\lineheight{1.25}\smash{\begin{tabular}[t]{l}$Z$\end{tabular}}}}%
    \put(0,0){\includegraphics[width=\unitlength,page=1]{Deformation.pdf}}%
    \put(0.15737944,0.09006424){\color[rgb]{0,0,0}\makebox(0,0)[lt]{\lineheight{1.25}\smash{\begin{tabular}[t]{l}$v$\end{tabular}}}}%
    \put(0,0){\includegraphics[width=\unitlength,page=2]{Deformation.pdf}}%
    \put(0.84563426,0.00870063){\color[rgb]{0,0,0}\makebox(0,0)[lt]{\lineheight{1.25}\smash{\begin{tabular}[t]{l}$\epsilon$\end{tabular}}}}%
    \put(0,0){\includegraphics[width=\unitlength,page=3]{Deformation.pdf}}%
    \put(0.8462338,0.60435323){\color[rgb]{0,0,0}\makebox(0,0)[lt]{\lineheight{1.25}\smash{\begin{tabular}[t]{l}$\epsilon'$\end{tabular}}}}%
    \put(0.38105657,0.32604826){\color[rgb]{0,0,0}\rotatebox{0.23966933}{\makebox(0,0)[lt]{\lineheight{1.25}\smash{\begin{tabular}[t]{l}$\psi$\end{tabular}}}}}%
    \put(0.54796089,0.32338996){\color[rgb]{0,0,0}\rotatebox{0.97786839}{\makebox(0,0)[lt]{\lineheight{1.25}\smash{\begin{tabular}[t]{l}$\psi_{(\epsilon,\epsilon')}$\end{tabular}}}}}%
    \put(0.19975353,0.13958627){\color[rgb]{0,0,0}\makebox(0,0)[lt]{\lineheight{1.25}\smash{\begin{tabular}[t]{l}$e_0$\end{tabular}}}}%
    \put(0.7268086,0.75015648){\color[rgb]{0,0,0}\makebox(0,0)[lt]{\lineheight{1.25}\smash{\begin{tabular}[t]{l}$e_{0,j}$'s\end{tabular}}}}%
    \put(0.0592926,0.04220885){\color[rgb]{0,0,0}\makebox(0,0)[lt]{\lineheight{1.25}\smash{\begin{tabular}[t]{l}$e_1$\end{tabular}}}}%
    \put(0.23769435,0.64880978){\color[rgb]{0,0,0}\makebox(0,0)[lt]{\lineheight{1.25}\smash{\begin{tabular}[t]{l}$e_{1,j}$'s\end{tabular}}}}%
    \put(0.1990652,0.7536568){\color[rgb]{0,0,0}\makebox(0,0)[lt]{\lineheight{1.25}\smash{\begin{tabular}[t]{l}$e_{2,j}$'s\end{tabular}}}}%
    \put(0.27045256,0.91454248){\color[rgb]{0,0,0}\makebox(0,0)[lt]{\lineheight{1.25}\smash{\begin{tabular}[t]{l}$e_{i,j}$'s\end{tabular}}}}%
    \put(-0.00175565,0.11710465){\color[rgb]{0,0,0}\makebox(0,0)[lt]{\lineheight{1.25}\smash{\begin{tabular}[t]{l}$e_2$\end{tabular}}}}%
    \put(0.05503793,0.23182893){\color[rgb]{0,0,0}\makebox(0,0)[lt]{\lineheight{1.25}\smash{\begin{tabular}[t]{l}$e_i$\end{tabular}}}}%
    \put(0,0){\includegraphics[width=\unitlength,page=4]{Deformation.pdf}}%
    \put(0.4813884,0.89902998){\color[rgb]{0,0,0}\makebox(0,0)[lt]{\lineheight{1.25}\smash{\begin{tabular}[t]{l}$v_{i,j,\epsilon'}$'s\end{tabular}}}}%
    \put(0.87705875,0.08603439){\color[rgb]{0,0,0}\makebox(0,0)[lt]{\lineheight{1.25}\smash{\begin{tabular}[t]{l}$v_\epsilon$\end{tabular}}}}%
    \put(0,0){\includegraphics[width=\unitlength,page=5]{Deformation.pdf}}%
  \end{picture}%
\endgroup%

%% file: Cdimone.bib
@article {Whyburn_TopChacSierCarpet,
    AUTHOR = {Whyburn, G. T.},
     TITLE = {Topological characterization of the {S}ierpi\'nski curve},
   JOURNAL = {Fund. Math.},
  FJOURNAL = {Polska Akademia Nauk. Fundamenta Mathematicae},
    VOLUME = {45},
      YEAR = {1958},
     PAGES = {320--324},
      ISSN = {0016-2736,1730-6329},
   MRCLASS = {54.00},
  MRNUMBER = {99638},
MRREVIEWER = {A.\ H.\ Stone},
       DOI = {10.4064/fm-45-1-320-324},
       URL = {https://doi-org.proxy.library.stonybrook.edu/10.4064/fm-45-1-320-324},
}

@book {Milnor_Book,
    AUTHOR = {Milnor, John},
     TITLE = {Dynamics in one complex variable},
    SERIES = {Annals of Mathematics Studies},
    VOLUME = {160},
   EDITION = {Third},
 PUBLISHER = {Princeton University Press, Princeton, NJ},
      YEAR = {2006},
     PAGES = {viii+304},
      ISBN = {978-0-691-12488-9; 0-691-12488-4},
   MRCLASS = {37Fxx (30-01 30D05 37-01)},
  MRNUMBER = {2193309},
}

@article {EB_EqualityConfDim,
    AUTHOR = {Eriksson-Bique, Sylvester},
     TITLE = {Equality of different definitions of conformal dimension for
              quasiself-similar and {CLP} spaces},
   JOURNAL = {Ann. Fenn. Math.},
  FJOURNAL = {Annales Fennici Mathematici},
    VOLUME = {49},
      YEAR = {2024},
    NUMBER = {2},
     PAGES = {405--436},
      ISSN = {2737-0690,2737-114X},
   MRCLASS = {30L10 (28A78)},
  MRNUMBER = {4767001},
}

@article{Hansson_PreOrderExt,
 ISSN = {00397857, 15730964},
 URL = {http://www.jstor.org/stable/20114617},
 author = {Bengt Hansson},
 journal = {Synthese},
 number = {4},
 pages = {443--458},
 publisher = {Springer},
 title = {Choice Structures and Preference Relations},
 urldate = {2024-09-21},
 volume = {18},
 year = {1968}
}

@book {Hubbard_vol2,
    AUTHOR = {Hubbard, John Hamal},
     TITLE = {Teichm\"{u}ller theory and applications to geometry, topology,
              and dynamics. {V}ol. 2},
      NOTE = {Surface homeomorphisms and rational functions},
 PUBLISHER = {Matrix Editions, Ithaca, NY},
      YEAR = {2016},
     PAGES = {x+262},
      ISBN = {978-1-943863-00-6},
   MRCLASS = {30-02 (30C15 30F60 32G15 37Fxx)},
  MRNUMBER = {3675959},
MRREVIEWER = {Athanase\ Papadopoulos},
}

@article {ParkWu_Yshape,
    AUTHOR = {Wu, Angela and Park, Insung},
     TITLE = {Quasi-self-similar fractals containing "{Y}" have dimension
              larger than one},
   JOURNAL = {Discrete Contin. Dyn. Syst.},
  FJOURNAL = {Discrete and Continuous Dynamical Systems. Series A},
    VOLUME = {44},
      YEAR = {2024},
    NUMBER = {4},
     PAGES = {1073--1086},
      ISSN = {1078-0947,1553-5231},
   MRCLASS = {37B99 (37F10 37F35)},
  MRNUMBER = {4696152},
       DOI = {10.3934/dcds.2023138},
       URL = {https://doi.org/10.3934/dcds.2023138},
}

@incollection {Haiss_TopChacOfRatlMaps,
    AUTHOR = {Ha{\"{\i}}ssinsky, Peter},
     TITLE = {Some topological characterizations of rational maps and
              {K}leinian groups},
 BOOKTITLE = {Dynamical systems},
    SERIES = {Banach Center Publ.},
    VOLUME = {115},
     PAGES = {73--97},
 PUBLISHER = {Polish Acad. Sci. Inst. Math., Warsaw},
      YEAR = {2018},
   MRCLASS = {37F30 (20F67 30F40 30L10 37F10 57M60)},
  MRNUMBER = {3821046},
}

@article {Kinneberg_ConfDimPlanarDomain,
    AUTHOR = {Kinneberg, Kyle},
     TITLE = {Conformal dimension and boundaries of planar domains},
   JOURNAL = {Trans. Amer. Math. Soc.},
  FJOURNAL = {Transactions of the American Mathematical Society},
    VOLUME = {369},
      YEAR = {2017},
    NUMBER = {9},
     PAGES = {6511--6536},
      ISSN = {0002-9947},
   MRCLASS = {30L10 (28A78 30C20 30F40)},
  MRNUMBER = {3660231},
MRREVIEWER = {Pekka J. Pankka},
       DOI = {10.1090/tran/6944},
       URL = {https://doi.org/10.1090/tran/6944},
}

@incollection {Cannon_Conjecture_11_34,
    AUTHOR = {Cannon, James W.},
     TITLE = {The theory of negatively curved spaces and groups},
 BOOKTITLE = {Ergodic theory, symbolic dynamics, and hyperbolic spaces
              ({T}rieste, 1989)},
    SERIES = {Oxford Sci. Publ.},
     PAGES = {315--369},
 PUBLISHER = {Oxford Univ. Press, New York},
      YEAR = {1991},
   MRCLASS = {57S30 (20F32 53C23 57M07)},
  MRNUMBER = {1130181},
}

@article {Pansu_Confdim,
    AUTHOR = {Pansu, Pierre},
     TITLE = {Dimension conforme et sph\`ere \`a l'infini des vari\'{e}t\'{e}s \`a
              courbure n\'{e}gative},
   JOURNAL = {Ann. Acad. Sci. Fenn. Ser. A I Math.},
  FJOURNAL = {Annales Academiae Scientiarum Fennicae. Series A I.
              Mathematica},
    VOLUME = {14},
      YEAR = {1989},
    NUMBER = {2},
     PAGES = {177--212},
      ISSN = {0066-1953},
   MRCLASS = {53C20 (30C60)},
  MRNUMBER = {1024425},
MRREVIEWER = {Gudlaugur Thorbergsson},
       DOI = {10.5186/aasfm.1989.1424},
       URL = {https://doi.org/10.5186/aasfm.1989.1424},
}

@book {MackayTyson_Confdim,
	AUTHOR = {Mackay, John M. and Tyson, Jeremy T.},
	TITLE = {Conformal dimension},
	SERIES = {University Lecture Series},
	VOLUME = {54},
	NOTE = {Theory and application},
	PUBLISHER = {American Mathematical Society, Providence, RI},
	YEAR = {2010},
	PAGES = {xiv+143},
	ISBN = {978-0-8218-5229-3},
	MRCLASS = {30L10 (28A78 28A80 37F35)},
	MRNUMBER = {2662522},
	MRREVIEWER = {Leonid V. Kovalev},
	DOI = {10.1090/ulect/054},
	URL = {https://doi-org.proxyiub.uits.iu.edu/10.1090/ulect/054},
}

@article {BFH_ClassCritPreper,
	AUTHOR = {Bielefeld, Ben and Fisher, Yuval and Hubbard, John},
	TITLE = {The classification of critically preperiodic polynomials as
	dynamical systems},
	JOURNAL = {J. Amer. Math. Soc.},
	FJOURNAL = {Journal of the American Mathematical Society},
	VOLUME = {5},
	YEAR = {1992},
	NUMBER = {4},
	PAGES = {721--762},
	ISSN = {0894-0347},
	MRCLASS = {58F23 (30D05)},
	MRNUMBER = {1149891},
	MRREVIEWER = {Feliks Przytycki},
	DOI = {10.2307/2152709},
	URL = {https://doi-org.proxyiub.uits.iu.edu/10.2307/2152709},
}

@article {Poi_CritPort,
	AUTHOR = {Poirier, Alfredo},
	TITLE = {Critical portraits for postcritically finite polynomials},
	JOURNAL = {Fund. Math.},
	FJOURNAL = {Fundamenta Mathematicae},
	VOLUME = {203},
	YEAR = {2009},
	NUMBER = {2},
	PAGES = {107--163},
	ISSN = {0016-2736},
	MRCLASS = {37F10 (37E15)},
	MRNUMBER = {2496235},
	MRREVIEWER = {Masayo Fujimura},
	DOI = {10.4064/fm203-2-2},
	URL = {https://doi-org.proxyiub.uits.iu.edu/10.4064/fm203-2-2},
}

@article {DMRS_ClassifiPCFixedNewton,
	AUTHOR = {Drach, Kostiantyn and Mikulich, Yauhen and R\"{u}ckert, Johannes
	and Schleicher, Dierk},
	TITLE = {A combinatorial classification of postcritically fixed
	{N}ewton maps},
	JOURNAL = {Ergodic Theory Dynam. Systems},
	FJOURNAL = {Ergodic Theory and Dynamical Systems},
	VOLUME = {39},
	YEAR = {2019},
	NUMBER = {11},
	PAGES = {2983--3014},
	ISSN = {0143-3857},
	MRCLASS = {37F20 (05A15)},
	MRNUMBER = {4015138},
	MRREVIEWER = {Peter Ha\"{\i}ssinsky},
	DOI = {10.1017/etds.2018.2},
	URL = {https://doi.org/10.1017/etds.2018.2},
}

@incollection {LMS_ClassfiPCFiniteNewton,
    AUTHOR = {Lodge, Russell and Mikulich, Yauhen and Schleicher, Dierk},
     TITLE = {A classification of postcritically finite {N}ewton maps},
 BOOKTITLE = {In the tradition of {T}hurston {II}. {G}eometry and groups},
     PAGES = {421--448},
 PUBLISHER = {Springer, Cham},
      YEAR = {[2022] \copyright 2022},
      ISBN = {978-3-030-97559-3; 978-3-030-97560-9},
   MRCLASS = {37F10 (05C90 37F12 37F20 37F25)},
  MRNUMBER = {4472058},
MRREVIEWER = {Alfredo\ Poirier},
       DOI = {10.1007/978-3-030-97560-9\_13},
       URL = {https://doi-org.proxy.library.stonybrook.edu/10.1007/978-3-030-97560-9_13},
}

@incollection {spider,
	AUTHOR = {Hubbard, John H. and Schleicher, Dierk},
	TITLE = {The spider algorithm},
	BOOKTITLE = {Complex dynamical systems ({C}incinnati, {OH}, 1994)},
	SERIES = {Proc. Sympos. Appl. Math.},
	VOLUME = {49},
	PAGES = {155--180},
	PUBLISHER = {Amer. Math. Soc., Providence, RI},
	YEAR = {1994},
	MRCLASS = {58F23 (30D05)},
	MRNUMBER = {1315537},
	DOI = {10.1090/psapm/049/1315537},
	URL = {https://doi.org/10.1090/psapm/049/1315537},
}

@article {LLM_KissCritFixAnti,
	AUTHOR = {Lodge, Russell and Luo, Yusheng and Mukherjee, Sabyasachi},
	TITLE = {Circle packings, kissing reflection groups and critically
	fixed anti-rational maps},
	JOURNAL = {Forum Math. Sigma},
	FJOURNAL = {Forum of Mathematics. Sigma},
	VOLUME = {10},
	YEAR = {2022},
	PAGES = {Paper No. e3, 38},
	MRCLASS = {37F31 (20F55 30F40 37F20 37F40 51F15 52C26)},
	MRNUMBER = {4365080},
	DOI = {10.1017/fms.2021.81},
	URL = {https://doi.org/10.1017/fms.2021.81},
}

@misc{Geyer_CritFixAnti,
	title={Classification of critically fixed anti-rational maps},
	author={Lukas Geyer},
	year={2020},
	eprint={2006.10788},
	note={arXiv:2006.10788},
	archivePrefix={arXiv},
	primaryClass={math.DS}
}

@article {pilgrim_Classification_criticallyfixed,
	AUTHOR = {Cordwell, Kristin and Gilbertson, Selina and Nuechterlein,
	Nicholas and Pilgrim, Kevin M. and Pinella, Samantha},
	TITLE = {On the classification of critically fixed rational maps},
	JOURNAL = {Conform. Geom. Dyn.},
	FJOURNAL = {Conformal Geometry and Dynamics. An Electronic Journal of the
	American Mathematical Society},
	VOLUME = {19},
	YEAR = {2015},
	PAGES = {51--94},
	MRCLASS = {37F20 (05C10 57M12 57M15)},
	MRNUMBER = {3323420},
	MRREVIEWER = {Daniel Meyer},
	DOI = {10.1090/S1088-4173-2015-00275-1},
	URL = {https://doi-org.proxyiub.uits.iu.edu/10.1090/S1088-4173-2015-00275-1},
}

@misc
{Hlushchanka_criticallyfixed,
	title={Tischler graphs of critically fixed rational maps and their applications},
	author={Mikhail Hlushchanka},
	year={2019},
	note={arXiv:1904.04759},
	eprint={1904.04759},
	archivePrefix={arXiv},
	primaryClass={math.DS}
}

@article {ConstructingRatFromSubdiv,
	AUTHOR = {Cannon, J. W. and Floyd, W. J. and Kenyon, R. and Parry, W.
	R.},
	TITLE = {Constructing rational maps from subdivision rules},
	JOURNAL = {Conform. Geom. Dyn.},
	FJOURNAL = {Conformal Geometry and Dynamics. An Electronic Journal of the
	American Mathematical Society},
	VOLUME = {7},
	YEAR = {2003},
	PAGES = {76--102},
	MRCLASS = {37F20 (20F67 37F10 52C26 57M12)},
	MRNUMBER = {1992038},
	MRREVIEWER = {Kevin M. Pilgrim},
	DOI = {10.1090/S1088-4173-03-00082-1},
	URL = {https://doi.org/10.1090/S1088-4173-03-00082-1},
}

@article {ConstructingSubdivFromRat,
	AUTHOR = {Cannon, J. W. and Floyd, W. J. and Parry, W. R.},
	TITLE = {Constructing subdivision rules from rational maps},
	JOURNAL = {Conform. Geom. Dyn.},
	FJOURNAL = {Conformal Geometry and Dynamics. An Electronic Journal of the
	American Mathematical Society},
	VOLUME = {11},
	YEAR = {2007},
	PAGES = {128--136},
	MRCLASS = {37F20 (20F67 37F10 52C20 52C26 57M12)},
	MRNUMBER = {2329140},
	MRREVIEWER = {Kevin M. Pilgrim},
	DOI = {10.1090/S1088-4173-07-00167-1},
	URL = {https://doi.org/10.1090/S1088-4173-07-00167-1},
}

@article {finite_subdivision_rule_exp2,
	AUTHOR = {Floyd, William and Parry, Walter and Pilgrim, Kevin M.},
	TITLE = {Expansion properties for finite subdivision rules {II}},
	JOURNAL = {Conform. Geom. Dyn.},
	FJOURNAL = {Conformal Geometry and Dynamics. An Electronic Journal of the
	American Mathematical Society},
	VOLUME = {24},
	YEAR = {2020},
	PAGES = {29--50},
	MRCLASS = {37F10 (52C20 57M12)},
	MRNUMBER = {4051834},
	MRREVIEWER = {Peter Ha\"{\i}ssinsky},
	DOI = {10.1090/ecgd/347},
	URL = {https://doi.org/10.1090/ecgd/347},
}

@article {finite_subdivision_rule_exp1,
	AUTHOR = {Floyd, William J. and Parry, Walter R. and Pilgrim, Kevin M.},
	TITLE = {Expansion properties for finite subdivision rules {I}},
	JOURNAL = {Sci. China Math.},
	FJOURNAL = {Science China. Mathematics},
	VOLUME = {61},
	YEAR = {2018},
	NUMBER = {12},
	PAGES = {2237--2266},
	ISSN = {1674-7283},
	MRCLASS = {37F10 (52C20 57M12)},
	MRNUMBER = {3881956},
	DOI = {10.1007/s11425-016-9265-y},
	URL = {https://doi.org/10.1007/s11425-016-9265-y},
}

@article {finite_suvdivision_rule_first,
	AUTHOR = {Cannon, J. W. and Floyd, W. J. and Parry, W. R.},
	TITLE = {Finite subdivision rules},
	JOURNAL = {Conform. Geom. Dyn.},
	FJOURNAL = {Conformal Geometry and Dynamics. An Electronic Journal of the
	American Mathematical Society},
	VOLUME = {5},
	YEAR = {2001},
	PAGES = {153--196},
	ISSN = {1088-4173},
	MRCLASS = {52C26 (52C20)},
	MRNUMBER = {1875951},
	MRREVIEWER = {Richard Kenyon},
	DOI = {10.1090/S1088-4173-01-00055-8},
	URL = {https://doi.org/10.1090/S1088-4173-01-00055-8},
}

@article {Hais_PlanarBdry,
	AUTHOR = {Ha{\"{\i}}ssinsky, Peter},
	TITLE = {Hyperbolic groups with planar boundaries},
	JOURNAL = {Invent. Math.},
	FJOURNAL = {Inventiones Mathematicae},
	VOLUME = {201},
	YEAR = {2015},
	NUMBER = {1},
	PAGES = {239--307},
	ISSN = {0020-9910},
	MRCLASS = {20F67 (20E08 37F30 57M60)},
	MRNUMBER = {3359053},
	MRREVIEWER = {Vassilis Metaftsis},
	DOI = {10.1007/s00222-014-0552-x},
	URL = {https://doi-org.proxyiub.uits.iu.edu/10.1007/s00222-014-0552-x},
}

@article {Poi_HubTree,
	AUTHOR = {Poirier, Alfredo},
	TITLE = {Hubbard trees},
	JOURNAL = {Fund. Math.},
	FJOURNAL = {Fundamenta Mathematicae},
	VOLUME = {208},
	YEAR = {2010},
	NUMBER = {3},
	PAGES = {193--248},
	ISSN = {0016-2736},
	MRCLASS = {37F10 (37E25)},
	MRNUMBER = {2650982},
	MRREVIEWER = {David A. Brown},
	DOI = {10.4064/fm208-3-1},
	URL = {https://doi-org.proxyiub.uits.iu.edu/10.4064/fm208-3-1},
}

@book {Pil_Thesis,
	AUTHOR = {Pilgrim, Kevin Michael},
	TITLE = {Cylinders for iterated rational maps},
	NOTE = {Thesis (Ph.D.)--University of California, Berkeley},
	PUBLISHER = {ProQuest LLC, Ann Arbor, MI},
	YEAR = {1994},
	PAGES = {202},
	MRCLASS = {Thesis},
	MRNUMBER = {2691488},
	URL =
	{http://gateway.proquest.com.proxyiub.uits.iu.edu/openurl?url_ver=Z39.88-2004&rft_val_fmt=info:ofi/fmt:kev:mtx:dissertation&res_dat=xri:pqdiss&rft_dat=xri:pqdiss:9504954},
}

@article {Thu_RubberBand,
    AUTHOR = {Thurston, Dylan P.},
     TITLE = {From rubber bands to rational maps: a research report},
   JOURNAL = {Res. Math. Sci.},
  FJOURNAL = {Research in the Mathematical Sciences},
    VOLUME = {3},
      YEAR = {2016},
     PAGES = {Paper No. 15, 49},
      ISSN = {2522-0144},
   MRCLASS = {37F20 (30C62 30F10 30F60 57M12 57M50)},
  MRNUMBER = {3500499},
MRREVIEWER = {Peter Ha\"{\i}ssinsky},
       DOI = {10.1186/s40687-015-0039-4},
       URL = {https://doi.org/10.1186/s40687-015-0039-4},
}

@article {Carrasco_ConfDimCombiMod,
	AUTHOR = {Carrasco, Matias},
	TITLE = {Conformal dimension and combinatorial modulus of compact
	metric spaces},
	JOURNAL = {C. R. Math. Acad. Sci. Paris},
	FJOURNAL = {Comptes Rendus Math\'{e}matique. Acad\'{e}mie des Sciences. Paris},
	VOLUME = {350},
	YEAR = {2012},
	NUMBER = {3-4},
	PAGES = {141--145},
	ISSN = {1631-073X},
	MRCLASS = {30L10 (20F67 28A78 37F35)},
	MRNUMBER = {2891099},
	MRREVIEWER = {Kevin M. Pilgrim},
	DOI = {10.1016/j.crma.2012.01.015},
	URL = {https://doi.org/10.1016/j.crma.2012.01.015},
}

@book {Pilgrim_Combination,
	AUTHOR = {Pilgrim, Kevin M.},
	TITLE = {Combinations of complex dynamical systems},
	SERIES = {Lecture Notes in Mathematics},
	VOLUME = {1827},
	PUBLISHER = {Springer-Verlag, Berlin},
	YEAR = {2003},
	PAGES = {x+118},
	ISBN = {3-540-20173-4},
	MRCLASS = {37F10 (30F40 57N10)},
	MRNUMBER = {2020454},
	MRREVIEWER = {Adam Lawrence Epstein},
	DOI = {10.1007/b14147},
	URL = {https://doi.org/10.1007/b14147},
}

@book {Park_thesis,
	AUTHOR = {Park, Insung},
	TITLE = {Julia {S}ets with {A}hlfors {R}egular {C}onformal {D}imension
	{O}ne},
	NOTE = {Thesis (Ph.D.)--Indiana University},
	PUBLISHER = {ProQuest LLC, Ann Arbor, MI},
	YEAR = {2021},
	PAGES = {144},
	ISBN = {979-8538-12992-8},
	MRCLASS = {Thesis},
	MRNUMBER = {4326759},
	URL =
	{http://gateway.proquest.com/openurl?url_ver=Z39.88-2004&rft_val_fmt=info:ofi/fmt:kev:mtx:dissertation&res_dat=xri:pqm&rft_dat=xri:pqdiss:28647817},
}

@misc{PilThu_ARConfDim,
	title={Ahlfors-regular conformal dimension and energies of graph maps}, 
	author={Kevin M. Pilgrim and Dylan P. Thurston},
	year={2021},
	eprint={2112.09041},
	note={arXiv:2112.09041},
	archivePrefix={arXiv},
	primaryClass={math.DS}
}

@article {Szpilrajn_PartialOrderExt,
    AUTHOR = {Szpilrajn, Edward},
     TITLE = {Sur l'extension de l'ordre partiel},
   JOURNAL = {Fund. Math.},
  FJOURNAL = {Fundamenta Mathematicae},
    VOLUME = {16},
      YEAR = {1930},
    NUMBER = {1},
     PAGES = {386--389},
}

@article {DH_ThuChac,
    AUTHOR = {Douady, Adrien and Hubbard, John H.},
     TITLE = {A proof of {T}hurston's topological characterization of
              rational functions},
   JOURNAL = {Acta Math.},
  FJOURNAL = {Acta Mathematica},
    VOLUME = {171},
      YEAR = {1993},
    NUMBER = {2},
     PAGES = {263--297},
      ISSN = {0001-5962},
   MRCLASS = {58F23},
  MRNUMBER = {1251582},
MRREVIEWER = {Ben Bielefeld},
       DOI = {10.1007/BF02392534},
       URL = {https://doi.org/10.1007/BF02392534},
}

@article {GHMZ_JCorveCarpetJulia,
    AUTHOR = {Gao, Yan and Ha\"{\i}ssinsky, Peter and Meyer, Daniel and Zeng,
              Jinsong},
     TITLE = {Invariant {J}ordan curves of {S}ierpi\'{n}ski carpet rational
              maps},
   JOURNAL = {Ergodic Theory Dynam. Systems},
  FJOURNAL = {Ergodic Theory and Dynamical Systems},
    VOLUME = {38},
      YEAR = {2018},
    NUMBER = {2},
     PAGES = {583--600},
      ISSN = {0143-3857},
   MRCLASS = {37F10},
  MRNUMBER = {3774834},
MRREVIEWER = {Guoping Zhan},
       DOI = {10.1017/etds.2016.47},
       URL = {https://doi.org/10.1017/etds.2016.47},
}

@misc{DHS_Decomp,
  
  note = {arXiv:2209.02800},
  
  author = {Dudko, Dzmitry and Hlushchanka, Mikhail and Schleicher, Dierk},
  
  title = {A canonical decomposition of postcritically finite rational maps and their maximal expanding quotients},
  year = {2022},
}

@article {CarrascoMackay_ConfdimGp,
	AUTHOR = {Carrasco, Matias and Mackay, John M.},
	TITLE = {Conformal dimension of hyperbolic groups that split over
	elementary subgroups},
	JOURNAL = {Invent. Math.},
	FJOURNAL = {Inventiones Mathematicae},
	VOLUME = {227},
	YEAR = {2022},
	NUMBER = {2},
	PAGES = {795--854},
	ISSN = {0020-9910},
	MRCLASS = {20F67 (30L10 51F99)},
	MRNUMBER = {4372224},
	DOI = {10.1007/s00222-021-01074-w},
	URL = {https://doi-org.proxyiub.uits.iu.edu/10.1007/s00222-021-01074-w},
}

@article {BonkKleiner_ConfdimHypGp,
    AUTHOR = {Bonk, Mario and Kleiner, Bruce},
     TITLE = {Conformal dimension and {G}romov hyperbolic groups with
              2-sphere boundary},
   JOURNAL = {Geom. Topol.},
  FJOURNAL = {Geometry and Topology},
    VOLUME = {9},
      YEAR = {2005},
     PAGES = {219--246},
      ISSN = {1465-3060},
   MRCLASS = {20F67 (30C65)},
  MRNUMBER = {2116315},
MRREVIEWER = {John Meier},
       DOI = {10.2140/gt.2005.9.219},
       URL = {https://doi.org/10.2140/gt.2005.9.219},
}

@article {QYZ_QSGeomCarpet,
    AUTHOR = {Qiu, Weiyuan and Yang, Fei and Zeng, Jinsong},
     TITLE = {Quasisymmetric geometry of {S}ierpi\'{n}ski carpet {J}ulia sets},
   JOURNAL = {Fund. Math.},
  FJOURNAL = {Fundamenta Mathematicae},
    VOLUME = {244},
      YEAR = {2019},
    NUMBER = {1},
     PAGES = {73--107},
      ISSN = {0016-2736},
   MRCLASS = {37F45 (30D05 30F45 37F10)},
  MRNUMBER = {3874666},
MRREVIEWER = {Shengyou Wen},
       DOI = {10.4064/fm494-12-2017},
       URL = {https://doi.org/10.4064/fm494-12-2017},
}

@article {BelkForrest_ThompBasilica,
    AUTHOR = {Belk, James and Forrest, Bradley},
     TITLE = {A {T}hompson group for the basilica},
   JOURNAL = {Groups Geom. Dyn.},
  FJOURNAL = {Groups, Geometry, and Dynamics},
    VOLUME = {9},
      YEAR = {2015},
    NUMBER = {4},
     PAGES = {975--1000},
      ISSN = {1661-7207},
   MRCLASS = {20F65 (20F28 37F10)},
  MRNUMBER = {3428407},
MRREVIEWER = {Hyungryul Baik},
       DOI = {10.4171/GGD/333},
       URL = {https://doi.org/10.4171/GGD/333},
}

@article {LyuMeren_QSBasilica,
    AUTHOR = {Lyubich, Mikhail and Merenkov, Sergei},
     TITLE = {Quasisymmetries of the basilica and the {T}hompson group},
   JOURNAL = {Geom. Funct. Anal.},
  FJOURNAL = {Geometric and Functional Analysis},
    VOLUME = {28},
      YEAR = {2018},
    NUMBER = {3},
     PAGES = {727--754},
      ISSN = {1016-443X},
   MRCLASS = {30D05 (20F65 37F10 37F30)},
  MRNUMBER = {3816522},
MRREVIEWER = {David Jonathon Sixsmith},
       DOI = {10.1007/s00039-018-0452-0},
       URL = {https://doi.org/10.1007/s00039-018-0452-0},
}

@article {Bonk_QSUnifCarpet,
    AUTHOR = {Bonk, Mario},
     TITLE = {Uniformization of {S}ierpi\'{n}ski carpets in the plane},
   JOURNAL = {Invent. Math.},
  FJOURNAL = {Inventiones Mathematicae},
    VOLUME = {186},
      YEAR = {2011},
    NUMBER = {3},
     PAGES = {559--665},
      ISSN = {0020-9910},
   MRCLASS = {37F30 (30C62 30F10 30F45)},
  MRNUMBER = {2854086},
MRREVIEWER = {Peter Ha\"{\i}ssinsky},
       DOI = {10.1007/s00222-011-0325-8},
       URL = {https://doi.org/10.1007/s00222-011-0325-8},
}

@article {BonkLyuMeren_QSofCarpets,
    AUTHOR = {Bonk, Mario and Lyubich, Mikhail and Merenkov, Sergei},
     TITLE = {Quasisymmetries of {S}ierpi\'{n}ski carpet {J}ulia sets},
   JOURNAL = {Adv. Math.},
  FJOURNAL = {Advances in Mathematics},
    VOLUME = {301},
      YEAR = {2016},
     PAGES = {383--422},
      ISSN = {0001-8708},
   MRCLASS = {37F10 (30D05)},
  MRNUMBER = {3539379},
MRREVIEWER = {Alastair N. Fletcher},
       DOI = {10.1016/j.aim.2016.06.007},
       URL = {https://doi.org/10.1016/j.aim.2016.06.007},
}

@article {QY_QSUnifCantorCircles,
	AUTHOR = {Qiu, Weiyuan and Yang, Fei},
	TITLE = {Quasisymmetric uniformization and {H}ausdorff dimensions of
	{C}antor circle {J}ulia sets},
	JOURNAL = {Trans. Amer. Math. Soc.},
	FJOURNAL = {Transactions of the American Mathematical Society},
	VOLUME = {374},
	YEAR = {2021},
	NUMBER = {7},
	PAGES = {5191--5223},
	ISSN = {0002-9947},
	MRCLASS = {37F10 (28A80 30C65 37F20 37F35)},
	MRNUMBER = {4273190},
	MRREVIEWER = {Shengyou Wen},
	DOI = {10.1090/tran/8392},
	URL = {https://doi-org.proxyiub.uits.iu.edu/10.1090/tran/8392},
}

@article {HP_CXC,
    AUTHOR = {Ha{\"{\i}}ssinsky, Peter and Pilgrim, Kevin M.},
     TITLE = {Coarse expanding conformal dynamics},
   JOURNAL = {Ast\'{e}risque},
  FJOURNAL = {Ast\'{e}risque},
    NUMBER = {325},
      YEAR = {2009},
     PAGES = {viii+139 pp. (2010)},
      ISSN = {0303-1179},
      ISBN = {978-2-85629-266-2},
   MRCLASS = {37F10 (30D05 37F30)},
  MRNUMBER = {2662902},
MRREVIEWER = {Volker Mayer},
}

@article {HP_QIneqCarpets,
    AUTHOR = {Ha{\"{\i}}ssinsky, Peter and Pilgrim, Kevin M.},
     TITLE = {Quasisymmetrically inequivalent hyperbolic {J}ulia sets},
   JOURNAL = {Rev. Mat. Iberoam.},
  FJOURNAL = {Revista Matem\'{a}tica Iberoamericana},
    VOLUME = {28},
      YEAR = {2012},
    NUMBER = {4},
     PAGES = {1025--1034},
      ISSN = {0213-2230},
   MRCLASS = {37F50 (28A80 30D05 37F10)},
  MRNUMBER = {2990132},
MRREVIEWER = {David A. Brown},
       DOI = {10.4171/RMI/701},
       URL = {https://doi.org/10.4171/RMI/701},
}

@article {HP_MinConfDimCXC,
    AUTHOR = {Ha{\"{\i}}ssinsky, Peter and Pilgrim, Kevin M.},
     TITLE = {Minimal {A}hlfors regular conformal dimension of coarse
              expanding conformal dynamics on the sphere},
   JOURNAL = {Duke Math. J.},
  FJOURNAL = {Duke Mathematical Journal},
    VOLUME = {163},
      YEAR = {2014},
    NUMBER = {13},
     PAGES = {2517--2559},
      ISSN = {0012-7094},
   MRCLASS = {37F30 (20F67 30C65 54E40)},
  MRNUMBER = {3265557},
       DOI = {10.1215/00127094-2819408},
       URL = {https://doi.org/10.1215/00127094-2819408},
}

@article {MMS_JStable,
    AUTHOR = {Ma\~{n}\'{e}, R. and Sad, P. and Sullivan, D.},
     TITLE = {On the dynamics of rational maps},
   JOURNAL = {Ann. Sci. \'{E}cole Norm. Sup. (4)},
  FJOURNAL = {Annales Scientifiques de l'\'{E}cole Normale Sup\'{e}rieure. Quatri\`eme
              S\'{e}rie},
    VOLUME = {16},
      YEAR = {1983},
    NUMBER = {2},
     PAGES = {193--217},
      ISSN = {0012-9593},
   MRCLASS = {58F10 (30D05 58F20)},
  MRNUMBER = {732343},
MRREVIEWER = {I. N. Baker},
       URL = {http://www.numdam.org/item?id=ASENS_1983_4_16_2_193_0},
}

@incollection {Mil_HypComp,
    AUTHOR = {Milnor, John},
     TITLE = {Hyperbolic components},
 BOOKTITLE = {Conformal dynamics and hyperbolic geometry},
    SERIES = {Contemp. Math.},
    VOLUME = {573},
     PAGES = {183--232},
      NOTE = {With an appendix by A. Poirier},
 PUBLISHER = {Amer. Math. Soc., Providence, RI},
      YEAR = {2012},
   MRCLASS = {37Fxx},
  MRNUMBER = {2964079},
       DOI = {10.1090/conm/573/11428},
       URL = {https://doi.org/10.1090/conm/573/11428},
}

@incollection {McM_AutoRatlMap,
    AUTHOR = {McMullen, Curt},
     TITLE = {Automorphisms of rational maps},
 BOOKTITLE = {Holomorphic functions and moduli, {V}ol. {I} ({B}erkeley,
              {CA}, 1986)},
    SERIES = {Math. Sci. Res. Inst. Publ.},
    VOLUME = {10},
     PAGES = {31--60},
 PUBLISHER = {Springer, New York},
      YEAR = {1988},
   MRCLASS = {58F99 (30D05 32G15)},
  MRNUMBER = {955807},
MRREVIEWER = {M. Rees},
       DOI = {10.1007/978-1-4613-9602-4\_3},
       URL = {https://doi.org/10.1007/978-1-4613-9602-4_3},
}

@article {Nek_combmodel,
    AUTHOR = {Nekrashevych, Volodymyr},
     TITLE = {Combinatorial models of expanding dynamical systems},
   JOURNAL = {Ergodic Theory Dynam. Systems},
  FJOURNAL = {Ergodic Theory and Dynamical Systems},
    VOLUME = {34},
      YEAR = {2014},
    NUMBER = {3},
     PAGES = {938--985},
      ISSN = {0143-3857},
   MRCLASS = {37C50 (20M35 37C15 37F15 37F20)},
  MRNUMBER = {3199801},
       DOI = {10.1017/etds.2012.163},
       URL = {https://doi.org/10.1017/etds.2012.163},
}

@article {Mackay_ConfdimCarpet,
    AUTHOR = {Mackay, John M.},
     TITLE = {Spaces and groups with conformal dimension greater than one},
   JOURNAL = {Duke Math. J.},
  FJOURNAL = {Duke Mathematical Journal},
    VOLUME = {153},
      YEAR = {2010},
    NUMBER = {2},
     PAGES = {211--227},
      ISSN = {0012-7094},
   MRCLASS = {30L10 (20F67 30C65 51F99)},
  MRNUMBER = {2667133},
MRREVIEWER = {Kevin Wildrick},
       DOI = {10.1215/00127094-2010-023},
       URL = {https://doi.org/10.1215/00127094-2010-023},
}

@article {HP_ThuObsConfDim,
    AUTHOR = {Ha{\"{\i}}ssinsky, Peter and Pilgrim, Kevin M.},
     TITLE = {Thurston obstructions and {A}hlfors regular conformal
              dimension},
   JOURNAL = {J. Math. Pures Appl. (9)},
  FJOURNAL = {Journal de Math\'{e}matiques Pures et Appliqu\'{e}es. Neuvi\`eme S\'{e}rie},
    VOLUME = {90},
      YEAR = {2008},
    NUMBER = {3},
     PAGES = {229--241},
      ISSN = {0021-7824},
   MRCLASS = {37F30 (30C65 30F40 37D20 37F15 53C23 54E40)},
  MRNUMBER = {2446078},
MRREVIEWER = {Ma\l gorzata Stawiska},
       DOI = {10.1016/j.matpur.2008.04.006},
       URL = {https://doi.org/10.1016/j.matpur.2008.04.006},
}

@article {Park_Levy,
    AUTHOR = {Park, Insung},
     TITLE = {Levy and {T}hurston obstructions of finite subdivision rules},
   JOURNAL = {Ergodic Theory Dynam. Systems},
  FJOURNAL = {Ergodic Theory and Dynamical Systems},
    VOLUME = {44},
      YEAR = {2024},
    NUMBER = {9},
     PAGES = {2649--2699},
      ISSN = {0143-3857,1469-4417},
   MRCLASS = {37F10 (20F06 37F15 37F20)},
  MRNUMBER = {4802766},
MRREVIEWER = {Alfredo\ Poirier},
       DOI = {10.1017/etds.2023.115},
       URL = {https://doi-org.proxy.library.stonybrook.edu/10.1017/etds.2023.115},
}

@article {BartholdiDudko_expanding,
    AUTHOR = {Bartholdi, Laurent and Dudko, Dzmitry},
     TITLE = {Algorithmic aspects of branched coverings {IV}/{V}.
              {E}xpanding maps},
   JOURNAL = {Trans. Amer. Math. Soc.},
  FJOURNAL = {Transactions of the American Mathematical Society},
    VOLUME = {370},
      YEAR = {2018},
    NUMBER = {11},
     PAGES = {7679--7714},
      ISSN = {0002-9947},
   MRCLASS = {37D20 (20F34 37F15 37F20 57M50)},
  MRNUMBER = {3852445},
MRREVIEWER = {Kevin M. Pilgrim},
       DOI = {10.1090/tran/7199},
       URL = {https://doi.org/10.1090/tran/7199},
}

@incollection {ShiShi_Rees,
    AUTHOR = {Shishikura, Mitsuhiro},
     TITLE = {On a theorem of {M}.\@ {R}ees for matings of polynomials},
 BOOKTITLE = {The {M}andelbrot set, theme and variations},
    SERIES = {London Math. Soc. Lecture Note Ser.},
    VOLUME = {274},
     PAGES = {289--305},
 PUBLISHER = {Cambridge Univ. Press, Cambridge},
      YEAR = {2000},
   MRCLASS = {37F10 (30D05)},
  MRNUMBER = {1765095},
MRREVIEWER = {Feliks Przytycki},
}

@article {Rees_semiconj,
    AUTHOR = {Rees, Mary},
     TITLE = {A partial description of parameter space of rational maps of
              degree two. {I}},
   JOURNAL = {Acta Math.},
  FJOURNAL = {Acta Mathematica},
    VOLUME = {168},
      YEAR = {1992},
    NUMBER = {1-2},
     PAGES = {11--87},
      ISSN = {0001-5962},
   MRCLASS = {58F23 (30D05)},
  MRNUMBER = {1149864},
MRREVIEWER = {Feliks Przytycki},
       DOI = {10.1007/BF02392976},
       URL = {https://doi.org/10.1007/BF02392976},
}

@article {ThuPosChac,
    AUTHOR = {Thurston, Dylan P.},
     TITLE = {A positive characterization of rational maps},
   JOURNAL = {Ann. of Math. (2)},
  FJOURNAL = {Annals of Mathematics. Second Series},
    VOLUME = {192},
      YEAR = {2020},
    NUMBER = {1},
     PAGES = {1--46},
      ISSN = {0003-486X},
   MRCLASS = {37F10 (37E25 37F12 37F31)},
  MRNUMBER = {4125449},
       DOI = {10.4007/annals.2020.192.1.1},
       URL = {https://doi.org/10.4007/annals.2020.192.1.1},
}

@article {ThuElaGraph,
    AUTHOR = {Thurston, Dylan P.},
     TITLE = {Elastic graphs},
   JOURNAL = {Forum Math. Sigma},
  FJOURNAL = {Forum of Mathematics. Sigma},
    VOLUME = {7},
      YEAR = {2019},
     PAGES = {Paper No. e24, 84},
   MRCLASS = {37E25 (05C21 31C20 58E20)},
  MRNUMBER = {3993808},
MRREVIEWER = {S. Minirani},
       DOI = {10.1017/fms.2019.4},
       URL = {https://doi.org/10.1017/fms.2019.4},
}
